\documentclass[reqno,12pt]{amsart}
\usepackage[a4paper, margin=20mm]{geometry}
\usepackage{amsmath,stmaryrd, mathtools}
\usepackage{amssymb}
\usepackage{amsthm}
\usepackage[utf8]{inputenc}
\usepackage{graphicx}
\usepackage{xcolor}
\usepackage[normalem]{ulem}
\usepackage[english]{babel}
\usepackage{Ludo}

\usepackage[misc]{ifsym}

\allowdisplaybreaks

\usepackage{csquotes}

\numberwithin{equation}{section}

\title[]{Linear independence of values of Dirichlet $L$ functions}
\author[]{Ludovic Mistiaen}
\thanks{Ludovic Mistiaen (Corresponding author) - Institut Fourier, Université Grenoble Alpes, CNRS, CS
40700, 38058 Grenoble cedex 9, France. \Letter \;ludovic.mistiaen@orange.fr}
\date{}



\begin{document}
\maketitle


\begin{abstract}
For any Dirichlet character $\chi$ with conductor $N$, we give a lower bound of order $\sqrt{s/\log(s)}$ on the dimension of the $\Q(e^{2i\pi/\varphi(N)})$-vector space spanned by the values of its $L$-function at integers $2\leq i \leq s$ of the parity opposite to that of $\chi$, where $\varphi$ is Euler's totient function. This generalizes a 2026 result of Fischler, which corresponds to principal characters. To this end, we construct linear combinations of these values of $L$-functions using a refined version of Siegel's lemma, and apply to them a linear independence criterion that generalizes Fischler's. To verify a crucial hypothesis of this criterion, we rely on a ``Shidlovskii's lemma".
\end{abstract}



\section{Introduction}

In 2001, Ball and Rivoal \cite{rivoal2000, ballrivoal2001} showed that infinitely many odd zeta values are irrational, by giving a lower bound on the dimension of the vector space they span over $\Q$.
\begin{theorem}
\label{thmballrivoal}
    (Ball-Rivoal, 2001).
    For odd integers $s\geq 3$, we have as $s$ goes to infinity
    \begin{equation*}
        \mathrm{dim}_\Q \mathrm{Span}_\Q \Big\{1, \zeta(3), \zeta(5), ..., \zeta(s)\Big\} \geq  \hspace{1mm}\frac{1+o(1)}{1+\log(2)}\log(s).
    \end{equation*}
\end{theorem}

To obtain this, they applied Nesterenko's linear independence criterion (see \cite{nesterenko85} or \cite[Théo\-rème II.1.3]{colmez2002}) to a sequence of explicit $\Q$-linear combinations of $1, \zeta(3), \zeta(5), ..., \zeta(s)$. These combinations were constructed by summing the values at integers of a sequence of rational functions. This process, detailed in \sref{Section}{partie2}, appears in Nikishin's work \cite{nikishin79} and is related to Padé approximations (see \cite[Section 3]{beukers81} for the case of $\zeta(3)$ and \cite{fischlerrivoal2003} for an overview).

Ball and Rivoal used a sequence of rational functions whose numerators have a symmetry that only allows odd zeta values to appear in the combinations. Their construction was generalized by Zudilin in \cite{zudilin2002}. Lai \cite{lai2026.1} recently improved the constant, from $\frac{1}{1+\log(2)} \approx 0.59$ to approximately $0.66$.

\smallskip
In 2019, Fischler, Sprang and Zudilin \cite{fischlersprangzudilin2019} used an elimination procedure that enabled them to construct, from rational functions, explicit $\Q$-linear combinations of $1, \zeta(3), \zeta(5), ..., \zeta(s)$ in which up to $2^{\frac{\log(s)}{\log\log(s)}(1+o(1))}$ of the numbers $\zeta(2i+1)$ do not appear. They thus obtained an asymptotic lower bound of $2^{\frac{\log(s)}{\log\log(s)}(1+o(1))}$ for the number of irrational numbers in the set $\big\{ \zeta(3), \zeta(5), ..., \zeta(s)\big\}$. This lower bound is asymptotically better than $\log(s)$, but the irrationality of a set of numbers is weaker than their linear independence over $\Q$. By choosing different rational functions, Lai and Yu \cite{laiyu2020} obtained the following result.

\begin{theorem}
\label{thmlaiyu}
    (Lai-Yu, 2020).
    For odd integers $s\geq 3$, the set $\Big\{\zeta(3), \zeta(5), ..., \zeta(s)\Big\}$ contains at least $\big(1.19+o(1)\big)\sqrt{{s}/{\log(s)}}$ irrational numbers as $s$ goes to infinity.
\end{theorem}
Lai \cite{lai2026.2} recently improved the constant, from $1.19$ to approximately $1.28$.

\smallskip 
In 2011, Nishimoto \cite{nishimoto2011} generalized \sref{Theorem}{thmballrivoal} to the numbers $L(f, i): = \sum_{m\geq 1} \frac{f(m)}{m^i}$, where a periodic function $f: \Z \to \C$ with period $N \geq 1$ is given. The case $f = \mathbbm{1}$ constantly equal to $1$ corresponds to $L(f, i) = \zeta(i)$. He constructed, from rational functions, explicit linear combinations of $1$ and the numbers $L(f, i)$, with the integers $i \in \llbracket 2, s\rrbracket$ of a given parity (see \sref{Remark}{remarquesurlaparité} below). He applied Nesterenko's linear independence criterion to this sequence of linear combinations, by using the saddle-point method to estimate the asymptotic behavior of their absolute value. He obtained the following result.

\begin{theorem}
\label{thmnishimoto2011}
    (Nishimoto, 2011).
    Let $N\geq 1$ be an integer. Let $f: \Z \to \C$ be a non-zero periodic function with period $N$.
    Let $\varepsilon \in \{0, 1\}$. For integers $s\geq 2$ of parity $\varepsilon$, we have as $s$ goes to infinity
    
    \begin{equation*}
        \mathrm{dim}_\Q \mathrm{Span}_\Q \Big\{ \;L(f, i) \;|\; 2\leq i\leq s, \;i\equiv \varepsilon \mod 2 \; \Big\} \geq \frac{1+o(1)}{N+\log(2)}\log(s).
    \end{equation*}
\end{theorem}
In 2018, Fischler \cite{fischler2018} obtained the same result (in a slightly refined version when $f$ is a Dirichlet character) with a different method. He also constructed a sequence of linear combinations from rational functions, but in a way that only allows him to get an asymptotic upper bound on their absolute value: Nesterenko's linear independence criterion cannot be applied. Instead, he used a linear independence criterion based on Siegel's ideas \cite[Proposition 4.6]{fischler2018}, in which the need for an asymptotic lower bound is replaced by the need to construct not one, but many sequences of independent linear combinations. Fischler relied on a ``Shidlovskii's lemma" to verify this independence.

In 2020, Fischler \cite[Théorème 1]{fischler2020} improved this theorem by removing the dependence of the constant on $N$: it is replaced by $\frac{1+o(1)}{1+\log(2)}$. In the same paper \cite[Théorème 2]{fischler2020}, he generalized the result of \cite{fischlersprangzudilin2019} by giving an asymptotic lower bound of $2^{\frac{\log(s)}{\log\log(s)}(1+o(1))}$ on the number of irrational numbers in the set $\big\{L(f, i)\;|\; 2\leq i\leq s, i\equiv\varepsilon \mod 2\big\}$.

Recently, Calegari, Dimitrov and Tang \cite{calegaridimitrovtang2024} proved the linear independence over $\Q$ of the numbers $1, \zeta(2)$ and $L(\chi_{-3}, 2)$, where $\chi_{-3}$ is the non-trivial Dirichlet character mod $3$. This of course implies the irrationality of the number $L(\chi_{-3}, 2)$.

\smallskip
In 2026, Fischler \cite{fischler2026} considered for the first time non-explicit $\Q$-linear combinations of $1, \zeta(3),$ $\zeta(5), ..., \zeta(s)$. This allowed him to strengthen the lower bound of \sref{Theorem}{thmballrivoal}.

\begin{theorem}
\label{thmfischler2026}
    (Fischler, 2026).
    For any sufficiently large odd integer $s$, we have
    \begin{equation*}
        \mathrm{dim}_\Q \mathrm{Span}_\Q \Big\{1, \zeta(3), \zeta(5), ..., \zeta(s)\Big\} \geq 0.21 \sqrt{\frac{s}{\log(s)}}.
    \end{equation*}
\end{theorem}

In this paper, we use his method along with the construction of \cite{fischler2018} to obtain a more general result (\autoref{théorèmeplusgénéral}), from which we deduce the following result. We define the parity of a Dirichlet character $\chi$ as the unique $\varepsilon_\chi \in\{0, 1\}$ such that $\chi(-1) = (-1)^{\varepsilon_\chi}\chi(1)$. We denote Euler's totient function by $\varphi$.

\begin{theorem}
\label{théorèmeprincipal}
    Let $N\geq 1$ be an integer. Let $\chi: \Z \to \C$ be a Dirichlet character of conductor $N$. Let $\varepsilon \in \{ 0, 1 \}$ be the parity opposite to that of $\chi$. Let $\K = \Q(e^{2i\pi/\varphi(N)})$. 

    For any sufficiently large integer $s$, we have

    \begin{equation*}
        \mathrm{dim}_\K \mathrm{Span}_\K \Big\{ \;L(\chi, i) \;|\; 2\leq i\leq s, \;i\equiv \varepsilon \mod 2 \; \Big\} \geq \frac{\mathcal{C}_N}{[\K:\Q]\sqrt{N}}\sqrt{\frac{s}{\log(s)}},
    \end{equation*}
    where $[\K:\Q] = \varphi\big(\varphi(N)\big)$ and 
    \begin{align*}
        &\mathcal{C}_N = 0.21 \;\;\text{if $N=1$ or $3,$} \hspace{2cm} \mathcal{C}_N = 0.21\sqrt{2} \;\;\text{if $N=4$,} \\
        &\mathcal{C}_N = 0.42 \;\;\text{if $N$ odd, $N\geq 5$,} \hspace{1.47cm} \mathcal{C}_N = 0.42\sqrt{2} \;\;\text{if $4|N, N\geq 8$.}
    \end{align*}
\end{theorem}

\begin{remark}
\label{remarquesurlaconstante}
    As noted in \sref{Section}{partie2}, there exists no Dirichlet character of conductor $N \equiv 2 \mod 4$. The definition of $\mathcal{C}_N$ thus properly covers all Dirichlet characters.

    When $N = 1, 3$ or $4$, we have $\K = \Q$, and thus $[\K_\infty : \R] = 1$ in \sref{Proposition}{critèred'indépendancelinéaire}. This is why the constant $\mathcal{C}_N$ is half as large in these cases. When $N = 1$, the characters of conductor $1$ are the principal characters $\chi_0$ mod $\widetilde{N}$, $\widetilde{N} \geq 1$. They are even, so that $\varepsilon = 1$, and as explained in \sref{Section}{partie2} below, we have $L(\chi_0, i) = \prod_{p|\widetilde{N}} (1-p^{-i}) \zeta(i) \in \Q^* \zeta(i)$ for $i\geq 2$. Hence, by picking any principal character $\chi_0$, we recover \autoref{thmfischler2026}.
\end{remark}

\begin{remark}
\label{remarquesurlaparité}
    It is crucial to consider only the terms with $i\equiv\varepsilon \mod 2$ to obtain a non-trivial result. Indeed, it is known that $L(\chi, i) \in \pi^i\overline{\Q}^*$ when $i$ is of the same parity as $\chi$ (see for example \cite[Corollary 2.10, p.~443]{neukirch97}). Since $\pi$ is transcendental, the numbers $L(\chi, i), \;i\equiv \varepsilon_\chi \mod 2$ are all linearly independent over $\K$, which would make the lower bound given by \autoref{théorèmeprincipal} trivial.

\end{remark}

\begin{remark}
    The previous versions of this article covered the case of an odd conductor $N$. Between them and the present version, an AI article achieving the case $\chi$ primitive mod $4$ of \autoref{théorèmeprincipal} with a constant $\widetilde{\mathcal{C}}_N = \sqrt{2}\mathcal{C}_N = 0.42$ was posted on hexagonmath.org by an anonymous user, without human verification. This case corresponds to Dirichlet's beta function $L(\chi, i) = \beta(i)$. 
\end{remark}


\smallskip
\noindent \textbf{Acknowledgment.} The author wishes to express sincere gratitude to Stéphane Fischler and Tanguy Rivoal, for their guidance, support, and precious insights throughout this work.

\section{Structure of the proof}
\label{partie2}

For an $M$-periodic function $f: \Z \to \C$, we introduce the power series
\begin{equation*}
    L(f, i, z): = \sum_{m=1}^{+\infty} \frac{f(m)z^m}{m^i}, \hspace{0.4cm}i\in\Z_{>0}, \hspace{0.4cm}|z|\leq 1, \hspace{0.8cm}z^M \neq 1 \;\text{if $i=1$}.
\end{equation*}
The case $f = \mathbbm{1}$ identically equal to $1$ corresponds to the polylogarithms $\Li_i(z) = \sum_{m=1}^{+\infty} \frac{z^m}{m^i}$, which satisfy $\Li_i(1) = \zeta(i)$ for $i\geq 2$. If a character $\chi$ mod $\widetilde{N}$ is not primitive and has conductor $N$, then there exists a primitive character $\chi'$ mod $N$ that induces $\chi$ (see  \cite[p.~440]{neukirch97}) and we have
\begin{equation*}
    L(\chi, i) = \prod_{{p|\widetilde{N}}\atop{p \nmid N}}\left(1-\frac{\chi'(p)}{p^i}\right)L(\chi', i) \in \K^* L(\chi', i),
\end{equation*}
since the values of $\chi'$ are $0$ or $\varphi(N)$-th roots of unity, hence lie in $\K = \Q(e^{2i\pi/\varphi(N)})$, and $|\chi'(p)p^{-i}| < 1$. Therefore, for any $I \subseteq \Z_{\geq 2}$, the numbers $L(\chi, i), i\in I$ and $L(\chi', i), i\in I$, span the same $\K$-vector space. Without loss of generality, we can thus assume that $\chi$ is a primitive character mod $N$. Notice that by \cite[p.~46, (3.7)]{iwanieckowalski2004}, there exist no primitive characters mod $N$ if $N \equiv 2\mod 4$, which means there exist no characters of conductor $N \equiv 2\mod 4$. 

\begin{notation}
\label{notationdudébut}
    In Sections 2 to 6, we fix an integer $M\geq 1$ and an $M$-periodic function $f : \Z \to \C$. We let $\mu := e^{2i\pi/M}$. This is a primitive $M$-th root of unity.

    We also fix a point $z_0$ which is an $M$-th root of $-1$.

    Lastly, we denote by $\mathcal{U}$ the complex plane minus a half-line starting from $0$ and crossing none of the points $z = \mu^\ell z_0, \;\;0\leq\ell\leq M-1$, nor $1$. By $\log$, we then mean the determination of the logarithm satisfying $\log(1) = 0$ on the simply connected open set $\mathcal{U}$.
\end{notation}

\smallskip
If $F$ is a rational function of degree $\leq -2$, whose poles lie in $\{0, -1, ..., -r\}$ and are of order $\leq a\in\Z_{>0}$, then the series
\begin{equation*}
    S^{[\infty]}(z) = \sum_{t=1}^{+\infty} F(t)z^{-t} \;\;\;\;\;\;\text{and}\;\;\;\;\;\; S^{[0]}(z) = \sum_{t=r+1}^{+\infty} F(-t)z^t
\end{equation*}
are linear combinations with polynomial coefficients of $1, \Li_1(1/z), ..., \Li_a(1/z)$, respectively $1, \Li_1(z),$ $..., \Li_a(z)$ (see for example \cite[Lemme 1]{ballrivoal2001} or \cite[§4.3]{fischler2018}). If the $p$-th derivative $F^{(p)}$ is used instead, polylogarithms of order $p+1$ to $a+p$ appear instead (see \cite{nesterenko96} for the example of $\zeta(3)$, and \cite[§4]{fischlerrivoal2003} for the general case). This mechanism relies on the partial fraction decomposition of such a rational function, as can be seen in the proof of \sref{Lemma}{lesSsontdesclenlespolylogs} below.


In \sref{Section}{partie3}, we fix an integer $a\geq 1$, we introduce rational parameters $r, \omega, \Omega$, and we use Siegel's lemma (\sref{Lemma}{lemmedesiegel}) to construct such a rational function $F_n$ for multiples $n$ of $M$ such that $\frac{rn}{M}, \omega n$ and $\Omega n$ are integers. Our construction is not explicit, but the properties of these rational functions enable us to construct linear combinations to which the linear independence criterion given by \sref{Proposition}{critèred'indépendancelinéaire} will apply.

In \sref{Section}{partie4}, we introduce two new parameters $h\in\Z_{\geq 0}, \kappa\in\Q$ and a parity $\varepsilon$. We denote by $\mathcal{N}$ the infinite set of multiples of $M$ such that $\frac{rn}{M}, \omega n, \Omega n$ and $\kappa n$ are integers. For each $n\in\mathcal{N}$ and $p\in\llbracket 0, h\rrbracket$, we use $F_n^{(p)}$ to construct linear combinations $S_{n, p}^{[\infty]}$ and $S_{n, p}^{[0]}$ with coefficients in $\Q[z]$ of the functions $1$ and $\Li_1(1/z), ..., \Li_{a+h}(1/z)$, respectively $\Li_1(z), ..., \Li_{a+h}(z)$. We differentiate them $k-1$ times with respect to $z$, we combine them in the fashion of \cite[§4.3]{fischler2018} and we evaluate at $z=z_0$ to obtain linear combinations $\Lambda_{n, (p, k)}, \;(p, k)\in\llbracket 0, h\rrbracket\times\llbracket 2rn+2, \kappa n\rrbracket$ of the numbers $-z_0^mf(m), 0\leq m\leq M-1$ and $L(f, i, z_0), \;\;1\leq i\leq a+h, \;\;i\equiv\varepsilon \mod 2$.

In \sref{Section}{partie5}, we prove that the linear combinations $\Lambda_{n, (p, k)}$ have coefficients in $\Z$ and satisfy the assumptions $(i)$ and $(ii)$ of the linear independence criterion (\sref{Proposition}{critèred'indépendancelinéaire}). That is, we obtain geometric bounds on their absolute values and on the absolute values of their coefficients.

In \sref{Section}{partie6}, we rely on a ``Shidlovskii's lemma" (\sref{Theorem}{Shidlovskii}) to prove that the linear combinations $\Lambda_{n, (p, k)}$ satisfy, under an assumption on the parameters $\omega, r, h$ and $\kappa$, the assumption $(iii)$ of the linear independence criterion for the linear forms $\psi_m$ defined by \eqref{définitiondespsim}, $0\leq m\leq M-1$. To this end, we will need to consider a differential system whose finite singularities are $0$ and the $M$-th roots of unity. Since $z_0$ is an $M$-th root of $-1$, it is a regular point of this differential system.

In \sref{Section}{partie7}, we state and prove a linear independence criterion (\sref{Proposition}{critèred'indépendancelinéaire}). We then fix some numerical values for the parameters $r, \omega, \kappa$, some values depending on $a$ for the parameters $\Omega$ and $h$ and use the construction of the preceding sections to prove \autoref{théorèmeplusgénéral}. We then deduce \autoref{théorèmeprincipal} from this theorem.

\section{A sequence of non explicit rational functions}
\label{partie3}

This section is dedicated to the proof of \sref{Proposition}{existencedescij}, which provides, for infinitely many integers $n$, the rational functions $F_n$ used in \sref{Section}{partie4} to construct linear combinations.

\smallskip
In \sref{Subsection}{souspartie3.1}, we state \sref{Proposition}{existencedescij} and translate its condition $(i)$ into a linear system with rational coefficients $\theta_{a, n, k, i, j, \ell}$. 

In \sref{Subsection}{paragraphesurlexpressionexplicitedestheta}, we give an explicit expression of the $\theta_{a, n, k, i, j, \ell}$, allowing us to derive estimates on their absolute values and the size of their denominators. Our approach is new, even for $M=1$.

In \sref{Subsection}{soussectionapplicationdulemmedesiegel}, we apply Siegel's lemma (\sref{Lemma}{lemmedesiegel}) to prove \sref{Proposition}{existencedescij}. 

In \sref{Subsection}{soussesctioncalcultechnique}, we carry out technical calculations extending the results of \sref{Subsection}{paragraphesurlexpressionexplicitedestheta}, that will be needed in \sref{Section}{partie5}.

\subsection{Statement of the proposition and its translation into a linear system}
\label{souspartie3.1}

Let $a\geq M+1$ be a fixed integer. This section aims to construct a rational function $F_n\in\Q(t)$ for infinitely many multiples $n$ of $M$. We only consider rational functions of degree $\leq -2$ and whose poles are among $0, -M, -2M, ..., -n$ and of order at most $a$. The rational function $F_n$ will thus be characterized by the coefficients $c_{n, i, j}$ of its partial fraction decomposition, $(i, j)\in\llbracket 1, a\rrbracket\times\llbracket 0, n/M\rrbracket$. We also introduce the coefficients $\mathfrak{A}_{n,k}$ of its Taylor expansion at $+\infty$. These coefficients all depend on $a$ and on the parameters $\omega, \Omega$ and $r$ defined below, but we omit this in the notation. We have
\begin{equation}
\label{DESdefn}
    F_n(t) = \sum_{i=1}^a\sum_{j=0}^{n/M} \frac{c_{n, i, j}}{(t+Mj)^i} = \sum_{k=1}^{+\infty} \frac{\mathfrak{A}_{n,k}}{t^k}.
\end{equation}

The rest of \sref{Section}{partie3} is dedicated to the proof of the following proposition, generalizing the case $M=1$ that appears in \cite[§3]{fischler2026}.


\begin{proposition}
\label{existencedescij}
    Let $a\geq M+1$ be an integer. Let $\omega, \Omega, r \in\Q$ be parameters satisfying $1\leq\omega\leq\Omega<\frac{a}{M}$ and $r\geq 1$. Then for any $n\in\Z_{>0}$ such that $\omega n, \Omega n, \frac{n}{M}, \frac{rn}{M}$ are integers, there exist coefficients $c_{n, i, j}\in\Z$ not all zero such that
    \begin{small}
    \begin{enumerate}
        \item[$(i)$] $F_n(t) \stackunder{$=$}{$\scriptscriptstyle{t\to+\infty}$} \mathcal{O}(t^{-\omega n})$,
        \item[$(ii)$] $|\mathfrak{A}_{n,k}| \stackunder{$\leq$}{$\scriptscriptstyle{n\to+\infty}$} \begin{cases}  r^{k-\Omega n}n^k k^a \xi^{n+o(n)} &\text{if $\omega n \leq k < \Omega n$,} \\ n^k k^a \xi^{n+o(n)} &\text{if $k \geq \Omega n$,} \end{cases}$ where the sequences $o(n)$ are uniform in $k$,
        \item[$(iii)$] $\forall i\in\llbracket 1, a\rrbracket \;\forall j\in\llbracket 0, n/M \rrbracket, \; |c_{n, i, j}| \stackunder{$\leq$}{$\scriptscriptstyle{n\to+\infty}$} \xi^{n+o(n)}$, where $$\xi: = \exp\Bigg(\frac{\omega\log(2)+2\omega^2 + 1.04\omega^2\log(a+1)+\frac{1}{2}\Omega^2\log(r)}{\frac{a}{M}-\omega}\Bigg).$$
        \end{enumerate}
        \end{small}
\end{proposition}

To prove this proposition, we will apply Siegel's lemma (\sref{Lemma}{lemmedesiegel}). We thus need to translate condition $(i)$ into a linear system with unknowns $c_{n, i, j}$, $\;(i, j)\in\llbracket 1, a\rrbracket\times\llbracket 0, n/M\rrbracket$.
A first idea would be to translate $(i)$ into $\mathfrak{A}_{n, k} = 0, \;1\leq k\leq \omega n-1.$ 

However, by using the Taylor expansion $\frac{1}{(t+Mj)^i} = \sum_{k=0}^{+\infty} \binom{k+i-1}{k} \frac{(-Mj)^k}{t^{k+i}}$, we observe that
\begin{small}
\begin{align*}
    F_n(t) &= \sum_{i=1}^a \sum_{j=0}^{n/M} \sum_{k=0}^{+\infty} \binom{k+i-1}{k} \frac{(-Mj)^kc_{n, i, j}}{t^{k+i}} = \sum_{i=1}^a \sum_{j=0}^{n/M} \sum_{k=i}^{+\infty} \binom{k-1}{k-i} \frac{(-Mj)^{k-i}c_{n, i, j}}{t^k} \\
            &= \sum_{k=1}^{+\infty} \bigg(\sum_{i=1}^{\min(a, k)} \sum_{j=0}^{n/M} \binom{k-1}{k-i} (-Mj)^{k-i}c_{n, i, j}\bigg)\frac{1}{t^k},
\end{align*}
\end{small}
hence the explicit expression
\begin{small}
\begin{equation}
\label{expressionexplicitedesafrak}
    \mathfrak{A}_{n,k} = \sum_{i=1}^{\min(a, k)} \sum_{j=0}^{n/M} \binom{k-1}{k-i} (-Mj)^{k-i}c_{n, i, j}, \hspace{1cm}k\geq 1.
\end{equation}
\end{small}

As noted by Fischler at the end of \cite[Section 3.1]{fischler2026}, the linear equations $\mathfrak{A}_{n,k} = 0$, with unknowns $c_{n, i, j}$, have coefficients that can reach order of magnitude $n^{\omega n}$, such as the coefficients $\binom{k-1}{k-1}(-n)^{k-1}$ for $k$ close to $\omega n$. Applying Siegel's lemma to such equations would yield coefficients $c_{n, i, j}$ with no geometric bound, and $(iii)$ would fail.

\smallskip
Following the approach in \cite[Section 3.1]{fischler2026}, we will translate $(i)$ in another way, namely $P_{n, k, 1}(1) = 0$, $1\leq k\leq \omega n-1$, where the $P_{n, k, i}$ are the rational functions defined by the following recurrence relation for $1\leq i\leq a\;$ and $\;k\geq 1$: 
\begin{equation}
\label{definitiondesPki}
    P_{n, i}(z): = \sum_{j=0}^{n/M} c_{n, i, j}z^{Mj}, \hspace{2cm}
    \begin{cases}
        P_{n, 1, i} = P_{n, i} \\
        P_{n, k+1, i} = P_{n, k, i}' - \frac{1}{z}P_{n, k, i+1},
    \end{cases} 
\end{equation}
with the convention that $P_{n, k, a+1} = 0$. To motivate this definition, let us consider on the open set $\mathcal{U}$ defined in \sref{Notation}{notationdudébut} the functions\footnote{The functions $R_n$ can be seen as remainders in the context of \sref{Subsection}{paragrapheénoncédeShidlovskii}. Although the first two zero coordinates of $Y$ and the first two rows/columns of $A$ seem useless here, their interest will appear in \sref{Subsection}{paragrapheprocessusdedérivation}.}
\begin{equation}
\label{définitiondesRn}
    R_n(z): = \sum_{i=1}^a P_{n, i}(z)\frac{\big(-\log(z)\big)^{i-1}}{(i-1)!}.
\end{equation}

The family $Y: = \Big(0, 0, 1, -\log(z), \frac{\big(-\log(z)\big)^2}{2!}, ..., \frac{\big(-\log(z)\big)^{a-1}}{(a-1)!}\Big)$ is a solution of the differential system $Y' = AY$, where $A\in M_{a+2}(\Q(z))$ is given by 
\begin{equation}
\label{systemediffdespolylogarithmes}
    A = \begin{bmatrix}
        0 &0 &0 &0 &\hdots &0 &0 \\
        0 &0 &0 &0 &\hdots &0 &0 \\
        \frac{1}{z(1-z)} &\frac{-1}{1-z} &0 &0 &\hdots &0 &0 \\
        0 &0 &\frac{-1}{z} &0 &\hdots &0 &0 \\
        0 &0 &0 &\frac{-1}{z} &\hdots &0 &0 \\
        \vdots &\vdots &\vdots &\vdots &\ddots &\vdots &\vdots \\
        0 &0 &0 &0 &\hdots &\frac{-1}{z} &0
    \end{bmatrix},
\end{equation}
so that \eqref{definitiondesPki} precisely defines the rational functions $P_{n, k, i}$ satisfying
\begin{equation}
\label{jenesaispascommentappelerceluilà}
    R_n^{(k-1)}(z) = \sum_{i=1}^a P_{n, k,i}(z)\frac{\big(-\log(z)\big)^{i-1}}{(i-1)!}, \hspace{0.5cm} k\geq 1.
\end{equation}
The advantage of $R_n$ lies in the following formula, which generalizes the case $M=1$ given by \cite[Proposition 2]{fischlerrivoal2003}.
\begin{lemma}
\label{formulebizarresurlesrhon}
    On the open set $\mathcal{U}$, we have
    \begin{equation*}
        R_n(z) = \sum_{k=1}^{+\infty} \mathfrak{A}_{n,k} \frac{\big(-\log(z)\big)^{k-1}}{(k-1)!}.
    \end{equation*}
\end{lemma}

\begin{proof}
    Let $z\in\mathcal{U}$ and set $x = \log(z)$. We have
    \begin{small}
    \begin{align*}
        R_n(e^x) &= \sum_{i=1}^a P_{n, i}(e^x)  \frac{(-x)^{i-1}}{(i-1)!} = \sum_{i=1}^a \sum_{j=0}^{n/M}  c_{n, i, j}e^{Mjx}\frac{(-x)^{i-1}}{(i-1)!} = \sum_{i=1}^a \sum_{j=0}^{n/M} \sum_{s=0}^{+\infty} (-1)^{i-1}c_{n, i, j}\frac{(Mj)^s x^{s+i-1}}{(i-1)!s!} \\
                    &= \sum_{i=1}^a \sum_{j=0}^{n/M} \sum_{k=i}^{+\infty} (-1)^{i-1}c_{n, i, j}\frac{(Mj)^{k-i} x^{k-1}}{(i-1)!(k-i)!} \hspace{1.6cm}(k = s + i) \\
                    &= \sum_{k = 1}^{+\infty} \Bigg(\sum_{i=1}^{\min(a, k)} \sum_{j=0}^{n/M} \binom{k-1}{i-1} (-Mj)^{k-i}c_{n, i, j} \Bigg) (-1)^{k-1} \frac{x^{k-1}}{(k-1)!}
    \end{align*}    
    \end{small}
    Thus using \eqref{expressionexplicitedesafrak}, we get $R_n(e^x) = \sum_{k=1}^{+\infty} \mathfrak{A}_{n,k} \frac{(-x)^{k-1}}{(k-1)!}$, which completes the proof.
\end{proof}
Because of this, we have the following chain of equivalences: 
\begin{small}
\begin{align*}
    F_n(t) \stackunder{$=$}{$\scriptscriptstyle{t\to+\infty}$} \mathcal{O}(t^{-\omega n}) 
    &\iff  \forall k\in\llbracket 1, \omega n-1\rrbracket \hspace{0.6cm} \mathfrak{A}_{n,k} = 0 \hspace{2.2cm}\text{(from \eqref{DESdefn})} \\
    &\iff  R_n(z) \stackunder{$=$}{$\scriptscriptstyle{z\to 1}$} \mathcal{O}\Big((z-1)^{\omega n -1}\Big) \hspace{2.3cm}\text{(\sref{Lemma}{formulebizarresurlesrhon})} \\
    &\iff  \forall k\in\llbracket 1, \omega n-1\rrbracket \hspace{0.6cm} R_n^{(k-1)}(1) = 0 \\
    &\iff  \forall k\in\llbracket 1, \omega n-1\rrbracket \hspace{0.6cm} P_{n, k, 1}(1) = 0 \hspace{1.6cm}\text{(from \eqref{jenesaispascommentappelerceluilà})}.
\end{align*}
\end{small}

We have indeed translated condition $(i)$ of \sref{Proposition}{existencedescij}. In the next subsection, we show that this translation is, as in \cite[Section 3.1]{fischler2026}, a linear system with unknowns $c_{n, i, j}$, whose coefficients $\theta_{a, n, k, i, j, \ell}$ and their denominators grow at most geometrically in $n$.

\subsection{Explicit computation of the coefficients of the linear system}
\label{paragraphesurlexpressionexplicitedestheta}

This subsection aims to prove \sref{Proposition}{propositionexistenceetestimationdestheta}, which reformulates and generalizes the case $M=1$ given by \cite[§3.4]{fischler2026}. To this end, we introduce a combinatorial framework by stating \sref{Lemma}{lemmecombinatoire}. 

\begin{notation}
\label{notationdketdeltaak}
    We denote by $d_k$ the $\mathrm{lcm}$ of all integers from $1$ to $k$, and by $\Delta_{a, k}$ the $\mathrm{lcm}$ of all products of at most $a$ pairwise distinct non-zero integers taken in an interval $I\subseteq \llbracket -k, k\rrbracket$ of length at most $k$.
\end{notation}
\noindent The proof of \cite[Lemma 2, p.~10]{fischler2026} gives $\log(\Delta_{a, k}) = \sum_{j=1}^{a} \log\big( d_{\lfloor k/j \rfloor}\big)$. 

Furthermore, according to \cite[Theorem 12]{rosserschoenfeld62}, we have $\log\big(d_{\lfloor x \rfloor}\big) \leq 1.04x$ for all $x\geq 1$. Thus, 
\begin{equation*}
    \log(\Delta_{a, k}) \leq 1.04k\sum_{j=1}^a \frac{1}{j} \leq 1.04k\big(\gamma + \log(a+1)\big),
\end{equation*}
where $\gamma$ is Euler's constant. This allows us to derive the following upper bounds, for integers $a, k\geq 1$: 
\begin{equation}
\label{estimationexplicitedketdeltaak}
    d_k \leq e^{1.04k}, \hspace{1.5cm} \Delta_{a, k} \leq (a+1)^{1.04k}e^{1.04\gamma k}.
\end{equation}

\smallskip
\begin{proposition}
\label{propositionexistenceetestimationdestheta}
    Let $a \geq M+1$ be an integer. Let $n\in\Z_{>0}$ be a multiple of $M$. There exist rational coefficients $\theta_{a, n, k, i, j, \ell} \in \Q$ such that for any choice of numbers $c_{n, i, j}$, the rational functions $P_{n, k, i}$ defined by the double recurrence
    \begin{small}
    \begin{equation*}
        \begin{cases}
        P_{n, 1, i}(z) = P_{n, i}(z): = \sum_{j=0}^{n/M} c_{n, i, j}z^{Mj}, \\
        P_{n, k+1, i} = P_{n, k, i}' - \frac{1}{z}P_{n, k, i+1},\hspace{0.45cm}k\geq 1,
    \end{cases} \hspace{0.5cm} 1\leq i\leq a, \hspace{1cm} \text{with  } P_{n, k, a+1} = 0, \hspace{0.5cm}k\geq 1,
    \end{equation*}
    \end{small}
    satisfy
    \begin{small}
    \begin{equation}
    \label{definitiondesthetakijij}
        \forall i\in\llbracket 1, a \rrbracket, \;\;\;\forall k\geq 1,  \hspace{1.5cm} z^{k-1}P_{n, k, i}(z) = \sum_{j=0}^{n/M} \Bigg( \sum_{\ell=0}^{a-i} \theta_{a, n, k, i, j, \ell}\; c_{n, i+\ell, j}\Bigg) z^{Mj}.
    \end{equation}
    \end{small}
    These coefficients satisfy \begin{small}
       $$(i)\;|\theta_{a, n, k, i, j, \ell}|\leq k^a2^n(k-1)!, \hspace{1.5cm}(ii)\;\frac{d_k\Delta_{a, \max(k, n)}}{(k-1)!}\theta_{a, n, k, i, j, \ell} \in \mathbb{Z}.$$ \end{small}
    
\end{proposition}

\smallskip
The existence of the coefficients $\theta_{a, n, k, i, j, \ell}$, and in particular the fact that $z^{k-1}P_{n, k, i}$ is a polynomial in $z^M$ of degree at most $\frac{n}{M}$, can be readily obtained from the definition of $P_{n, k, i}$ by induction on $k$. Indeed, the polynomial $P_{n, k, i}$ only depends on the numbers $c_{n, \tilde{\imath}, j}$ with $\tilde{\imath} \geq i$.

We will derive an explicit expression for the coefficients $\theta_{a, n, k, i, j, \ell}$ to estimate their growth, and the growth of their denominators. To this end, we introduce some notation. As in \cite{fischler2026}, for any ${k\in\Z_{>0}}$ and $\ell\in\llbracket 0, k-1\rrbracket$, let $H_{\ell, k}$ be the set of all $(\ell+1)$-tuples of positive integers whose sum equals $k$: 
\vspace{-0.1cm}
\begin{small}
\begin{equation*}
    H_{\ell, k} = \bigg\{ \uh = (h_0, ..., h_\ell) \in (\Z_{>0})^{\ell +1} \;\;\bigg|\;\; \sum_{v=0}^\ell h_v = k \bigg\}.
\end{equation*}
\end{small}
Let $H_{\ell, k} = \varnothing$ whenever $\ell\in\Z$ is not in $\llbracket 0, k-1\rrbracket$. To each $(\ell+1)$-tuple $\uh\in H_{\ell, k}$ we associate a polynomial
\begin{small}
\begin{equation*}
    \phi(\uh, X) = \frac{\prod_{m=1}^{k-1} (X+1-m)}{\prod_{u=0}^{\ell-1} \Big(X+1-\sum_{v=0}^u h_v\Big)},
\end{equation*}
\end{small}
where an empty product is understood as $1$. In particular, $H_{0, 1} = \big\{(1)\big\}$ and $\phi\big((1), X\big) = 1$.

Notice that each term in the denominator also appears in the numerator: $\phi(\uh, X)$ is the product $X(X-1)...(X-k+2)$ with the factors $(X+1-m)$ removed, where $m$ is precisely one of the partial sums $\sum_{v=0}^u h_v$, $\;0\leq u\leq \ell-1$.

We state a combinatorial property of these polynomials, and then use it to compute an explicit expression of the coefficients $\theta_{a, n, k, i, j, \ell}$.

\begin{lemma}
\label{lemmecombinatoire}
    For every $k\in\Z_{>0}, \ell\in\Z$, we have
    \begin{equation*}
        \sum_{\uh\in H_{\ell, k+1}} \phi(\uh, X) = (X-k+1) \sum_{\uh'\in H_{\ell, k}} \phi(\uh', X) + \sum_{\uh'\in H_{\ell-1, k}} \phi(\uh', X).
    \end{equation*}
\end{lemma}

\begin{proof}
    If $\ell\notin\llbracket 0, k\rrbracket$, all of the sums are empty and the equation is trivial. If $\ell\in\llbracket 0, k\rrbracket$, there is a bijection
    \begin{small}
    \begin{equation*}
    \nu: 
        \begin{cases}
            H_{\ell, k+1} \widetilde{\longrightarrow} H_{\ell, k} \sqcup H_{\ell-1, k}, \\
            (h_0, ..., h_{\ell-1}, h_\ell) \mapsto \begin{cases}
                (h_0, ..., h_{\ell-1}, h_\ell - 1) &\text{if $h_\ell \geq 2$}, \\
                (h_0, ..., h_{\ell-1}) &\text{if $h_\ell = 1$}.
            \end{cases}
        \end{cases}
    \end{equation*}
    \end{small}
    Now, let $\uh = (h_0, ..., h_{\ell-1}, h_\ell) \in H_{\ell, k+1}$. 
    \begin{enumerate}
        \item[$\bullet$] In the case $h_\ell \geq 2$, $\uh$ and $\nu(\uh)$ have the same partial sums, thus $\phi(\uh, X)$ and $\phi(\nu(\uh), X)$ have the same denominators. In the numerator, there is an additional factor $(X-k+1)$ in the expression of $\phi(\uh, X)$ because $\uh\in H_{\ell, k+1}$, so that $\phi(\uh, X) = (X-k+1)\phi(\nu(\uh), X)$.
        \item[$\bullet$] In the case $h_\ell = 1$, the last partial sum of $\uh$ equals $k$, hence a factor $(X-k+1)$ in the denominator of $\phi(\uh, X)$, which cancels out with the one in its numerator. Since all of its other partial sums are equal to the ones of $\nu(\uh)$, we have $\phi(\uh, X) = \phi(\nu(\uh), X).$
    \end{enumerate}
    \vspace{-0.6cm}
\end{proof}

\begin{lemma}
    For every $k\in\Z_{>0}$, $i\in\llbracket 1, a\rrbracket$, $j\in\llbracket 0, n/M\rrbracket$, $\ell\in\llbracket 0, a-i\rrbracket$, we have
    \begin{equation}
    \label{expressionexplicitedestheta}
        \theta_{a, n, k, i, j, \ell} = (-1)^\ell\sum_{\uh\in H_{\ell, k}} \phi(\uh, Mj) 
    \end{equation}
\end{lemma}

\begin{remark*}
    In particular, the coefficient $\theta_{a, n, k, i, j, \ell}$ is zero whenever $\ell \geq k$, since in that case $H_{\ell, k} = \varnothing$. 
\end{remark*}

\begin{proof}
    We prove the following formula by induction on $k\in\Z_{>0}$, which suffices to prove the lemma : 
    \begin{small}
    \begin{equation}
    \label{hypothèsederécurrencePki}
        \forall i\in\llbracket 1, a\rrbracket, \hspace{1cm} P_{n, k, i}(z) = \sum_{j=0}^{n/M} \Bigg(\sum_{\ell=0}^{a-i} (-1)^{\ell}\bigg(\sum_{\uh\in H_{\ell, k}} \phi(\uh, Mj)\bigg)c_{n, i+\ell, j}\Bigg) z^{Mj-k+1}.
    \end{equation}
    \end{small}
    For $k=1$, this is exactly the definition $P_{n, 1, i} = P_{n, i} = \sum_{j=0}^{n/M} c_{n, i, j}z^{Mj}$. We now assume that \eqref{hypothèsederécurrencePki} holds for an integer $k\in\Z_{>0}$. From the recurrence relation \eqref{definitiondesPki}, we deduce
    \begin{small}
    \begin{align*}
        P_{n, k+1, i}(z) &= \sum_{j=0}^{n/M} \Bigg( \sum_{\ell=0}^{a-i} (-1)^{\ell}\bigg((Mj-k+1)\sum_{\uh\in H_{\ell, k}} \phi(\uh, Mj)\bigg)c_{n, i+\ell, j}\Bigg)z^{Mj-k} \\
         &\hspace{3cm} - \sum_{j=0}^{n/M} \Bigg( \sum_{\ell=0}^{a-i-1} (-1)^{\ell}\bigg(\sum_{\uh\in H_{\ell, k}}\phi(\uh, Mj)\bigg)c_{n, i+1+\ell, j}\Bigg) z^{Mj-k} \\
        &= \sum_{j=0}^{n/M} \Bigg(\sum_{\ell=0}^{a-i} (-1)^{\ell} \bigg((Mj-k+1)\sum_{\uh\in H_{\ell, k}}\phi(\uh, Mj) + \sum_{\uh\in H_{\ell-1, k}}\phi(\uh, Mj)\bigg)c_{n, i+\ell, j}\Bigg)z^{Mj-k}
    \end{align*}
    \end{small}
    by doing the index translation $\ell \xleftarrow[]{} \ell + 1$ in the second sum over $\ell$, and by artificially adding the term $\ell = 0$, which is zero because $H_{-1, k} = \varnothing$.

    By applying \sref{Lemma}{lemmecombinatoire}, we obtain
    \begin{equation*}
        \forall i\in\llbracket 1, a\rrbracket, \hspace{1cm} P_{n, k+1, i}(z) = \sum_{j=0}^{n/M} \Bigg(\sum_{\ell=0}^{a-i} (-1)^{\ell}\bigg(\sum_{\uh\in H_{\ell, k+1}} \phi(\uh, Mj)\bigg)c_{n, i+\ell, j}\Bigg) z^{Mj-k},
    \end{equation*}
    which concludes the proof by induction.
\end{proof}

\smallskip
We will now deduce from this explicit expression the estimates $(i)$ and $(ii)$ of \sref{Proposition}{propositionexistenceetestimationdestheta}. These estimates are crucial to obtain geometric upper bounds in $n$ for the integers $c_{n, i, j}$ when applying Siegel's lemma in the next subsection. 
    
    \smallskip
    Let $k\geq 1$, $\;i\in\llbracket 1, a\rrbracket$, $\;j\in\llbracket 0, n/M\rrbracket$, $\;\ell\in\llbracket 0, a-i\rrbracket$ and $\uh \in H_{\ell, k}$ be fixed. 

    Because of \eqref{expressionexplicitedestheta}, and since $\mathrm{Card}\big(H_{\ell, k}\big) \leq k^{|\ell|} \leq k^a$, it is enough to prove  $$\left| \frac{\phi(\uh, Mj)}{(k-1)!} \right| \leq 2^n \hspace{0.5cm}\text{and}\hspace{0.5cm} \frac{d_k\Delta_{a, \max(k, n)}}{(k-1)!}\phi(\uh, Mj)\in\mathbb{Z}.$$
    We recall that
    \begin{equation}
    \label{expressiondephihnj}
        \phi(\uh, Mj) = \frac{\prod_{m=1}^{k-1} (Mj+1-m)}{\prod_{u=0}^{\ell-1} \left(Mj+1-\sum_{v=0}^u h_v\right)}.
    \end{equation}
    There are two cases to investigate.

    \begin{enumerate}  
        \item[$\bullet$] 
           If $Mj > k-2$, all factors in the above expression are positive. We can then write
                \begin{small}
                \begin{equation*}
                    \left|\frac{\phi(\uh, Mj)}{(k-1)!}\right| = \left|\frac{\frac{(Mj)!}{(Mj-k+1)!}}{(k-1)!\prod_{u=0}^{\ell -1} \Big(Mj+1-\sum_{v=0}^u h_v\Big)}\right| \leq \binom{Mj}{k-1} \leq 2^{Mj}\leq 2^n,
                \end{equation*}
                \end{small}
                since the product in the denominator has only positive integer factors, and is thus greater than or equal to~$1$. We can also write
                \begin{small}
                \begin{equation*}
                    \frac{d_k\Delta_{a, \max(k, n)}}{(k-1)!}\phi(\uh, Mj) = d_k \binom{Mj}{k-1} \frac{\Delta_{a, \max(k, n)}}{\prod_{u=0}^{\ell -1} \Big(Mj+1-\sum_{v=0}^u h_v\Big)}.
                \end{equation*}
                \end{small}
                This quantity is an integer. Indeed, recall that $\Delta_{a, k}$ was defined in \sref{Notation}{notationdketdeltaak} and notice that the product in the denominator has $\ell \leq a$ distinct factors all lying in $\llbracket 1, n\rrbracket$.

        \item[$\bullet$]
            If $Mj \leq k-2$, a factor in the numerator vanishes in \eqref{expressiondephihnj}. If this factor does not appear in the denominator, then $\phi(\uh, Mj) = 0$. If it does appear in the denominator, say for a partial sum $\sum_{v=0}^{u_0} h_v$, it cancels out and we can write
                    \begin{small}
                    \begin{equation*}
                        \left|\frac{\phi(\uh, Mj)}{(k-1)!}\right| = \left|(-1)^{k-Mj}\frac{(Mj)!(k-2-Mj)!}{(k-1)!\prod_{u=0, u\neq u_0}^{\ell -1} \Big(Mj+1-\sum_{v=0}^u h_v\Big)}\right| \leq \frac{1}{(k-1)\binom{k-2}{Mj}} \leq 1 \leq 2^n,
                    \end{equation*}
                    \end{small}
                    since the product in the denominator has only non zero integer factors, and is thus greater than or equal to~$1$ in absolute value. We can also write
                \begin{small}
                \begin{equation*}
                    \frac{d_k\Delta_{a, \max(k, n)}}{(k-1)!}\phi(\uh, Mj) = (-1)^{Mj-k}\frac{d_k}{(k-1)\binom{k-2}{Mj}}\frac{\Delta_{a, \max(k, n)}}{\displaystyle{\prod_{{u=0}\atop{u\neq u_0}}^{\ell -1} \Big(Mj+1-\sum_{v=0}^u h_v\Big)}}.
                \end{equation*}
                \end{small}
                This quantity is an integer. Indeed, recall that $d_k$ and $\Delta_{a, k}$ were defined in \sref{Notation}{notationdketdeltaak}. The first quotient is an integer because of the main result of \cite{farhi2009}: the $\mathrm{lcm}$ of the numbers $\binom{k-2}{0}, \binom{k-2}{1}, ..., \binom{k-2}{k-2}$ is $\frac{d_{k-1}}{k-1}$. The second quotient is an integer because its denominator has $\ell-1\leq a$ distinct factors, all non-zero and lying in an interval of length $k$, itself included in $\llbracket -k, n\rrbracket$.
    \end{enumerate}

\subsection{Application of Siegel's lemma}
\label{soussectionapplicationdulemmedesiegel}

In this subsection, we complete the proof of \sref{Proposition}{existencedescij} by applying a version of Siegel's lemma due to Fischler \cite[Lemma 1, p.~4]{fischler2026}.

\begin{lemma}
\label{lemmedesiegel}
    (Siegel's lemma) \;Let $L > K \geq K_0 > 0$ be integers. 
    We consider integer coefficients $\lambda_{k, i}$, $\;1\leq k\leq K, \; 1\leq i\leq L$.
    For $1\leq k\leq K$, we consider a real number $H_k \geq \sqrt{\sum_{i=1}^L \lambda_{k, i}^2}$.
    For $K_0+1\leq k\leq K$, we also consider a real number $G_k \geq 1$.
    We define 
     $   X: = \sqrt{L}\Big(H_1...H_{K_0}G_{K_0+1}...G_K\Big)^{\frac{1}{L-K_0}}.$

    Then the linear system with $K_0$ equations, $K-K_0$ inequalities and $L$ unknowns
    \begin{small}
    \begin{equation*}
        \begin{cases}
            \displaystyle{\sum_{i=1}^L \lambda_{k, i} x_i = 0}, &1\leq k\leq K_0, \\
            \displaystyle{\left|\sum_{i=1}^L \lambda_{k, i} x_i \right| \leq \frac{H_k X}{G_k}}, &K_0+1\leq k\leq K,
        \end{cases}
    \end{equation*}
    \end{small}
    has a non-trivial integer solution $\mathbf{x} = (x_1, ..., x_L)\in\Z^L$ such that $\;\;\forall i\in\llbracket 1, L\rrbracket, \;|x_i| \leq X$.
\end{lemma}

We fix parameters $\omega, \Omega$ and $r$ satisfying the assumptions of \sref{Proposition}{existencedescij}, and an integer $n\in\Z_{>0}$ such that $\omega n, \Omega n, \frac{n}{M}, \frac{rn}{M}$ are integers. We apply Siegel's lemma to the following linear system with $K = \Omega n -1$, $K_{0} = \omega n - 1$ and whose $L = a(\frac{n}{M}+1)$ unknowns are $c_{n, i, j}$, $i\in\llbracket 1, a\rrbracket$, $j\in\llbracket 0, n/M\rrbracket$: 
\begin{small}
\begin{equation*}
    \begin{cases}
        \displaystyle{\frac{d_k\Delta_{a, \max(k, n)}}{(k-1)!}P_{n, k, 1}(1) = 0}, &\;\;\;1\leq k\leq \omega n-1,  \hspace{2cm}(\tilde{\imath})\\
        \displaystyle{|\mathfrak{A}_{n,k}| \leq \frac{H_{k}X}{G_{k}}}, &\;\;\;\omega n\leq k\leq \Omega n-1. \hspace{1.6cm}(\tilde{\imath}\tilde{\imath})
    \end{cases}
\end{equation*}
\end{small}
Here, {\small $H_{k}$} and {\small $G_{k}$} are yet to be chosen, and {\small $X = \sqrt{L}\Big(H_{1}...H_{\omega n-1}G_{\omega n}...G_{\Omega n-1}\Big)^{\frac{1}{L-K_{0}}}$}. With the help of \sref{Proposition}{propositionexistenceetestimationdestheta} and \eqref{expressionexplicitedesafrak}, this linear system can be rewritten as
\begin{small}
\begin{equation*}
    \begin{cases}
        \displaystyle{\sum_{j=0}^{n/M} \sum_{\ell=0}^{a-1} \Bigg(\frac{d_k\Delta_{a, \max(k, n)}}{(k-1)!} \theta_{a, n, k, 1, j, \ell}\Bigg) c_{n, 1+\ell, j} = 0}, &1\leq k\leq \omega n-1,  \hspace{2cm}(i)\\
        \displaystyle{\left|\sum_{i=1}^{\min(a, k)} \sum_{j=0}^{n/M} \Bigg(\binom{k-1}{k-i}(-Mj)^{k-i}\Bigg)c_{n, i, j}\right| \leq \frac{H_{k}X}{G_{k}}}, &\omega n\leq k\leq \Omega n-1. \hspace{1.7cm}(ii)
    \end{cases}
\end{equation*}
\end{small}
The coefficients inside parentheses in $(ii)$ are integers bounded by $k^an^k$. Thus, choosing $H_{k} = \sqrt{L}k^an^k$ works for $\omega n\leq k\leq \Omega n - 1$. To obtain the second condition of \sref{Proposition}{existencedescij}, we take $G_{k} = r^{\Omega n - k}$.
Because of \sref{Proposition}{propositionexistenceetestimationdestheta} and \eqref{estimationexplicitedketdeltaak}, the coefficients inside parentheses in $(i)$ are integers bounded~by
\begin{equation*}
    e^{1.04k}(a+1)^{1.04\max(k, n)}e^{1.04\gamma \max(k, n)}k^a2^n \leq (\omega n)^a \Big( 2(a+1)^{1.04\omega} e^{2\omega}\Big)^n,
\end{equation*}
since $\omega\geq 1$,  $1\leq k<\omega n$ and $1.04\gamma + 1.04 < 2$. We can thus choose $$H_{k} := H_1 = \sqrt{L} (\omega n)^a \Big( 2(a+1)^{1.04\omega} e^{2\omega}\Big)^n$$ for $\;1\leq k\leq \omega n-1.$

By Siegel's lemma, we obtain integers $c_{n, i, j}$ satisfying condition $(i)$ of \sref{Proposition}{existencedescij}, and such that
\begin{equation*}
    \begin{cases} 
        |\mathfrak{A}_{n, k}| \leq \sqrt{L}r^{k-\Omega n} k^an^k X, \hspace{1cm} \omega n \leq k < \Omega n \\
        \displaystyle{|c_{n, i, j}| \leq X \leq \sqrt{L}\Bigg[ H_1^{\omega n-1} \prod_{k=\omega n}^{\Omega n-1} r^{\Omega n - k}\Bigg]^{\frac{1}{L-K_{0}}}}.
    \end{cases}
\end{equation*}
We conclude the proof by verifying that
\begin{small}
\begin{align*}
    \log(X) &\leq \frac{1}{2}\log(L) + \frac{1}{L-K_{0}} \\
    &\hspace{-1cm}\times \Bigg[ (\omega n-1)\bigg(\frac{1}{2}\log(L) + a\log(\omega n) + n\big(\log(2) + 1.04\omega\log(a+1) + 2\omega\big)\bigg) + \sum_{k=\omega n}^{\Omega n-1} (\Omega n-k)\log(r)\Bigg] \\
            &\leq \frac{1}{2}\log(L) + \frac{n}{L-K_{0}}\bigg[ \frac{\omega}{2}\log(L) + a\omega\log(\omega n) \bigg] \\
            &\hspace{2.1cm}+ \frac{n^2}{L - K_{0}} \bigg[ \omega\log(2)+1.04\omega^2\log(a+1)+2\omega^2+\frac{1}{2}\Omega^2\log(r) \bigg] \\
            &\leq \frac{1}{2}\frac{a/M}{a/M-\omega}\log(L) + \frac{a\omega}{a/M-\omega}\log(\omega n) + n\log(\xi),
\end{align*}
\end{small}
where we used $\frac{n}{L - K_{0}} \leq \frac{1}{a/M - \omega}$. Thus $X \stackunder{$\leq$}{$\scriptscriptstyle{n\to+\infty}$} \xi^{n+o(n)}$. 

The second inequality in condition $(ii)$ then follows from \eqref{expressionexplicitedesafrak}.

\subsection{A technical result}
\label{soussesctioncalcultechnique}
In \sref{Subsection}{paragraphesurlexpressionexplicitedestheta}, we considered rational functions $P_{n, k, i}$ with $i\geq 1$. But when looking at the differential system $Y'=AY$ where $A$ is given by \eqref{systemediffdespolylogarithmes}, it is natural to also consider the rational functions $P_{n, k, 0}$ and $\Pbar_{n, k, 0}$ defined by induction on $k\geq 1$: 
\begin{equation}
\label{recurrencedesPk0}
    \begin{cases}
        P_{n, 1, 0} = 0, \\
        P_{n, k+1, 0} = P_{n, k, 0}' + \frac{1}{z(1-z)}P_{n, k, 1},
    \end{cases}
    \hspace{2cm}
    \begin{cases}
        \Pbar_{n, 1, 0} = 0, \\
        \Pbar_{n, k+1, 0} = \Pbar_{n, k, 0}' - \frac{1}{1-z}P_{n, k, 1}.
    \end{cases}
\end{equation}

\begin{proposition}
\label{propositionexistenceetestimationdesthetazéro}
    Using the same notation as in \sref{Proposition}{propositionexistenceetestimationdestheta}, there exist rational coefficients $\vartheta_{a, n, k, 0, j, \ell, t}, \varthetabar_{a, n, k, 0, j, \ell, t} \in \Q$ such that for any choice of numbers $c_{n, i, j}$, the rational functions $P_{n, k, 0}$ and $\Pbar_{n, k, 0}$ defined by the recurrence relations \eqref{recurrencedesPk0} satisfy
    \begin{align}
        \begin{split}
        \label{expressiondesP0enfonctiondestheta}
            z^{k-1}(1-z)^{k-1}P_{n, k, 0}(z) = \sum_{t=0}^{n+k-1} \Bigg(\sum_{\ell=0}^{a-1} \sum_{j=0}^{n/M} \vartheta_{a, n, k, 0, j, \ell, t} \;c_{n, 1+\ell, j} \Bigg) z^t, \\
    z^{k-1}(1-z)^{k-1}\Pbar_{n, k, 0}(z) = \sum_{t=0}^{n+k-1} \Bigg(\sum_{\ell=0}^{a-1} \sum_{j=0}^{n/M} \varthetabar_{a, n, k, 0, j, \ell, t} \;c_{n, 1+\ell, j} \Bigg) z^t.
        \end{split}
    \end{align}
    These coefficients satisfy
    \begin{enumerate}
        \item[$(i)$] $|\vartheta_{a, n, k, 0, j, \ell, t}| \leq k^{a+1}8^{\max(k, n)}(k-1)!$, \hspace{1.6cm} $(\overline{i})\;\;|\varthetabar_{a, n, k, 0, j, \ell, t}| \leq k^{a+1}8^{\max(k, n)}(k-1)!,$
        \vspace{1mm}
        \item[$(ii)$] $\displaystyle{\frac{d_k^2\Delta_{a, \max(k, n)}}{(k-1)!}\vartheta_{a, n, k, 0, j, \ell, t}\in\mathbb{Z}}$, \hspace{2.9cm} $(\overline{ii})\;\;\displaystyle{\frac{d_k^2\Delta_{a, \max(k, n)}}{(k-1)!}\varthetabar_{a, n, k, 0, j, \ell, t}\in\mathbb{Z}.}$
    \end{enumerate}
\end{proposition}

It is clear that $z^{k-1}(1-z)^{k-1}P_{n, k, 0}$ and $z^{k-1}(1-z)^{k-1}\Pbar_{n, k, 0}$ are polynomials of degree at most $n+k-1$. The existence of the coefficients $\vartheta_{a, n, k, 0, j, \ell, t}, \varthetabar_{a, n, k, 0, j, \ell, t}$ stems from the recurrence relation \eqref{recurrencedesPk0} and the expression \eqref{definitiondesthetakijij} for the rational functions $P_{n, k, 1}$.

We now give an explicit expression for these coefficients.
We use Pochhammer's symbol
\begin{equation}
\label{Pochhammer}
    (x)_j: = x(x+1)...(x+j-1), \;\;\;\;\;x\in\R, \;\;j\in\Z_{\geq 0}.
\end{equation}
The properties $(x)_j = (-1)^j (-x-j+1)_j$, $\;\;\left(\frac{d}{dz}\right)^j z^x = (x-j+1)_jz^{x-j}$, $\;\;\binom{k}{j} = \frac{(k-j+1)_j}{j!}$ and $\;\;\binom{k}{j} \leq 2^k$ will be used without further reference.

\smallskip
\begin{lemma}
\label{lemmeexpressionexplicitedesthetazero}
    Let $k\geq 1$, $j \in \llbracket 0, n/M \rrbracket$, $\ell\in\llbracket 0, a-1\rrbracket $ and $t\in\llbracket 0, n+k-1\rrbracket$. 

    We have $\vartheta_{a, n, k, 0, j, \ell, t} = 0$ unless $Mj \leq t \leq Mj + k -2$. In the latter case, we have
    \begin{align}
    \label{explicitetheta0}
        \vartheta_{a, n, k, 0, j, \ell, t} = \sum_{u=0}^{t-Mj} \sum_{v=u}^{k-2} \theta_{a, n, k-v-1, 1, j, \ell} &\binom{k-u-2}{t-Mj-u} (-1)^{t-Mj-u} \\
        &\hspace{2.8cm}\times(v-u+1)_u (Mj-k+u+2)_{v-u}. \nonumber
    \end{align}

    \smallskip
    Moreover, we have $\varthetabar_{a, n, k, 0, j, \ell, t} = 0$ unless $Mj + 1 \leq t \leq Mj + k - 1$. In the latter case, we have
    \begin{align}
    \label{explicitethetabar0}
        \varthetabar_{a, n, k, 0, j, \ell, t} = \sum_{u=0}^{t-Mj-1} \sum_{v=u}^{k-2} \theta_{a, n, k-v-1, 1, j, \ell} &\binom{k-u-2}{t-1-Mj-u}(-1)^{t-Mj-u} \\
        &\hspace{2.4cm}\times (v-u+1)_u (Mj-k+u+3)_{v-u}. \nonumber
    \end{align}
\end{lemma}

\begin{proof}
    From the recurrence relation \eqref{recurrencedesPk0}, we derive \begin{small}$$P_{n, k, 0}(z) = \sum_{v=0}^{k-2} \left(\frac{d}{dz}\right)^v \Bigg[\frac{P_{n, k-v-1, 1}(z)}{z(1-z)}\Bigg].$$\end{small}
    We use the expression \eqref{definitiondesthetakijij} for $P_{n, k, 1}$ and the Leibniz rule for differentiation to compute 
    
    \begin{align*}
        &P_{n, k, 0}(z) = \sum_{v=0}^{k-2} \left(\frac{d}{dz}\right)^v \Bigg[ \sum_{j=0}^{n/M} \bigg(\sum_{\ell=0}^{a-1} \theta_{a, n, k-v-1, 1, j, \ell}\;c_{n, 1+\ell, j}\bigg)\frac{z^{Mj-k+v+1}}{1-z}\Bigg] \\
                    &\hspace{0.7cm}= \sum_{j=0}^{n/M} \sum_{v=0}^{k-2} \bigg(\sum_{\ell=0}^{a-1} \theta_{a, n, k-v-1, 1, j, \ell}\;c_{n, 1+\ell, j}\bigg) \sum_{u=0}^v \binom{v}{u} (Mj-k+u+2)_{v-u}z^{Mj-k+u+1}\frac{u!}{(1-z)^{u+1}} \\
                    &\hspace{0.7cm}= \sum_{j=0}^{n/M} \sum_{u=0}^{k-2} \Bigg( \sum_{\ell=0}^{a-1} \bigg(\sum_{v=u}^{k-2} \theta_{a, n, k-v-1, 1, j, \ell} (v-u+1)_u (Mj-k+u+2)_{v-u} \bigg)c_{n, 1+\ell, j} \Bigg) \frac{z^{Mj-k+u+1}}{(1-z)^{u+1}}.
    \end{align*}

    For $u \leq k-2$, we have $\frac{1}{(1-z)^{u+1}} = (1-z)^{1-k}\sum_{w=0}^{k-u-2} (-1)^w \binom{k-u-2}{w} z^w$, and thus
    \begin{align*}
        &P_{n, k, 0}(z) = z^{1-k}(1-z)^{1-k} \sum_{j=0}^{n/M} \sum_{u=0}^{k-2} \sum_{w = 0}^{k-u-2} z^{Mj+u+w}\\
        &\times \Bigg( \sum_{\ell=0}^{a-1} \bigg(\sum_{v=u}^{k-2}\theta_{a, n, k-v-1, 1, j, \ell} \binom{k-u-2}{w} (-1)^w (v-u+1)_u (Mj-k+u+2)_{v-u} \bigg) c_{n, 1+\ell, j}  \Bigg) .
    \end{align*}
    By identifying this with the expression \eqref{expressiondesP0enfonctiondestheta}, the expected expression for $\vartheta_{a, n, k, 0, j, \ell, t}$ is obtained by looking at the pairs $(u, w)$ with $Mj + u + w = t$. Indeed, in the cases $t < Mj$ and $t > Mj + k -2$, there are none since $0 \leq u + w \leq k-2$. In the case $Mj \leq t \leq Mj + k -2$, we sum over $u\leq t-Mj$ by letting $w = t - Mj - u$.

    We proceed in a similar way for $\Pbar_{n, k, 0}$.
\end{proof}

\smallskip
We shall deduce the estimates $(i)$ and $(ii)$ of \sref{Proposition}{propositionexistenceetestimationdesthetazéro} from these explicit expressions. First, notice that for $v \geq u \geq 0$,
        \begin{equation*}
            (v-u+1)_u (Mj-k+u+2)_{v-u} = \begin{cases}
                v!\binom{Mj-k+v+1}{Mj-k+u+1} &\text{if $Mj-k+u+2 > 0$,} \\
                (-1)^{v-u}v!\binom{k-Mj-u-2}{k-Mj-v-2} &\text{if $Mj-k+v+1 < 0$,} \\
                0 &\text{otherwise.}
            \end{cases}
        \end{equation*}
    In all three cases, the product $(v-u+1)_u (Mj-k+u+2)_{v-u}$ can be written as $v!\mathscr{B}$ with $\mathscr{B}\in\mathbb{Z}$, and 
       $ |(v-u+1)_u (Mj-k+u+2)_{v-u}| \leq v!2^{\max(k, n)}.$
    We obtain $(i)$ using \eqref{explicitetheta0} and writing: 
        \begin{align*}
              |\vartheta_{a, n, k, 0, j, \ell, t}| &\leq \sum_{u=0}^{t-Mj} \sum_{v=u}^{k-2} |\theta_{a, n, k-v-1, 1, j, \ell}| |(v-u+1)_u (Mj-k+u+2)_{v-u}| \binom{k-u-2}{t-Mj-u} \\
                                     &\leq \sum_{v=0}^{k-2} k^a 2^n (k-v-1)! v!2^{\max(k, n)} \sum_{u=0}^{\min(t-Mj, k-2)} \binom{k-u-2}{t-Mj-u} \hspace{0.6cm} \text{(\sref{Proposition}{propositionexistenceetestimationdestheta})} \\
                                      &\leq (k-1)k^a2^n(k-1)!2^{\max(k, n)}2^k \\
                                      &\leq k^{a+1}8^{\max(k, n)}(k-1)!.
        \end{align*}
    To obtain $(ii)$, we write for $0\leq u\leq v\leq k-2$:
        \begin{align*}
              &\frac{d_k^2\Delta_{a, \max(k, n)}}{(k-1)!} \theta_{a, n, k-v-1, 1, j, \ell}(v-u+1)_u(Mj-k+u+2)_{v-u} \\
              =\;\; &\frac{d_k\Delta_{a, \max(k, n)}\theta_{a, n, k-v-1, 1, j, \ell}}{(k-v-2)!}\frac{d_k}{\binom{k-2}{v}(k-1)}\mathscr{B}.
         \end{align*}
    This quantity is an integer. Indeed, the first quotient is an integer because of \sref{Proposition}{propositionexistenceetestimationdestheta}, and the second quotient is an integer because of the main result of \cite{farhi2009} mentioned at the end of \sref{Subsection}{paragraphesurlexpressionexplicitedestheta}. By \eqref{explicitetheta0}, this proves $(ii)$.

    We proceed in a very similar way to obtain  $(\overline{i})$ and $(\overline{ii})$, which concludes the proof.

\section{Construction of the linear combinations}
\label{partie4}

We fix an integer $a\geq 3M$ and rational parameters $r, \omega$ and $\Omega$ satisfying $2 \leq 2r < \omega \leq \Omega < \frac{a}{M}$.
We also fix an integer parameter $h\in\llbracket 0, a\rrbracket$, a rational parameter $\kappa$ satisfying $2r < \kappa < \omega$, and a parity $\varepsilon$. We denote by $\mathcal{N}$ the infinite set of integers $n\geq 2$ such that $\omega n, \Omega n, \frac{n}{M}, \frac{rn}{M}$ and $\kappa n$ are all integers. 

In this section, we fix $n\in\mathcal{N}$ and we construct complex numbers $\Lambda_{n, (p, k)}$, ${0\leq p\leq h}$, $2rn+2 \leq k \leq \kappa n$, that are linear combinations with rational coefficients of the numbers $-z_0^mf(m), 0\leq m\leq M-1$ and $L(f, i, z_0), \;\;1\leq i\leq a+h, \;\;i\equiv\varepsilon \mod 2$. We will see in \sref{Section}{partie5} that these linear combinations actually have integer coefficients.

\smallskip
In \sref{Subsection}{soussection4.1}, we construct for $p\in\llbracket 0, h\rrbracket$ linear combinations $S_{n, p}^{[\infty]}(z)$ and $S_{n, p}^{[0]}(z)$ with coefficients in $\Q[z]$ of the functions $1$ and $\Li_i(1/z)$, respectively $\Li_i(z)$, $1\leq i\leq a+h$.

In \sref{Subsection}{paragrapheprocessusdedérivation}, we differentiate them with respect to $z$ to obtain for $k\in\llbracket 2rn+2, \kappa n\rrbracket$ linear combinations $S_{n, p}^{[\infty](k-1)}(z)$ and $S_{n, p}^{[0](k-1)}(z)$ with coefficients in $\Q(z)$ of the functions $1$ and $\Li_i(1/z)$, respectively $\Li_i(z)$, $1\leq i\leq a+h$.

In \sref{Subsection}{soussection4.3}, we combine the quantities $S_{n, p}^{[\infty](k-1)}\left(\frac{1}{\mu^\ell z}\right)$ and $S_{n, p}^{[0](k-1)}(\mu^\ell z)$, $0\leq\ell\leq M-1$, to obtain linear combinations $\widetilde{\Lambda}_{n, (p, k)}(z)$ with coefficients in $\Q[z^{\pm 1}]$ of the constant functions $f(m), 0\leq m\leq M-1$ and the functions $L(f, i, z)$, $1\leq i\leq a+h$.

In \sref{Subsection}{paragraphenotationpourlasuite}, we evaluate at the point $z=z_0$ defined in \sref{Notation}{notationdudébut} to obtain linear combinations $\Lambda_{n, (p, k)}$ with coefficients in $\Q$ of the numbers $-z_0^mf(m), 0\leq m\leq M-1$ and $L(f, i, z_0), \;\;1\leq i\leq a+h, \;\;i\equiv\varepsilon \mod 2$. We then introduce some notation to improve clarity for \sref{Sections}{partie5},\sref{}{partie6} and \sref{}{partie7}.

\subsection{Linear combinations of polylogarithms}
\label{soussection4.1}

\smallskip
Let $F_n$ be the rational function whose coefficients $c_{n, i, j}$ of the partial fraction decomposition \eqref{DESdefn} are given by \sref{Proposition}{existencedescij}. We consider for $p\in\llbracket 0, h\rrbracket$ the following series. They converge for $|z|\geq 1$, respectively $|z|\leq 1$, since $\mathrm{deg}\big(F_n^{(p)}\big) \leq -\omega n -p \leq -2$: 
\begin{equation}
\label{définitiondessériesS}
    S_{n, p}^{[\infty]}(z): = z^{rn}\sum_{t=rn+1}^{+\infty} F_n^{(p)}(t)z^{-t}, \;\;\;\; S_{n, p}^{[0]}(z): = z^{rn}\sum_{t=rn+1}^{+\infty} F_n^{(p)}(-t)z^t, \;\;\;\;\;\;\;\;\;\; 0\leq p\leq h.
\end{equation}

The following lemma extends \cite[Lemma 4]{fischler2026}. It shows that these series are linear combinations of $1$ and the polylogarithms, with coefficients in $\Q[z]$:

\begin{lemma}
\label{lesSsontdesclenlespolylogs}
    There exist polynomials with rational coefficients $V_{n, p}^{[\infty]}, V_{n, p}^{[0]}$ of degree at most $2rn$ and $Q_{n, i, p}, \;\;1\leq i\leq a+h,$ of degree at most $(r+1)n$ such that
    \begin{align}
         \begin{split}
         \label{sériescommedesclenlespolylogarithmesbonscoefficients}
            S_{n, p}^{[\infty]}(z) &= V_{n, p}^{[\infty]}(z) + \sum_{i=1}^{a+h} Q_{n, i, p}(z)\Li_i(1/z), \hspace{2cm}|z| \geq 1,\\
            S_{n, p}^{[0]}(z) &= V_{n, p}^{[0]}(z) + \sum_{i=1}^{a+h} Q_{n, i, p}(z)(-1)^i\Li_i(z), \hspace{1.6cm}|z|\leq 1.
         \end{split}
    \end{align}
\end{lemma}

    The proof is as follows. Differentiating equation \eqref{DESdefn} $p$ times, we get
    \begin{equation}
    \label{dérivéepièmdelafractionrationnelle}
        F_n^{(p)}(t) = \sum_{i=1}^a \sum_{j=0}^{n/M} \frac{c_{n, i, j}(-1)^p(i)_p}{(t+Mj)^{i+p}}.
    \end{equation}
    On the one hand, we compute for $|z| \geq 1$: 
    \begin{small}
    \begin{align*}
        S_{n, p}^{[\infty]}(z) &= z^{rn}\sum_{t=rn+1}^{+\infty} \sum_{i=1}^a \sum_{j=0}^{n/M} \frac{c_{n, i, j}(-1)^p (i)_p}{(t+Mj)^{i+p}}z^{-t} = \sum_{i=1}^a (-1)^p(i)_pz^{rn} \sum_{j=0}^{n/M} c_{n, i, j} \sum_{t=rn+1}^{+\infty} \frac{z^{-t}}{(t+Mj)^{i+p}} \\
                        &= \sum_{i=1}^a (-1)^p(i)_pz^{rn} \sum_{j=0}^{n/M} c_{n, i, j}z^{Mj} \sum_{t=rn+Mj+1}^{+\infty} \frac{z^{-t}}{t^{i+p}} \hspace{2cm}(t \xleftarrow{} t + Mj) \\
                        &= \sum_{i=1}^a (-1)^p (i)_pz^{rn} \sum_{j=0}^{n/M} c_{n, i, j} z^{Mj}  \Bigg[ \Li_{i+p}(1/z) - \sum_{t=1}^{rn+Mj} \frac{z^{-t}}{t^{i+p}} \Bigg] \\
                        &= V_{n, p}^{[\infty]}(z) + \sum_{i=1}^a (-1)^p(i)_pz^{rn}P_{n, i}(z)\Li_{i+p}(1/z),
    \end{align*}
    \end{small}
    where the polynomials $P_{n, i}$ are defined by \eqref{definitiondesPki}. The substitution $t \xleftarrow[]{} rn + Mj - t$ yields
    \begin{small}
    \begin{equation*}
        V_{n, p}^{[\infty]}(z): =  -\sum_{i=1}^a (-1)^p(i)_p \sum_{j=0}^{n/M} c_{n, i, j} \sum_{t=0}^{rn+Mj-1}\frac{z^{t}}{(rn+Mj-t)^{i+p}}.
    \end{equation*}
    \end{small}
    On the other hand, we compute for $|z| \leq 1$: 
    \begin{small}
    \begin{align*}
        S_{n, p}^{[0]}(z) &= z^{rn}\sum_{t=rn+1}^{+\infty} \sum_{i=1}^a \sum_{j=0}^{n/M} \frac{c_{n, i, j}(-1)^p(i)_p}{(Mj-t)^{i+p}}z^t = \sum_{i=1}^a (-1)^p(i)_pz^{rn} \sum_{j = 0}^{n/M} c_{n, i, j}(-1)^{i+p} \sum_{t = rn+1}^{+\infty} \frac{z^t}{(t-Mj)^{i+p}} \\
                    &= \sum_{i=1}^a (-1)^p(i)_pz^{rn} \sum_{j=0}^{n/M} c_{n, i, j}z^{Mj}(-1)^{i+p} \sum_{t=rn-Mj+1}^{+\infty} \frac{z^t}{t^{i+p}} \hspace{2.5cm}(t \xleftarrow{} t - Mj) \\
                    &= \sum_{i=1}^a (-1)^p(i)_pz^{rn}\sum_{j=0}^{n/M} c_{n, i, j}z^{Mj}\Bigg[(-1)^{i+p}\Li_{i+p}(z) - \sum_{t=1}^{rn-Mj} \frac{z^t}{(-t)^{i+p}} \Bigg] \\
                    &= V_{n, p}^{[0]}(z) + \sum_{i=1}^a (-1)^p(i)_pz^{rn}P_{n, i}(z)(-1)^{i+p}\Li_{i+p}(z).
    \end{align*}
    \end{small}
    The substitution $t \xleftarrow[]{} rn + Mj + t$ yields
    \begin{small}
    \begin{equation*}
        V_{n, p}^{[0]}(z): = -\sum_{i=1}^a (-1)^p(i)_p \sum_{j=0}^{n/M} c_{n, i, j} \sum_{t=rn+Mj+1}^{2rn} \frac{z^t}{(rn+Mj-t)^{i+p}}.
    \end{equation*}
    \end{small}
    
    It is convenient to introduce the polynomials
    \begin{equation}
    \label{définitiondesQip}
        Q_{n, i+p, p}(z): = (-1)^p(i)_pz^{rn}P_{n, i}(z), \;\;1\leq i\leq a, \;\;\;\;\;\;\;\;\;\;\;\; Q_{n, i, p} = 0, \;\; i\notin \llbracket p+1, a+p \rrbracket,
    \end{equation}
    so that the above results can be rewritten as
    \begin{small}
    \begin{align*}
        \begin{split}
            S_{n, p}^{[\infty]}(z) &= V_{n, p}^{[\infty]}(z) + \sum_{i=1}^{a+h} Q_{n, i, p}(z)\Li_{i}(1/z), \hspace{2cm} |z|\geq 1,\\
            S_{n, p}^{[0]}(z) &= V_{n, p}^{[0]}(z) + \sum_{i=1}^{a+h} Q_{n, i, p}(z)(-1)^{i}\Li_{i}(z), \hspace{1.6cm} |z|\leq 1.
        \end{split}
    \end{align*}
    \end{small}

\subsection{Differentiation process}
\label{paragrapheprocessusdedérivation}
Following \cite[Section 4.1]{fischler2026}, in order to produce additional linear combinations and apply \sref{Proposition}{critèred'indépendancelinéaire} to them, we consider derivatives of the functions $S_{n, p}^{[\infty]}(z), S_{n, p}^{[0]}(z)$.

The following differentiated series for $(p, k)\in\llbracket 0, h\rrbracket \times \llbracket 2rn+2, \kappa n\rrbracket$ are, again, linear combinations of $1$ and the polylogarithms:
\begin{small}
\begin{align}
    \begin{split}
    \label{sériesdérivéescommedesclenlespolylogarithmesbonscoefficients}
    S_{n, p}^{[\infty](k-1)}(z) &= Q_{n, 0, (p, k)}(z) + \sum_{i=1}^{a+h} Q_{n, i, (p, k)}(z)\Li_i(1/z), \hspace{2.05cm}|z| \geq 1, \\
    S_{n, p}^{[0](k-1)}(z) &= \Qbar_{n, 0, (p, k)}(z) + \sum_{i=1}^{a+h} Q_{n, i, (p, k)}(z)(-1)^i\Li_i(z), \hspace{1.5cm}|z| \leq 1.
    \end{split}
\end{align}
\end{small}
Indeed, the vectors of functions $\Big(1, 0, \Li_1(1/z), \Li_2(1/z), ..., \Li_{a+h}(1/z)\Big)$ and $\Big(0, 1, -\Li_1(z),$ $\Li_2(z), ..., (-1)^{a+h}\Li_{a+h}(z)\Big)$ are solutions of the differential system $Y' = AY$ where $A$ is the same as in \eqref{systemediffdespolylogarithmes} with $a$ replaced by $a+h$, so that the rational functions $Q_{n, i, (p, k)}$ can be computed with the following recurrence relation and the convention $Q_{n, a+h+1, (p, k)} = 0$: 
\begin{small}
\begin{equation}
\label{récurrencevérifiéeparlesQipk}
    \begin{cases}
        Q_{n, i, (p, k+1)} = Q'_{n, i, (p, k)} - \frac{1}{z} Q_{n, i+1, (p, k)}, &Q_{n, i, (p, 1)} = Q_{n, i, p}, \;\;\;\;\;\;\;1\leq i\leq a+h, \\
        Q_{n, 0, (p, k+1)} = Q'_{n, 0, (p, k)} + \frac{1}{z(1-z)} Q_{n, 1, (p, k)}, &Q_{n, 0, (p, 1)} = 0, \\
        \Qbar_{n, 0, (p, k+1)} = \Qbar'_{n, 0, (p, k)} - \frac{1}{1-z} Q_{n, 1, (p, k)}, &\Qbar_{n, 0, (p, 1)} = 0.
    \end{cases}
\end{equation}
\end{small}

We only consider indices $k \geq 2rn+2$ so that the terms $V_{n, p}^{[\infty](k-1)}$ and $V_{n, p}^{[0](k-1)}$, which would otherwise appear in \eqref{sériesdérivéescommedesclenlespolylogarithmesbonscoefficients}, vanish. The polynomials $V_{n, p}^{[\infty]}$ and $V_{n, p}^{[0]}$ are indeed of degree at most $(r+1)n-1$, respectively $2rn$.

Moreover, the only singularities that could appear at $z=1$ in equation \eqref{sériesdérivéescommedesclenlespolylogarithmesbonscoefficients} are: 
\begin{enumerate}
    \item[$\bullet$] a logarithmic singularity due to $\Li_1$ if $Q_{n, 1, (p, k)}(1) \neq 0$,
    \item[$\bullet$] a pole if $Q_{n, 0, (p, k)}$ or $\Qbar_{n, 0, (p, k)}$ has a pole at $1$.
\end{enumerate}
But we only consider indices $k \leq \kappa n$, so that the series $S_{n, p}^{[\infty](k-1)}(z)$ and $S_{n, p}^{[0](k-1)}(z)$ converge on the circle $|z|=1$. Indeed, we have
\begin{small}
\begin{align*}
    \begin{split}
        S_{n, p}^{[\infty](k-1)}(z) = z^{rn}\sum_{t=rn+1}^{+\infty} F_n^{(p)}(t)(-1)^{k-1}(t-rn)_{k-1}z^{-t-k+1}, \\
        S_{n, p}^{[0](k-1)}(z) = z^{rn}\sum_{t=rn+1}^{+\infty} F_n^{(p)}(-t)(rn+t-k+2)_{k-1}z^{t-k+1},
    \end{split}
\end{align*}
\end{small}
and the terms $F_n^{(p)}(t)(-1)^{k-1}(t-rn)_{k-1}$ and $F_n^{(p)}(-t)(rn+t-k+2)_{k-1}$ are rational functions of degree $-\omega n -p +k - 1 \leq -2$ in the variable $t$. Thus, for $k\leq\kappa n <\omega n$, neither of these two types of singularity occurs at $z=1$, because they cannot cancel each other out. In particular, $Q_{n, 1, (p, k)}(1) = 0$ and the only pole that $Q_{n, 0, (p, k)}$ and $\Qbar_{n, 0, (p, k)}$ can possibly have is 0. Because of the recurrence relation \eqref{récurrencevérifiéeparlesQipk}, this pole has order at most $k-1$. Thus, we have
\begin{small}
\begin{equation}
    \begin{cases}
    \label{pasdedivergencepolynomialenilogarithmique}
        \forall (p,k)\in\llbracket 0, h\rrbracket\times\llbracket 2rn+2,\kappa n\rrbracket, \;\; Q_{n, 1, (p, k)}(1) = 0, \\
        \forall (p,k)\in\llbracket 0, h\rrbracket\times\llbracket 2rn+2,\kappa n\rrbracket, \;\; \text{$z^{k-1}Q_{n, 0, (p, k)}$, $z^{k-1}\Qbar_{n, 0, (p, k)}$ are polynomials of degree $\leq (r+1)n$.}
    \end{cases}
\end{equation}
\end{small}

\subsection{Linear combinations of $L$ functions}
\label{soussection4.3}

For $1\leq i\leq a+h$, notice that the polynomials $Q_{n, i, p}$ defined by \eqref{définitiondesQip} are in $\Q[z^M]$, since $P_{n, i}\in\Q[z^M]$ and $M \mid rn$. The recurrence relation \eqref{récurrencevérifiéeparlesQipk} then implies $z^{k-1}Q_{n, i, (p, k)} \in \Q[z^M]$ for all $k \geq 1$. Recall that $\mu$ is a primitive $M$-th root of unity, so that
\begin{equation}
\label{lesQnipksontdansQzN}
    \forall i\in\llbracket 1, a+h\rrbracket,\;\;\;\forall \ell\in\llbracket 0, M-1\rrbracket, \;\;\;\;\;\;\;\;  (\mu^\ell z)^{k-1}Q_{n, i, (p, k)}(\mu^\ell z) = z^{k-1}Q_{n, i, (p, k)}(z).
\end{equation}

\smallskip
For $i=0$, the polynomials $z^{k-1}Q_{n, 0, (p, k)}(z)$ and $z^{k-1}\Qbar_{n, 0, (p, k)}(z)$ are not, in general, in $\Q[z^M]$. This is why we introduce the decompositions
\begin{small}
\begin{align}
    \begin{split}
    \label{décompositiondesQ0}
    z^{k-1}Q_{n, 0, (p, k)}(z) &=: \sum_{m=0}^{M-1} z^mQ_{n, 0, (p, k)}^{<m>}(z), \;\;\;\;\;\;\;\;\;\;Q_{n, 0, (p, k)}^{<m>}\in\Q[z^M], \\
    z^{k-1}\Qbar_{n, 0, (p, k)}(z) &=: \sum_{m=0}^{M-1} z^m\Qbar_{n, 0, (p, k)}^{<m>}(z), \;\;\;\;\;\;\;\;\;\;\Qbar_{n, 0, (p, k)}^{<m>}\in\Q[z^M]. 
    \end{split} 
\end{align}
\end{small}
We also set $Q_{n, 0, (p, k)}^{<M>} := z^{-M}Q_{n, 0, (p, k)}^{<0>} \in \Q[z^{\pm M}]$. Thus, for $0\leq m\leq M$, $z^m Q_{n, 0, (p, k)}^{<m>}$ is the sum of the monomials of $z^{k-1}Q_{n, 0, (p, k)}(z)$ whose exponents are congruent to $m \mod M$.

Consider the coefficients
\begin{equation*}
    \hat{f}(\ell): = \frac{1}{M}\sum_{m=0}^{M-1} f(m)\mu^{-\ell m}, \;\;\;\;\;\;\;\; 0 \leq \ell \leq M-1.
\end{equation*}
They allow us to reconstruct the function $L(f, i, z)$ from the functions $\Li_i(z), \Li_i(\mu z), \Li_i(\mu^2 z), ...,$ $\Li_i(\mu^{M-1} z)$, as shown in the following lemma. We also state a polynomial identity there, that will be used in \sref{Subsection}{paragraphestructureshidlovskii}.
\begin{lemma}
\label{formulesd'inversiondeFourier}
    (Fourier inversion formulas) \;\;The following formulas hold.
    \begin{enumerate}
        \item[$(i)$] $\sum_{\ell = 0}^{M-1} \hat{f}(\ell)\mu^{\ell m} = f(m),\hspace{1.9cm} m\in\Z$,
        \item[$(ii)$] $\sum_{\ell = 0}^{M-1} \hat{f}(\ell)\Li_i(\mu^\ell z) = L(f, i, z), \hspace{0.5cm} i\in\Z_{>0}, \;|z|\leq 1, \;(i, z) \notin \{(1, \mu^\ell) \;|\; 0\leq\ell\leq M-1 \}$.
    \end{enumerate}
    \smallskip
    Moreover, the polynomials $Q_{n, 0, (p, k)}^{\{\ell\}}(z): = \mu^{\ell(k-1)}Q_{n, 0, (p, k)}(\mu^\ell z)$ satisfy for any $z\in\C^*$
    \begin{enumerate}
        \item[$(iii.a)$] $\sum_{m=0}^{M-1} \mu^{\ell m}z^{m}Q_{n, 0, (p, k)}^{<m>}(z) = z^{k-1}Q_{n, 0, (p, k)}^{\{\ell\}}(z), \hspace{2.75cm}0\leq \ell\leq M-1$,
        \item[$(iii.b)$] $\frac{1}{M}\sum_{\ell=0}^{M-1}\mu^{-\ell m}z^{k-1}Q_{n, 0, (p, k)}^{\{\ell\}}(z) = z^mQ_{n, 0, (p, k)}^{<m>}(z),\hspace{2.1cm}0\leq m\leq M$.
    \end{enumerate}
    The same holds for $0\leq m\leq M-1$ if $Q$ is replaced by $\Qbar$.
\end{lemma}

\begin{proof}
    \begin{enumerate}
        \item[]
        \item[$(i)$] We compute for $m\in\Z$
            \begin{align*}
                \sum_{\ell=0}^{M-1} \hat{f}(\ell) \mu^{\ell m} = \frac{1}{M} \sum_{\ell=0}^{M-1} \sum_{\widetilde{m}=0}^{M-1} f(\widetilde{m}) \mu^{-\ell(\widetilde{m}-m)} = f(m) + \frac{1}{M}\sum_{{\widetilde{m}=0}\atop{\widetilde{m} \not\equiv m \mod M}}^{M-1} f(\widetilde{m}) \left( \sum_{\ell = 0}^{M-1} \mu^{-\ell(\widetilde{m}-m)} \right).
            \end{align*}        
            We obtain the result since for $\widetilde{m} \not\equiv m \mod M$, the inner sum over $\ell$ vanishes.
        \item[$(ii)$]
            We use $(i)$ and compute
            \begin{align*}
                \sum_{\ell=0}^{M-1} \hat{f}(\ell)\Li_i(\mu^\ell z) &= \sum_{\ell=0}^{M-1} \hat{f}(\ell) \left(\sum_{m=1}^{+\infty} \frac{\mu^{\ell m} z^m}{m^i}\right) = \sum_{m=1}^{+\infty} \frac{z^m}{m^i} \left(\sum_{\ell=0}^{M-1} \hat{f}(\ell) \mu^{\ell m}\right) = \sum_{m=1}^{+\infty} \frac{f(m)z^m}{m^i}.
            \end{align*}
        \item[$(iii.a)$] It is enough to apply, for $\ell \in \llbracket 0, M-1\rrbracket$, the substitution $\;z \xleftarrow[]{} \mu^\ell z\;$ in definition \eqref{décompositiondesQ0}. Recall that for $m\in\llbracket 0, M-1\rrbracket$ we have $Q_{n, 0, (p, k)}^{<m>}(\mu^\ell z) = Q_{n, 0, (p, k)}^{<m>}(z)$, since $Q_{n, 0, (p, k)}^{<m>} \in \Q[z^{\pm M}]$.
        \item[$(iii.b)$] We use $(iii.a)$ and compute for $m \in \llbracket 0, M-1\rrbracket$: 
            \begin{align*}
                \frac{1}{M}\sum_{\ell=0}^{M-1} \mu^{-\ell m} &z^{k-1}Q_{n, 0, (p, k)}^{\{\ell\}}(z) = \frac{1}{M}\sum_{\ell=0}^{M-1} \sum_{\widetilde{m} = 0}^{M-1} \mu^{\ell(\widetilde{m}-m)} z^{\widetilde{m}} Q_{n, 0, (p, k)}^{<\widetilde{m}>}(z) \\
                &\hspace{1cm}= z^{m}Q_{n, 0, (p, k)}^{<m>}(z) + \frac{1}{M}\sum_{{\widetilde{m}=0}\atop{\widetilde{m} \neq m}}^{M-1} z^{\widetilde{m}} Q_{n, 0, (p, k)}^{<\widetilde{m}>}(z) \left( \sum_{\ell = 0}^{M-1} \mu^{-\ell(\widetilde{m}-m)} \right).
            \end{align*}
            We thus obtain the result for $m \in \llbracket 0, M-1\rrbracket$, since for $\widetilde{m} \neq m$ the inner sum over $\ell$ vanishes. The result for $m=M$ follows from the case $m=0$.
    \end{enumerate}
    \vspace{-4mm}
\end{proof}

We define, for $(p, k)\in\llbracket 0, h\rrbracket \times \llbracket 2rn+2, \kappa n\rrbracket$ and $z\in\C^*$ satisfying $|z| \leq 1$: 
\begin{equation}
\label{définitiondeLambdapk}
    \widetilde{\Lambda}_{n, (p, k)}(z): = \sum_{\ell = 0}^{M-1} \hat{f}(\ell) \Bigg[  \left(\frac{1}{\mu^\ell z}\right)^{k-1}S_{n, p}^{[\infty](k-1)}\left(\frac{1}{\mu^\ell z}\right) + (-1)^\varepsilon (\mu^\ell z)^{k-1}S_{n, p}^{[0](k-1)}(\mu^\ell z)  \Bigg]. \\
\end{equation}
This is a linear combination with coefficients in $\Q[z^{\pm 1}]$ of the constant functions $f(m)$ and the functions $L(f, i, z)$. 

\begin{lemma}
\label{lemmeexpressiondelambdaenfonctiondez}
    Let $(p, k)\in\llbracket 0, h\rrbracket\times\llbracket 2rn+2, \kappa n\rrbracket$. The following equation holds for every\footnote{Recall that \eqref{pasdedivergencepolynomialenilogarithmique} shows $Q_{n, 1, (p, k)}(1) = 0$. Thus, the series $L(f, 1, 1) = \sum_{m=1}^{+\infty} \frac{f(m)}{m}$, which is divergent if $\sum_{m=0}^{M-1}f(m) \neq 0$, does not appear if we evaluate at $z=1$.} $z\in\C^*$ with $|z| \leq 1$: 
    \begin{align}
        \begin{split}
        \label{expressiondelambdaenfonctiondez}
            \widetilde{\Lambda}_{n, (p, k)}(z) &= \sum_{m=0}^{M-1} \Bigg( \frac{1}{z^{M-m}}Q_{n, 0, (p, k)}^{<M-m>}\left(\frac{1}{z}\right) + (-1)^\varepsilon z^m\Qbar_{n, 0, (p, k)}^{<m>}(z)\Bigg) f(m) \\
                        &\hspace{1.2cm}+\sum_{i=1}^{a+h} \Bigg( \frac{1}{z^{k-1}} Q_{n, i, (p, k)}\left(\frac{1}{z}\right) + (-1)^{i+\varepsilon} z^{k-1}Q_{n, i, (p, k)}(z) \Bigg)L(f, i, z).
        \end{split}
    \end{align}
\end{lemma}

\begin{proof}
    We compute
    \begin{small}
    \begin{align*}
        \widetilde{\Lambda}_{n, (p, k)}(z) &= \sum_{\ell = 0}^{M-1} \hat{f}(\ell) \Bigg[  \sum_{m=0}^{M-1} \left(\frac{1}{\mu^\ell z}\right)^m Q_{n, 0, (p, k)}^{<m>}\left(\frac{1}{\mu^\ell z}\right) + (-1)^\varepsilon \sum_{m=0}^{M-1} (\mu^\ell z)^m\Qbar_{n, 0, (p, k)}^{<m>}(\mu^\ell z)  \Bigg] \\
                         &\hspace{0.5cm}+ \sum_{\ell = 0}^{M-1} \hat{f}(\ell) \Bigg[  \sum_{i = 1}^{a+h} \left(\frac{1}{\mu^\ell z}\right)^{k-1}Q_{n, i, (p, k)}\left(\frac{1}{\mu^\ell z}\right)\Li_i(\mu^\ell z) \hspace{1.5cm}  \\
                         &\hspace{2cm} + (-1)^\varepsilon \sum_{i = 1}^{a+h} (\mu^\ell z)^{k-1}Q_{n, i, (p, k)}(\mu^\ell z)(-1)^i\Li_i(\mu^\ell z) \Bigg] \hspace{0.5cm}\text{\Big(from \eqref{sériesdérivéescommedesclenlespolylogarithmesbonscoefficients} and \eqref{décompositiondesQ0}\Big)} \\ 
                         &= \sum_{m=0}^{M-1} \Bigg[  \frac{1}{z^m} Q_{n, 0, (p, k)}^{<m>}\left(\frac{1}{z}\right)\Bigg(\sum_{\ell = 0}^{M-1} \hat{f}(\ell)\mu^{-\ell m} \Bigg) + (-1)^\varepsilon z^m \Qbar_{n, 0, (p, k)}^{<m>}(z) \Bigg(\sum_{\ell = 0}^{M-1} \hat{f}(\ell)\mu^{\ell m} \Bigg) \Bigg] \\
                         &\hspace{0.5cm}+ \sum_{i=1}^{a+h} \Bigg[  \Bigg( \frac{1}{z^{k-1}} Q_{n, i, (p, k)}\left(\frac{1}{z}\right) + (-1)^{i+\varepsilon}z^{k-1}Q_{n, i, (p, k)}(z) \Bigg) \Bigg( \sum_{\ell = 0}^{M-1} \hat{f}(\ell) \Li_i(\mu^\ell z)\Bigg) \Bigg] \\ 
                         &\hspace{10.6cm}\text{\Big(from \eqref{lesQnipksontdansQzN} and \eqref{décompositiondesQ0}\Big)} \\
                         &= \sum_{m=0}^{M-1} \frac{1}{z^{m}}Q_{n, 0, (p, k)}^{<m>}\left(\frac{1}{z}\right)f(-m) + (-1)^\varepsilon \sum_{m=0}^{M-1} z^m\Qbar_{n, 0, (p, k)}^{<m>}(z) f(m)   \\
                         &\hspace{0.5cm}+ \sum_{i=1}^{a+h} \Bigg( \frac{1}{z^{k-1}} Q_{n, i, (p, k)}\left(\frac{1}{z}\right) + (-1)^{i+\varepsilon} z^{k-1}Q_{n, i, (p, k)}(z) \Bigg) L(f, i, z).\hspace{0.5cm}\text{\Big(\sref{Lemma}{formulesd'inversiondeFourier}\Big)} 
    \end{align*}
    \end{small}
Substituting $m \xleftarrow[]{} M-m$ in the first sum and using the fact that $f$ is $M$-periodic and $Q_{n, 0, (p, k)}^{<M>}(1/z) = (1/z)^{-M}Q_{n, 0, (p, k)}^{<0>}(1/z)$, we get
\begin{align*}
    \widetilde{\Lambda}_{n, (p, k)}(z) &= \sum_{m=0}^{M-1} \Bigg( \frac{1}{z^{M-m}}Q_{n, 0, (p, k)}^{<M-m>}\left(\frac{1}{z}\right) + (-1)^\varepsilon z^m\Qbar_{n, 0, (p, k)}^{<m>}(z)\Bigg) f(m)   \\
                         &\hspace{0.5cm}+ \sum_{i=1}^{a+h} \Bigg( \frac{1}{z^{k-1}} Q_{n, i, (p, k)}\left(\frac{1}{z}\right) + (-1)^{i+\varepsilon} z^{k-1}Q_{n, i, (p, k)}(z) \Bigg) L(f, i, z).
\end{align*}
\end{proof}

\subsection{Notation for later use}
\label{paragraphenotationpourlasuite}

We define $\delta_{n, k}: = d_k^2\Delta_{a+h, \max\big(k, (r+1)n\big)}$, where $d_k$ and $\Delta_{a, k}$ are defined in \sref{Notation}{notationdketdeltaak}. The quantity $\delta_{n , k}$ depends on the parameters $a+h \leq 2a$ and $r$. We are interested in its asymptotics as $n\to+\infty$, with $k$ bounded by $\kappa n$. Since $\kappa > 2r \geq r+1$, we have $\max\big(k, (r+1)n\big) \leq \kappa n$ and \eqref{estimationexplicitedketdeltaak} yields
\begin{equation}
\label{asymptotiquededelta}
    \delta_{n, k} \leq \Big( e^3 (2a+1)^{1.04} \Big)^{\kappa n}.
\end{equation}

Let us consider the functions $\frac{\delta_{n, k}}{(k-1)!}\widetilde{\Lambda}_{n, (p, k)}(z)$. According to \sref{Lemma}{lemmeexpressiondelambdaenfonctiondez}, these are linear combinations of the functions $f(0), ..., f(M-1)$, $L(f, 1, z), ...,$ $L(f, a+h, z)$ with coefficients in $\Q[z^{\pm1}]\;$ (we will show in \sref{Subsection}{paragrapheestimationdescoefficients} that the coefficients are actually in $\Z[z^{\pm 1}]$). The differential system \eqref{grossystemediffcompliqué} that we will consider to apply ``Shidlovskii's lemma" in \sref{Section}{partie6} has a singularity at every $M$-th root of unity. To avoid these singularities, we evaluate at the point $z=z_0$ defined in \sref{Notation}{notationdudébut}, which is an $M$-th root of $-1$. We thus consider for every $n\in\mathcal{N}$ and every $(p, k)\in\llbracket 0, h\rrbracket\times\llbracket 2rn+2, \kappa n\rrbracket$ the quantity
\begin{equation}
\label{definitiondelambdafinal}
    \Lambda_{n, (p, k)} := \frac{\delta_{n, k}}{(k-1)!}\widetilde{\Lambda}_{n, (p, k)}(z_0),
\end{equation}
Recall that the polynomials $z^{k-1}Q_{n, i, (p, k)}$ and $Q_{n, 0, (p, k)}^{<M-m>}$ are in $Q[z^{\pm M}]$ and that $z_0^{\pm M} = -1$, so that they take the same value when evaluated at $z_0$ or $z_0^{-1}$. Equation \eqref{expressiondelambdaenfonctiondez} evaluated at $z=z_0$ thus yields
\begin{small}
\begin{align}
    \begin{split}
        \label{cllambdaen-1}
    \Lambda_{n, (p, k)} &= \frac{\delta_{n, k}}{(k-1)!}\sum_{m=0}^{M-1} -z_0^m \bigg( Q_{n, 0, (p, k)}^{<M-m>}(z_0) - (-1)^\varepsilon \Qbar_{n, 0, (p, k)}^{<m>}(z_0) \bigg) f(m) \\
                        &\hspace{2.3cm}+ \frac{\delta_{n, k}}{(k-1)!}\sum_{i=1}^{a+h} \Big(1 + (-1)^{i+\varepsilon}\Big)z_0^{k-1}Q_{n, i, (p, k)}(z_0)L(f, i, z_0).
    \end{split}
\end{align}
\end{small}

These are linear combinations of the numbers $-z_0^mf(m)$, $0\leq m\leq M-1$ and $L(f, i, z_0)$, $1\leq i\leq a+h$. All of the terms $L(f, i, z_0)$ with $i\not\equiv\varepsilon\mod 2$ are absent, because of the coefficient $\big(1+(-1)^{i+\varepsilon}\big)$ in \eqref{cllambdaen-1}. Let us define
\begin{equation}
\label{définitiondeszetai}
    \zeta_i: = \begin{cases}
        2L(f, i, z_0) &\text{if $i\equiv\varepsilon\mod 2$}, \\
        0 &\text{otherwise,}
    \end{cases} \hspace{0.3cm} 1\leq i\leq a+h, \hspace{0.6cm} \zeta_0^{<m>}: = -z_0^m f(m), \hspace{0.3cm} 0\leq m\leq M-1.
\end{equation}
Equation \eqref{cllambdaen-1} can then be rewritten as
\begin{equation}
\label{Lambdacommecldeszeta}
    \Lambda_{n, (p, k)} = \sum_{m=0}^{M-1}\lambda_{n, 0, (p, k)}^{<m>} \zeta_0^{<m>} + \sum_{i=1}^{a+h} \lambda_{n, i, (p, k)} \zeta_i,
\end{equation}
with
\begin{small}
\begin{equation}
\label{definitiondespetitslambdas}
    \begin{split}
        \lambda_{n, i, (p, k)} &:= \frac{\delta_{n, k}}{(k-1)!}z_0^{k-1}Q_{n, i, (p, k)}(z_0), \\
    \lambda_{n, 0, (p, k)}^{<m>} &:= \frac{\delta_{n, k}}{(k-1)!}\bigg( Q_{n, 0, (p, k)}^{<M-m>}(z_0) - (-1)^\varepsilon\Qbar_{n, 0, (p, k)}^{<m>}(z_0) \bigg).
    \end{split}
\end{equation}
\end{small}

\section{Upper bounds on the combinations and their coefficients}
\label{partie5}

To apply the linear independence criterion (\sref{Proposition}{critèred'indépendancelinéaire}) to the linear combinations $\Lambda_{n, (p, k)}$ defined in \sref{Subsection}{paragraphenotationpourlasuite}, we show in \sref{Proposition}{majorationdescoefficients} that they have integer coefficients growing at most geometrically with $n$. We then show in \sref{Proposition}{majorationdescombinaisons} that if $r\geq 2$, $|\Lambda_{n, (p, k)}|$ is bounded by an exponential term $\alpha^{n+o(n)}$. Notice that with the choices of parameters made in \sref{Section}{partie7} and for sufficiently large $a$, we will have $0 < \alpha < 1$, so that $|\Lambda_{n, (p, k)}|$ will tend to $0$ at least geometrically as $n\to+\infty$. These two propositions are generalizations of \cite[Lemma 5]{fischler2026} and \cite[Lemma 6]{fischler2026}, respectively.

\subsection{Upper bound on the coefficients}
\label{paragrapheestimationdescoefficients}

Recall that $\xi$ was defined in \sref{Proposition}{existencedescij}.

\begin{proposition}
\label{majorationdescoefficients}
    For every $n\in\mathcal{N}$, every $i\in\llbracket 1, a+h\rrbracket$, every $m\in\llbracket 0, M-1\rrbracket$ and uniformly for $(p, k)\in\llbracket 0, h\rrbracket \times \llbracket 2rn+2, \kappa n\rrbracket$, we have
    \begin{enumerate}[left=9pt]
        \item[$(i)$] $|\lambda_{n, i, (p, k)}|,\; |\lambda_{n, 0, (p, k)}^{<m>}| \leq \beta^{n+o(n)}$, with $\beta: = \Big(32e^3(2a+1)^{1.04}\Big)^\kappa \xi$,
        \item[$(ii)$] $\lambda_{n, i, (p, k)},\; \lambda_{n, 0, (p, k)}^{<m>} \in \mathbb{Z}$.
    \end{enumerate}
\end{proposition}

\begin{proof}
    We fix $p\in\llbracket 0, h\rrbracket$. We write $Q_{n, i, p}(z) =: \sum_{j=0}^{(r+1)n/M} \gamma_{n, i, j}z^{Mj}$ for $i\geq 1$. From the definition \eqref{définitiondesQip} we deduce
    \begin{equation}
        \begin{cases}
        \label{majorationdesgammaenfonctiondesc}
            \gamma_{n, i, j} \in \Z, \\
            \displaystyle{\max_{{1\leq i \leq a+h}\atop{0\leq j\leq (r+1)n/M}} |\gamma_{n, i, j}| \leq (a+h)_p \max_{{1\leq i\leq a}\atop{0\leq j\leq n/M}} |c_{n, i, j}|}.
        \end{cases}
    \end{equation}
    
    The rational functions $Q_{n, i, (p, k)}$, $Q_{n, 0, (p, k)}$ and $\Qbar_{n, 0, (p, k)}$ satisfy the recurrence relations \eqref{récurrencevérifiéeparlesQipk}, which are exactly the same as equations $\eqref{definitiondesPki}$ and \eqref{recurrencedesPk0}. Propositions \sref{}{propositionexistenceetestimationdestheta} and \sref{}{propositionexistenceetestimationdesthetazéro} give
    \begin{small}
    \begin{align*}
        z^{k-1}Q_{n, i, (p, k)}(z) &= \sum_{j=0}^{(r+1)n/M} \Bigg( \sum_{\ell=0}^{a+h-i} \theta_{a+h, (r+1)n, k, i, j, \ell} \; \gamma_{n, i+\ell, j} \Bigg) z^{Mj},\\
        z^{k-1}(1-z)^{k-1}Q_{n, 0, (p, k)}(z) &= \sum_{t=0}^{(r+1)n+k-1} \Bigg( \sum_{\ell=0}^{a+h-1} \sum_{j=0}^{(r+1)n/M} \vartheta_{a+h, (r+1)n, k, 0, j, \ell, t} \; \gamma_{n, 1+\ell, j} \Bigg) z^t, \\
        z^{k-1}(1-z)^{k-1}\Qbar_{n, 0, (p, k)}(z) &= \sum_{t=0}^{(r+1)n+k-1} \Bigg( \sum_{\ell=0}^{a+h-1} \sum_{j=0}^{(r+1)n/M} \varthetabar_{a+h, (r+1)n, k, 0, j, \ell, t} \; \gamma_{n, 1+\ell, j} \Bigg) z^t,
    \end{align*}
    \end{small}
    with
    \begin{align}
    &\begin{cases}
    \label{majorationtheta}
        |\theta_{a+h, (r+1)n, k, i, j, \ell}| \leq k^{a+h} 2^{(r+1)n} (k-1)!, \\
        |\vartheta_{a+h, (r+1)n, k, 0, j, \ell, t}|, \;|\varthetabar_{a+h, (r+1)n, k, 0, j, \ell, t}| \leq k^{a+h+1}8^{\max\big(k, (r+1)n\big)}(k-1)!,
    \end{cases} \\
    &\begin{cases}
    \label{denomtheta}
        \frac{\delta_{n, k}}{(k-1)!}\theta_{a+h, (r+1)n, k, i, j, \ell} \in \mathbb{Z}, \\
        \frac{\delta_{n, k}}{(k-1)!}\vartheta_{a+h, (r+1)n, k, 0, j, \ell, t}, \;\;\frac{\delta_{n, k}}{(k-1)!}\varthetabar_{a+h, (r+1)n, k, 0, j, \ell, t} \in \mathbb{Z}.
    \end{cases}
    \end{align}
    For $k$ bounded by $\kappa n$, since $|z_0| = 1$, we deduce from \eqref{majorationdesgammaenfonctiondesc} and \eqref{majorationtheta} the asymptotic upper bound
    
    \begin{small}
    \begin{align*}
        |\lambda_{n, i, (p, k)}| &=  \frac{\delta_{n, k}}{(k-1)!}  \left|z_0^{k-1}Q_{n, i, (p, k)}(z_0)\right| \hspace{7.5cm}\text{\big(from \eqref{definitiondespetitslambdas}\big)}\\
        &\leq \frac{\delta_{n, k}}{(k-1)!} \Bigg(\sum_{j=0}^{(r+1)n/M}  \sum_{\ell=0}^{a+h-i} |\theta_{a+h, (r+1)n, k, i, j, \ell}| \Bigg) \max_{i, j} |\gamma_{n, i, j}| \\
        &\leq \frac{\delta_{n, k}}{(k-1)!}\Big((r+1)\frac{n}{M} + 1\Big)(a+h)k^{a+h}2^{(r+1)n}(k-1)! (a+h)_h \max_{i, j}|c_{n, i, j}| \\
        &\stackunder{$\leq$}{$\scriptscriptstyle{n\to+\infty}$} \Big( e^3(2a+1)^{1.04}\Big)^{\kappa n} \Big((r+1)\frac{n}{M} + 1\Big)(a+h)k^{a+h} 2^{\kappa n} (a+h)_h \xi^{n+o(n)} \\
        &\hspace{7cm} \big(\text{from \eqref{asymptotiquededelta}, $r+1 < \kappa$ and \sref{Proposition}{existencedescij}}\big) \\
        &\stackunder{$\leq$}{$\scriptscriptstyle{n\to+\infty}$} \Bigg(\Big(2e^3(2a+1)^{1.04}\Big)^\kappa \xi \Bigg)^{n+o(n)}.
    \end{align*}
    \end{small}

    Moreover, we compute the coefficients of the polynomial $z^{k-1}Q_{n, 0, (p, k)}(z)$ as the Cauchy product of $z^{k-1}(1-z)^{k-1}Q_{n, 0, (p, k)}(z)$ and $\frac{1}{(1-z)^{k-1}} = \sum_{t=0}^{+\infty} \binom{t+k-2}{t} z^t$. For $q\in\llbracket 0, (r+1)n\rrbracket$, the coefficient of $z^q$ in $z^{k-1}Q_{n, 0, (p, k)}(z)$ is
    \begin{equation}
    \label{coefficientsdupolynomezk-1Q0}
        \sum_{t=0}^q \binom{q-t+k-2}{q-t} \Bigg( \sum_{\ell=0}^{a+h-1} \sum_{j=0}^{(r+1)n/M} \vartheta_{a+h, (r+1)n, k, 0, j, \ell, t} \; \gamma_{n, 1+\ell, j} \Bigg).
    \end{equation}
    
    By definition \eqref{décompositiondesQ0}, the polynomial $Q_{n, 0, (p, k)}^{<m>}$ has, for every $m\in\llbracket 0, M-1\rrbracket$, at most $\frac{(r+1)n}{M}+1$ non-zero coefficients, which are all of this form. This also holds when replacing $Q$ by $\Qbar$, and also holds for the Laurent polynomial $Q_{n, 0, (p, k)}^{<M>} = z^{-M}Q_{n, 0, (p, k)}^{<0>}.$ We can thus compute for $0\leq m\leq M-1$, recalling that $|z_0| = 1$ \\
    \begin{small}
    \begin{align*}
        |\lambda_{n, 0, (p, k)}^{<m>}| &\leq \frac{\delta_{n, k}}{(k-1)!} \bigg(\left|Q_{n, 0, (p, k)}^{<M-m>}(z_0)\right| + \left|\Qbar_{n, 0, (p, k)}^{<m>}(z_0)\right|\bigg) \hspace{4cm}\text{\big(from \eqref{definitiondespetitslambdas}\big)} \\
        &\leq \frac{\delta_{n, k}}{(k-1)!}2\Big(\frac{(r+1)n}{M}+1\Big) \\
        &\hspace{2.5cm} \times \sum_{t=0}^{(r+1)n} \Bigg( 2^{(r+1)n+k - 2}\bigg( \sum_{\ell=0}^{a+h-1} \sum_{j=0}^{(r+1)n/M} |\vartheta_{a+h, (r+1)n, k, 0, j, \ell, t}| \bigg)\Bigg) \max_{i, j}|\gamma_{n, i, j}| \\
        &\leq \frac{\delta_{n, k}}{(k-1)!}2\Big(\frac{(r+1)n}{M}+1\Big)\Big((r+1)n+1\Big)2^{(r+1)n+\kappa n-2}(a+h)\Big(\frac{(r+1)n}{M}+1\Big) \\
        &\hspace{6cm} \times k^{a+h+1}8^{\max\big(k, (r+1)n\big)}(k-1)!(a+h)_h \max|c_{n, i, j}| \\
        &\stackunder{$\leq$}{$\scriptscriptstyle{n\to+\infty}$} \Big(e^3(2a+1)^{1.04}\Big)^{\kappa n} \Big((r+1)n+1\Big)^3 2^{2\kappa n} (a+h)k^{a+h+1}  8^{\kappa n} (a+h)_h \xi^{n+o(n)} \\ &\hspace{7cm} \big(\text{from \eqref{asymptotiquededelta}, $r+1 < \kappa$ and \sref{Proposition}{existencedescij}}\big) \\
        &\stackunder{$\leq$}{$\scriptscriptstyle{n\to+\infty}$} \Bigg(\Big(32e^3(2a+1)^{1.04}\Big)^\kappa \xi \Bigg)^{n+o(n)}.
    \end{align*}
    \end{small}

    Lastly, we deduce from \eqref{majorationdesgammaenfonctiondesc}, \eqref{denomtheta} and \eqref{coefficientsdupolynomezk-1Q0} that the polynomials $\frac{\delta_{n, k}}{(k-1)!}z^{k-1}Q_{n, i, (p, k)}(z)$, $\frac{\delta_{n, k}}{(k-1)!}z^{k-1}Q_{n, 0, (p, k)}(z)$ and $\frac{\delta_{n, k}}{(k-1)!}z^{k-1}\Qbar_{n, 0, (p, k)}(z)$ are in $\Z[z]$. Then, by equations \eqref{lesQnipksontdansQzN} and \eqref{décompositiondesQ0}, the polynomials $\frac{\delta_{n, k}}{(k-1)!}z^{k-1}Q_{n, i, (p, k)}(z)$, $\frac{\delta_{n, k}}{(k-1)!}Q_{n, 0, (p, k)}^{<m>}(z)$ and $\frac{\delta_{n, k}}{(k-1)!}\Qbar_{n, 0, (p, k)}^{<m>}(z)$ are in $\Z[z^{\pm M}]$. Therefore, they take integer values when evaluated at $z_0$, which is an $M$-th root of $-1$. This proves $(ii)$ by \eqref{definitiondespetitslambdas}.
\end{proof}


\subsection{Upper bound on the combinations}
\label{paragrapheestimationdelacombinaison}

Recall that $\xi$ was defined in \sref{Proposition}{existencedescij} and that the Pochhammer symbol $(x)_j$ and some of its properties are given by \eqref{Pochhammer}. 

\begin{proposition}
\label{majorationdescombinaisons}
    If $r\geq 2$, then for every $n\in\mathcal{N}$ and uniformly for $(p, k)\in\llbracket 0, h\rrbracket \times \llbracket 2rn+2, \kappa n\rrbracket$, we have
    \begin{equation*}
        \left| \Lambda_{n, (p, k)}\right| \leq \alpha^{n+o(n)}, \;\;\;\text{with}\; \alpha: = \frac{1}{r^\Omega} \Big(e^4 (2a+1)^{1.04}\Big)^\kappa \xi,
    \end{equation*}
    provided $\alpha \neq 1$.
\end{proposition}

    To prove this proposition, we start by bounding $S_{n, p}^{[\infty]}(z)$ from above, for $|z|\geq 1$, $z\neq 1$. Recall that this series is defined by \eqref{définitiondessériesS}. Differentiating equation \eqref{DESdefn} $p$ times, we get
    \begin{equation*}
        F_n^{(p)}(t) = \sum_{q=1}^{+\infty} \frac{(-1)^p(q)_p\mathfrak{A}_{n, q}}{t^{q+p}}, \;\;\;\text{and thus}\; S_{n, p}^{[\infty]}(z) = \sum_{t=rn+1}^{+\infty} \sum_{q=1}^{+\infty} \frac{(-1)^p (q)_p \mathfrak{A}_{n, q}}{t^{q+p}} z^{rn-t}.
    \end{equation*}
    Because of \sref{Proposition}{existencedescij} $(i)$, we have $\mathfrak{A}_{n, q} = 0$ for $q\leq \omega n-1$, so that the inner sum actually starts at $q=\omega n$. Thus
    \begin{small}
    \begin{align*}
        \frac{\delta_{n, k}}{(k-1)!}S_{n, p}^{[\infty](k-1)}(z) &= \delta_{n, k} \sum_{t=rn+1}^{+\infty} \sum_{q=\omega n}^{+\infty} \frac{(-1)^p (q)_p \mathfrak{A}_{n, q}}{t^{q+p}}\frac{(rn-t-k+2)_{k-1}}{(k-1)!}z^{rn-t-k+1} \\
                                                    &= (-1)^{k-1}\delta_{n, k} \sum_{t=rn+1}^{+\infty} \sum_{q=\omega n}^{+\infty} \frac{(-1)^p (q)_p \mathfrak{A}_{n, q}}{t^{q+p}} \binom{t-rn+k-2}{k-1} z^{rn-t-k+1}.
    \end{align*}
    \end{small}
    Since $\frac{1}{t^p} \leq 1$ and $|z^{rn-t-k+1}| \leq 1$ for $t\geq rn+1$, we can write
    \begin{small}
    \begin{align}
    \label{majoration1}
        \left| \frac{\delta_{n, k}}{(k-1)!}S_{n, p}^{[\infty](k-1)}(z) \right| &\leq \delta_{n, k} \sum_{t=rn+1}^{+\infty} \binom{t-rn+k-2}{k-1} \left(\frac{n}{t}\right)^{\omega n} \sum_{q=\omega n}^{+\infty} \frac{(q)_p|\mathfrak{A}_{n, q}|}{t^q} \left(\frac{n}{t}\right)^{-\omega n} \nonumber \\ 
                                                                    &\leq \delta_{n, k} \sum_{t=rn+1}^{+\infty} \binom{t-rn+k-2}{k-1} \left(\frac{n}{t}\right)^{\omega n} \sum_{q=\omega n}^{+\infty} u_{n, t, q},
    \end{align}
    \end{small}
    where the expressions
    \begin{small}
    \begin{align*}
        u_{n, t, q} &\stackunder{$=$}{$\scriptscriptstyle{n\to+\infty}$} \frac{1}{r^{(\Omega - \omega)n}} (q)_p q^a \left(\frac{rn}{t}\right)^{q-\omega n} \xi^{n+o(n)}, &\omega n\leq q\leq\Omega n -1,\\
        u_{n, t, q} &\stackunder{$=$}{$\scriptscriptstyle{n\to+\infty}$} (q)_p q^a \left(\frac{n}{t}\right)^{q-\omega n} \xi^{n+o(n)}, &q\geq \Omega n,
    \end{align*}
    \end{small}
    stem from \sref{Proposition}{existencedescij} $(ii)$.

    On the one hand, we have for $t \geq rn+1$
    \begin{small}
    \begin{align*}
        \sum_{q=\omega n}^{\Omega n-1} u_{n, t, q} \stackunder{$\leq$}{$\scriptscriptstyle{n\to+\infty}$}  \sum_{q=\omega n}^{\Omega n-1} \frac{1}{r^{(\Omega - \omega)n}} (\Omega n)_p (\Omega n)^a \xi^{n+o(n)} = (\Omega - \omega)n \frac{1}{r^{(\Omega - \omega)n}} (\Omega n)_p (\Omega n)^a \xi^{n+o(n)}.
    \end{align*}
    \end{small}
    On the other hand, for $t\geq rn+1$ and $q\geq\Omega n$, consider the quotient
    \begin{small}
    \begin{equation}
    \label{nassezgrand1}
        \frac{u_{n, t, q+1}}{u_{n, t, q}} = \frac{n}{t}\left(1+\frac{1}{q}\right)^a\left(1+\frac{p}{q}\right) \leq \frac{1}{r}\left(1+\frac{1}{q}\right)^a\left(1+\frac{p}{q}\right).
    \end{equation}
    \end{small}
    For $n$ (and thus $q$) sufficiently large, this quotient is bounded by $\frac{3}{2r}$ . Since $r\geq 2$, this quotient is bounded by $\frac{3}{4}$ when $n$ is sufficiently large. It follows that
    \begin{small}
    \begin{align*}
        \sum_{q=\Omega n}^{+\infty} u_{n, t, q} \leq u_{n, t, \Omega n}\sum_{q=\Omega n}^{+\infty} \left(\frac{3}{4}\right)^{q-\Omega n} \stackunder{$\leq$}{$\scriptscriptstyle{n\to+\infty}$} \frac{4}{r^{(\Omega - \omega)n}} (\Omega n)_p (\Omega n)^a \xi^{n+o(n)}.
    \end{align*}   
    \end{small}

    Inequality \eqref{majoration1} thus becomes
    \begin{small}
    \begin{align}
    \label{majoration2}
        \left| \frac{\delta_{n, k}}{(k-1)!}S_{n, p}^{[\infty](k-1)}(z) \right| &\stackunder{$\leq$}{$\scriptscriptstyle{n\to+\infty}$} \delta_{n, k} \sum_{t=rn+1}^{+\infty} \binom{t-rn+k-2}{k-1} \left(\frac{n}{t}\right)^{\omega n} \nonumber \\
                                                                    &\hspace{4.5cm} \times\Big((\Omega-\omega)n + 4\Big)\frac{1}{r^{(\Omega-\omega)n}} (\Omega n)_p(\Omega n)^a\xi^{n+o(n)} \nonumber \\
                                                                    &\stackunder{$\leq$}{$\scriptscriptstyle{n\to+\infty}$} \delta_{n, k} \frac{(\Omega n + p + 4)^{a+p+1}}{r^{\Omega n}} \xi^{n+o(n)} \sum_{t=rn+1}^{+\infty} \binom{t-rn+k-2}{k-1} \left(\frac{rn}{t}\right)^{\omega n}. 
    \end{align}
    \end{small}
To obtain an upper bound on the latter sum, notice that for $t \geq rn+1$ we have
\begin{small}
\begin{equation*}
    t-rn+k-2 \leq t +rn\left(\frac{k-1}{rn}-1\right) \leq  t+t\left(\frac{k-1}{rn}-1\right) = \frac{k-1}{rn}t.
\end{equation*}
\end{small}
By using $(k-1)! \geq \left(\frac{k-1}{e}\right)^{k-1}$, we get
\begin{small}
\begin{align*}
    \sum_{t=rn+1}^{+\infty} \binom{t-rn+k-2}{k-1} \left(\frac{rn}{t}\right)^{\omega n} &\leq \sum_{t=rn+1}^{+\infty} \frac{\left(\frac{k-1}{rn}t\right)^{k-1}}{(k-1)!}\left(\frac{rn}{t}\right)^{\omega n} \\
                                                                                        &\leq \sum_{t=rn+1}^{+\infty} \frac{e^{k-1}\left(\frac{k-1}{rn}\right)^{k-1}(rn)^{k-1}}{(k-1)^{k-1}}\left(\frac{rn}{t}\right)^{\omega n - k - 1}\left(\frac{rn}{t}\right)^2 \\
                                                                                        &\leq \frac{\pi^2(rn)^2}{6}e^{k-1} \hspace{1cm} \text{(since $\omega n-k-1>\kappa n -k-1 \geq -1$).}
\end{align*}
\end{small}
By using \eqref{asymptotiquededelta}, inequality \eqref{majoration2} finally becomes
\begin{small}
\begin{equation}
\begin{split}
\label{majoration3}
    \left| \frac{\delta_{n, k}}{(k-1)!}S_{n, p}^{[\infty](k-1)}(z) \right| &\stackunder{$\leq$}{$\scriptscriptstyle{n\to+\infty}$} \delta_{n, k} \frac{(\Omega n + p + 4)^{a+p+1}}{r^{\Omega n}} \xi^{n+o(n)} \frac{\pi^2(rn)^2}{6}e^{\kappa n-1}  \\
                    &\stackunder{$\leq$}{$\scriptscriptstyle{n\to+\infty}$} \bigg( \frac{1}{r^\Omega}\Big(e^4(2a+1)^{1.04}\Big)^\kappa \xi \bigg)^{n+o(n)}\hspace{-0.9cm}. \hspace{2cm}\text{(since $\alpha\neq 1$)}
\end{split}
\end{equation}
\end{small}

We proceed in a similar way to establish
\begin{small}
\begin{equation}
\label{majoration5}
    \left| \frac{\delta_{n, k}}{(k-1)!}S_{n, p}^{[0](k-1)}(z) \right| \stackunder{$\leq$}{$\scriptscriptstyle{n\to+\infty}$} \bigg( \frac{1}{r^\Omega}\Big(e^4(2a+1)^{1.04}\Big)^\kappa \xi \bigg)^{n+o(n)}\hspace{-0.9cm}.
\end{equation}
\end{small}
To conclude, recall that $\widetilde{\Lambda}_{n, (p, k)}(z)$ and $\Lambda_{n, (p, k)}$ are defined by \eqref{définitiondeLambdapk} and \eqref{definitiondelambdafinal}, so that
\begin{small}
\begin{align*}
    \Lambda_{n, (p, k)} &= \frac{\delta_{n, k}}{(k-1)!}\widetilde{\Lambda}_{n, (p, k)}(z_0) \\
    &= \sum_{\ell=0}^{M-1} \hat{f}(\ell) \Bigg[ \frac{(\mu^\ell z_0)^{-(k-1)}\delta_{n, k}}{(k-1)!}S_{n, p}^{[\infty](k-1)}\big((\mu^\ell z_0)^{-1}\big)+(-1)^\varepsilon \frac{(\mu^\ell z_0)^{k-1}\delta_{n, k}}{(k-1)!} S_{n, p}^{[0](k-1)}(\mu^\ell z_0)\Bigg].
\end{align*}
\end{small}
The upper bounds \eqref{majoration3} and \eqref{majoration5} hold for any $z$ with $|z|=1$ and $z\neq 1$. Recall that $z_0$ is defined in \sref{Notation}{notationdudébut} and is an $M$-th root of $-1$, so that all of the points $(\mu^\ell z_0)^{\pm 1}$ have modulus $1$ and are different from $1$. Hence
\begin{small}
\begin{equation}
\begin{split}
    \left| \Lambda_{n, (p, k)} \right| \stackunder{$\leq$}{$\scriptscriptstyle{n\to+\infty}$}  \max_{0\leq\ell\leq M-1}|\hat{f}(\ell)|\cdot2M \bigg( \frac{1}{r^\Omega}\Big(e^4(2a+1)^{1.04}\Big)^\kappa \xi \bigg)^{n+o(n)} \stackunder{$\leq$}{$\scriptscriptstyle{n\to+\infty}$} \bigg( \frac{1}{r^\Omega}\Big(e^4(2a+1)^{1.04}\Big)^\kappa \xi \bigg)^{n+o(n)}\hspace{-0.9cm}. 
\end{split}
\end{equation}
\end{small}

\section{Application of a zero estimate}
\label{partie6}

In this section, we prove that the linear combinations $\Lambda_{n, (p, k)}$ constructed in \sref{Subsection}{paragraphenotationpourlasuite} satisfy assumption $(iii)$ of the linear independence criterion (\sref{Proposition}{critèred'indépendancelinéaire}). 

\smallskip
In \sref{Subsection}{paragrapheénoncédeShidlovskii}, we state a version of ``Shidlovskii's lemma" in a general context.

In \sref{Subsection}{paragraphestructureshidlovskii}, we state \sref{Proposition}{propositiondéduitedeshidlovskii}, which asserts that assumption $(iii)$ is satisfied by the linear combinations $\Lambda_{n, (p, k)}$ under an assumption on the parameters $\omega, r, h$ and $\kappa$. The rest of \sref{Section}{partie6} is dedicated to the proof of this proposition, whose strategy is detailed at the end of \sref{Subsection}{paragraphestructureshidlovskii}.

\subsection{Statement of a ``Shidlovskii's lemma"}

\label{paragrapheénoncédeShidlovskii}

Here, we give the context and the statement of a ``Shidlovskii's lemma" that will be applied in \sref{Subsection}{appplicationdeShidlovskii} to ensure that we constructed sufficiently many linear combinations in \sref{Section}{partie4}. This generic term stands for a result based on Shidlovskii's ideas \cite{shidlovskii89}, which have been generalized many times, notably by Bertrand and Beukers \cite{bertrandbeukers85}, Bertrand \cite{bertrand2012} and Fischler \cite{fischler2018}.
Consider an integer $d\geq 1$, a matrix $A \in M_d\big(\C(z)\big)$ and weighting polynomials $S_1, ..., S_{d}\in\C[X]$ of degree at most $\Delta$. To each solution $Y = \prescript{t}{}{(}y_1, ..., y_{d})$ of the differential system $Y' = AY$, we associate the remainder
    $\rho(Y)(z): = \sum_{i = 1}^d S_{i}(z)y_i(z).$

Recall that (see for example \cite[Theorem 6.6, p.~181]{coddingtoncarlson97}) for a regular singular point $\sigma\in\C\cup\{\infty\}$ of $A$, the differential system $Y' = AY$ has a basis of local solutions belonging to the Nilsson class at $\sigma$, i.e. which can be written as a finite sum
\begin{equation}
    g(z) = \sum_{e\in\C}\sum_{s\in\Z_{\geq 0}} g_{e, s}(z)(z-\sigma)^e \log(z-\sigma)^s
\end{equation}
where $g_{e, s}$ is holomorphic and does not vanish at $\sigma$. If $\sigma = \infty$, $``(z - \infty)"$ means ${1}/{z}$.

If $g$ is non-zero, its order at $\sigma$, denoted by $\mathrm{ord}_\sigma(g)$, is the smallest element of the finite set $\big\{\mathfrak{Re}(e) \;\;\big|\;\; \exists s\in\Z_{\geq 0} \;\;\; g_{e, s} \neq 0 \big\}$. In particular, logarithmic factors do not impact the order of $g$ at $\sigma$. In this context, the following inequality holds. A proof can be found in \cite[Theorem 3.1]{fischler2018}.

\begin{theorem}
\label{Shidlovskii}
    (``Shidlovskii's lemma").
    Let $\Sigma$ be a finite set of points of $\C\cup\{\infty\}$. For every $\sigma\in\Sigma$, let $(Y_j)_{j\in J_\sigma}$ be a finite family of solutions of $Y' = AY$ in the Nilsson class at $\sigma$ and such that the remainders $\big(\rho(Y_j)\big)_{j\in J_\sigma}$ are $\C$-linearly independent functions. 

    There exists a constant $c_1$ depending only on $A$ and $\Sigma$, such that for any differential operator $L\in \C(z)[\frac{d}{dz}]$ of order $\nu$ annihilating every $\rho(Y_j), j\in J_\sigma, \sigma\in\Sigma$, we have
    \begin{equation*}
        \sum_{\sigma\in\Sigma}\sum_{j\in J_\sigma} \mathrm{ord}_\sigma \rho(Y_j) \leq (\Delta+1)(\nu - \#J_\infty) + c_1.
    \end{equation*}
\end{theorem}

\begin{corollary}
\label{corollairedeShidlovskii}
    In this context, there exists a constant $c_1$ depending only on $A$ and $\Sigma$ such that
    \begin{equation*}
        \sum_{\sigma\in\Sigma}\sum_{j\in J_\sigma} \mathrm{ord}_\sigma \rho(Y_j) \leq (\Delta+1)(d - \#J_\infty) + c_1,
    \end{equation*}
    where $d$ is the size of the square matrix $A \in M_d\big(\C(z)\big)$.
\end{corollary}

\begin{proof}
    Because of \sref{Theorem}{Shidlovskii}, it is enough to find a differential operator $L \in \C(z)[\frac{d}{dz}]$ of order at most $d$ that annihilates $\rho(Y)$ for any solution of the differential system $Y' = AY$. Our method is based on Shidlovskii's ideas \cite[Chapter 3, §5]{shidlovskii89}.

    We view vectors as row matrices, and $\cdot$ denotes the matrix product. We define by induction the rational functions $S_{i, k}$:
    \begin{equation*}
        \begin{cases}
            (S_{1, 1}, ..., S_{d, 1}) = (S_1, ..., S_d), \\
            (S_{1, k+1}, ..., S_{d, k+1}) = (S_{1, k}, ..., S_{d, k})\cdot\left(A + \frac{d}{dz}I_d\right).
         \end{cases}
    \end{equation*}
    We thus have, for any solution $Y = \prescript{t}{}{(}y_1, ..., y_d)$ of $Y' = AY$ and for any $k\in\Z_{>0}$
    \begin{equation*}
        \rho(Y)^{(k-1)}(z) = \sum_{i = 1}^d S_{i, k}(z)y_i(z).
    \end{equation*}

    Now, consider the matrix $\mathfrak{S}: = (S_{i, k})_{i, k} \in M_{d, d+1}\big(\C(z)\big)$. Its $d+1$ columns are linearly dependent. We~thus find $\nu_1, ..., \nu_{d+1}\in\C(z)$, not all zero, such that $\mathfrak{S}\prescript{t}{}{(}\nu_1, ..., \nu_{d+1}) = 0$. Since $(y_1, ..., y_d)\mathfrak{S} = (\rho(Y), \rho(Y)',$ $..., \rho(Y)^{(d)})$ for any solution of $Y' = AY$, we deduce that the operator $L: = \sum_{k=1}^{d+1} \nu_k\left(\frac{d}{dz}\right)^{k-1}$ works.
\end{proof}

\subsection{A crucial proposition}
\label{paragraphestructureshidlovskii}

We fix $n\in\mathcal{N}$. The quantities $\Lambda_{n, (p, k)}$, $\;(p, k)\in\llbracket 0, h\rrbracket \times \llbracket 2rn+2, \kappa n\rrbracket$, defined in \sref{Subsection}{paragraphenotationpourlasuite} are linear combinations of the numbers $\zeta_0^{<0>}, ..., \zeta_0^{<M-1>}, \zeta_1, ..., \break\zeta_{a+h}$, whose coefficients $\lambda_{n, 0, (p, k)}^{<0>}, ..., \lambda_{n, 0, (p, k)}^{<M-1>}, \lambda_{n, 1, (p, k)}, ..., \lambda_{n, a+h, (p, k)}$ are defined by \eqref{definitiondespetitslambdas}. The rest of Section $6$ is dedicated to the proof of \sref{Proposition}{propositiondéduitedeshidlovskii} below, which will allow us to check assumption $(iii)$ of the linear independence criterion used in \sref{Section}{partie7}. 

\begin{proposition}
\label{propositiondéduitedeshidlovskii}
    Assume that $(h+1)(\kappa-2r)M + \omega > a$ and that $n\in\mathcal{N}$ is sufficiently large. Then for every $\bm{x} = \prescript{t}{}{(}x_0^{<0>}, ..., x_0^{<M-1>}, x_1, ..., x_{a+h}) \in \C^{a+h+M}$ such that
    \begin{small}
    \begin{equation}
    \label{THEequation}
        \forall (p, k)\in\llbracket 0, h\rrbracket \times \llbracket 2rn+2, \kappa n\rrbracket \;\;\; \sum_{m=0}^{M-1} \lambda_{n, 0, (p, k)}^{<m>} x_0^{<m>} + \sum_{i=1}^{a + h} \lambda_{n, i, (p, k)} x_i = 0,
    \end{equation}
    \end{small}
    we have for all $m\in\llbracket 0, M\rrbracket$ :
    \begin{equation}
    \label{définitiondespsim}
    \psi_m(\bm{x}) := x_0^{<m>} + (-1)^{\varepsilon} x_0^{<M-m>} = 0,
    \end{equation}     
    with the convention $x_0^{<M>} = - x_0^{<0>}.$
\end{proposition}

To prove this proposition, we adapt the strategy of \cite[Section 4.5]{fischler2026} to our current setting. In particular, the multiplicity estimate will be applied to more points. We first define the integer
\begin{equation}
\label{définitiondebn}
    b_n: = \max \Big\{ i \in \llbracket 1, a \rrbracket \;\;\Big|\;\; P_{n, i} \neq 0 \Big\},
\end{equation}
which exists because the polynomials $P_{n, i}$ are defined by \eqref{definitiondesPki} and the integers $c_{n, i, j}$ obtained from \sref{Proposition}{existencedescij} are not all zero. This integer satisfies $P_{n, b_n} \neq 0$ and $P_{n, i} = 0$ for $i > b_n$. Because of definitions \eqref{définitiondesQip}, \eqref{récurrencevérifiéeparlesQipk} and \eqref{definitiondespetitslambdas}, the polynomials $Q_{n, i, (p, k)}$ and the integers $\lambda_{n, i, (p, k)}$ all vanish for $b_n+p+1\leq i\leq a+h$. In particular, they vanish for $i > b_n + h$.

\smallskip
We consider, for $\;0\leq\ell\leq M-1$, the polynomials
\begin{small}
\begin{equation}
    \begin{split}
    \label{definitiondesQ0l}
        Q_{n, 0, (p, k)}^{\{\ell\}}(z): = \mu^{\ell(k-1)}Q_{n, 0, (p, k)}(\mu^\ell z), \\
        \Qbar_{n, 0, (p, k)}^{\{\ell\}}(z): = \mu^{\ell(k-1)}\Qbar_{n, 0, (p, k)}(\mu^\ell z).
    \end{split}
\end{equation}
\end{small}
Equation $(iii.b)$ of \sref{Lemma}{formulesd'inversiondeFourier} evaluated at $z=z_0$ yields
\begin{small}
\begin{equation*}
\label{fourieren-1}
    Q_{n, 0, (p, k)}^{<m>}(z_0) = \frac{z_0^{k-m-1}}{M}\sum_{\ell=0}^{M-1}\mu^{-\ell m} Q_{n, 0, (p, k)}^{\{\ell\}}(z_0), \hspace{1cm}0\leq m\leq M,
\end{equation*}
\end{small}
and the same holds for $0\leq m\leq M-1$ upon replacing $Q$ by $\Qbar$. Equation \eqref{definitiondespetitslambdas} thus becomes
\begin{small}
\begin{equation}
\label{expressiondespetitslambdasaveclesQ}
    \begin{split}
        &\lambda_{n, i, (p, k)} = \frac{\delta_{n, k}}{(k-1)!}z_0^{k-1}Q_{n, i, (p, k)}(z_0), \hspace{8.7cm}1\leq i\leq b_n+h,\\
        &\lambda_{n, 0, (p, k)}^{<m>} = -\frac{\delta_{n, k}}{(k-1)!}z_0^{k-1} \sum_{\ell = 0}^{M-1} \Big( \frac{(\mu^{\ell}z_0)^{m}}{M}Q_{n, 0, (p, k)}^{\{\ell\}}(z_0) + (-1)^{\varepsilon}\frac{(\mu^{\ell}z_0)^{-m}}{M}\Qbar_{n, 0, (p, k)}^{\{\ell\}}(z_0)\Big), \hspace{0.5cm} 0\leq m\leq M-1.
    \end{split}
\end{equation}
\end{small}

\smallskip
Assume, for a contradiction, that \sref{Proposition}{propositiondéduitedeshidlovskii} is false. Recall that by definition of $b_n$, all of the terms in the second sum with $i > b_n + h$ in \eqref{THEequation} vanish. Hence there exist $\bm{x}_n = \prescript{t}{}{(}x_{n, 0}^{<0>}, ..., x_{n, 0}^{<M-1>}, x_{n, 1}, ..., x_{n, b_n+h})\in\C^{b_n+h+M}$ and $m_0 \in \llbracket 0, M-1\rrbracket$\footnote{Notice that $\psi_M = -\psi_0$ because of the convention $x_0^{<M>} = -x_0^{<0>}$, so that it is equivalent to assume $m_0\in\llbracket 0, M-1\rrbracket$ or $m_0\in\llbracket 0, M\rrbracket$.} with
\begin{equation}
\label{hypothèsesurx}
    \forall (p, k)\in\llbracket 0, h\rrbracket \times \llbracket 2rn+2, \kappa n\rrbracket \;\;\; \sum_{m=0}^{M-1} \lambda_{n, 0, (p, k)}^{<m>} x_{n, 0}^{<m>} + \sum_{i=1}^{b_n + h} \lambda_{n, i, (p, k)} x_{n, i} = 0
\end{equation}
and
\begin{equation}
\label{hypothèsesurm0}
    \psi_{m_0}(\bm{x}_n) = x_{n, 0}^{<m_0>} + (-1)^\varepsilon x_{n, 0}^{<M-m_0>} \neq 0.
\end{equation}

We define
\begin{small}
\begin{equation}
\label{definitiondesxi}
\xi_{n, 0}^{\{\ell\}} := \sum_{m=0}^{M-1} \frac{(\mu^{\ell}z_0)^{m}}{M} x_{n, 0}^{<m>}, \hspace{1cm} \bar{\xi}_{n, 0}^{\{\ell\}} :=  (-1)^{\varepsilon} \sum_{m=0}^{M-1} \frac{(\mu^{\ell}z_0)^{-m}}{M} x_{n, 0}^{<m>}, \hspace{1cm} \xi_{n, i} := -x_{n, i}.
\end{equation}
\end{small}

Substituting $m \xleftarrow[]{} M-m$ in the definition of $\bar{\xi}_{n, 0}^{\{\ell\}}$, we get
\begin{equation*}
    \xi_{n, 0}^{\{\ell\}} - \bar{\xi}_{n, 0}^{\{\ell\}} = \sum_{m=0}^{M-1} \frac{(\mu^{\ell}z_0)^{m}}{M} \big(x_{n, 0}^{<m>} + (-1)^\varepsilon x_{n, 0}^{<M-m>}\big) = \sum_{m=0}^{M-1} \frac{(\mu^{\ell}z_0)^{m}}{M}\psi_m(\bm{x}_n).
\end{equation*}
Since the matrix $\Big((\mu^{\ell}z_0)^{m}\Big)_{0\leq \ell, m\leq M-1}$ is an invertible Vandermonde matrix, \eqref{hypothèsesurm0} implies the existence of some $\ell_0 \in \llbracket 0, M-1 \rrbracket$ with
\begin{equation}
\label{hypothèsesurl0}
    \xi_{n, 0}^{\{\ell_0\}} - \bar{\xi}_{n, 0}^{\{\ell_0\}} \neq 0.
\end{equation}

By using \eqref{expressiondespetitslambdasaveclesQ} and \eqref{hypothèsesurx}, we readily obtain 
\begin{small}
\begin{align}
\label{hypothesesurlesxini}
    \forall (p, k) \in \llbracket 0, h \rrbracket &\times \llbracket 2rn+2, \kappa n \rrbracket  \\
    &\sum_{\ell=0}^{M-1} \bigg(\xi_{n, 0}^{\{\ell\}}Q_{n, 0, (p, k)}^{\{\ell\}}(z_0) + \bar{\xi}_{n, 0}^{\{\ell\}}\Qbar_{n, 0, (p, k)}^{\{\ell\}}(z_0)\bigg) + \sum_{i=1}^{b_n+h}Q_{n, i, (p, k)}(z_0)\xi_{n, i} = 0. \nonumber
\end{align}
\end{small}


\smallskip
In \sref{Subsection}{constructiondesfp}, we use the coefficients $\xi_{n, 0}^{\{\ell\}}, \bar{\xi}_{n, 0}^{\{\ell\}}$ and $\xi_{n, i}$ to construct functions $f_{n, p}$ vanishing with a high order at $z = z_0$. They can be shifted to obtain functions $f_{n, p}(\mu^\ell z)$ with a high order of vanishing at $z = \mu^{-\ell}z_0$, $\;\;0\leq \ell\leq M-1$.

From these functions, we construct in \sref{Subsection}{constructiondesrhoq} new functions $\rho_{n, q}(\mu^\ell z)$ with the same order of vanishing, and that can be seen as remainders in the context of \sref{Subsection}{paragrapheénoncédeShidlovskii}.

We construct in \sref{Subsection}{constructiondestaubeta} other functions $\tau_{n, u}$ that can be seen as remainders for the same differential system, and with a high order of vanishing at $0, 1$ and $\infty$.

All of these remainders with high orders of vanishing allow us to obtain, for sufficiently large $n$, a contradiction in \sref{Subsection}{appplicationdeShidlovskii} when applying ``Shidlovskii's lemma" (\sref{Theorem}{Shidlovskii}). This completes the proof of \sref{Proposition}{propositiondéduitedeshidlovskii}.

\subsection{Construction of the functions $f_{n, p}$}
\label{constructiondesfp}

We consider an open and simply connected neighborhood $\mathscr{D}$ of $z=z_0$ which does not contain $0$, nor any $M$-th root of unity. For $\log(\mu^\ell z)$ to be well defined for any $z\in\mathscr{D}$ and $\ell\in\llbracket 0, M-1\rrbracket$, we also assume that $\mu^\ell\mathscr{D} \subset \mathcal{U}$ for $\ell\in\llbracket 0, M-1\rrbracket$, where the open set $\mathcal{U}$ is defined in \sref{Notation}{notationdudébut}. We define by induction holomorphic functions $g_{n, i}: \mathscr{D} \to \C$ as follows:
\begin{small}
\begin{equation*}
    \begin{cases}
        g_{n, 1}(z_0) = \xi_{n, 1}, &g'_{n, 1}(z) = \sum_{\ell=0}^{M-1} \frac{\xi_{n, 0}^{\{\ell\}}-\bar{\xi}_{n, 0}^{\{\ell\}}\mu^\ell z}{z(1-\mu^\ell z)}, \\
        g_{n, i}(z_0) = \xi_{n, i}, &g'_{n, i}(z) = \frac{-1}{z}g_{n, i-1}(z), \hspace{2cm} 2\leq i\leq b_n+h.
    \end{cases}
\end{equation*}
\end{small}

For $0 \leq p \leq h$, we then define
\begin{equation}
\label{définitiondesfp}
    f_{n, p}(z): = T_{n, p}(z) + \sum_{i=1}^{b_n+h} Q_{n, i, p}(z)g_{n, i}(z),
\end{equation}
where $T_{n, p}(z)$ is minus the Taylor polynomial of order $2rn$ of $\sum_{i=1}^{b_n+h} Q_{n, i, p}(z)g_{n, i}(z)$ at $z_0$, so that $\mathrm{ord}_{z_0}(f_{n, p}) \geq 2rn+1$.

It turns out that the order of vanishing of $f_{n, p}$ at $z_0$ is much larger than this. Indeed, the substitution $z \xleftarrow[]{} \mu^\ell z$ in \eqref{récurrencevérifiéeparlesQipk}, along with equations \eqref{definitiondesQ0l} and \eqref{lesQnipksontdansQzN}, yields
\begin{small}
\begin{equation*}
    \begin{cases}
        Q_{n, 0, (p, k+1)}^{\{\ell\}}(z) = (Q_{n, 0, (p, k)}^{\{\ell\}})'(z) + \frac{1}{z(1-\mu^{\ell} z)}Q_{n, 1, (p, k)}, &Q_{n, 0, (p, 1)}^{\{\ell\}} = 0, \\
        \Qbar_{n, 0, (p, k+1)}^{\{\ell\}}(z) = (\Qbar_{n, 0, (p, k)}^{\{\ell\}})'(z) - \frac{\mu^\ell}{1-\mu^{\ell} z}Q_{n, 1, (p, k)}, &\Qbar_{n, 0, (p, 1)}^{\{\ell\}} = 0.
    \end{cases}
\end{equation*}
\end{small}
Thus, by \eqref{définitiondesfp} and the definition of the functions $g_{n, i}$, one can show by induction on $k \geq 1$: 
\begin{equation*}
    f_{n, p}^{(k-1)}(z) = T_{n, p}^{(k-1)}(z) + \sum_{\ell=0}^{M-1} \bigg(\xi_{n, 0}^{\{\ell\}}Q_{n, 0, (p, k)}^{\{\ell\}}(z) + \bar{\xi}_{n, 0}^{\{\ell\}}\Qbar_{n, 0, (p, k)}^{\{\ell\}}(z)\bigg) + \sum_{i=1}^{b_n+h}Q_{n, i, (p, k)}(z)g_{n, i}(z).
\end{equation*}
For $2rn + 2 \leq k \leq \kappa n$, the polynomial $T_{n, p}^{(k-1)}$ vanishes because $\mathrm{deg}\big(T_{n, p}\big) \leq 2rn$. It follows that $f_{n, p}^{(k-1)}(z_0) = 0$ because of \eqref{hypothesesurlesxini}, since $g_{n, i}(z_0) = \xi_{n, i}$. Thus $\mathrm{ord}_{z_0}(f_{n, p}) \geq \kappa n, \;\; 0\leq p\leq h$.

The functions $f_{n, p}$ can be seen in the context of \sref{Subsection}{paragrapheénoncédeShidlovskii} as the remainders associated with the solution $\prescript{t}{}{(}1, g_{n, 1}, ..., g_{n, b_n+h})$ of the differential system
\begin{small}
\begin{equation}
\label{premiersystemediff}
    Y' = \begin{bmatrix}
        0 &0 &0 &\hdots &0 &0 \\
        \sum_{\ell=0}^{M-1} \frac{\xi_{n, 0}^{\{\ell\}}-\bar{\xi}_{n, 0}^{\{\ell\}}\mu^\ell z}{z(1-\mu^\ell z)} &0 &0 &\hdots &0 &0 \\
        0 &\frac{-1}{z} &0 &\hdots &0 &0 \\
        0 &0 &\frac{-1}{z} &\hdots &0 &0 \\
        \vdots &\vdots &\vdots &\ddots &\vdots &\vdots \\
        0 &0 &0 &\hdots &\frac{-1}{z} &0 
    \end{bmatrix}Y
\end{equation}
\end{small}
weighted by the polynomials $T_{n, p}, Q_{n, 1, p}, ..., Q_{n, b_n+h, p}$. This is not suitable for applying \sref{Theorem}{Shidlovskii}: we have only one solution $(g_{n, i})_i$ weighted by different families $(Q_{n, i, p})_i$ depending on $p$ to obtain the remainders $f_{n, p}$. We should instead have different solutions $(y_{n, i, q})_i$, weighted by only one family $(S_{n, i})_i$ independent of $q$ to obtain some remainders $\rho_{n, q}$.

\smallskip
We will construct all of these objects related to the differential system \eqref{grossystemediffcompliqué} in \sref{Subsection}{constructiondesrhoq}. Before that, notice that we can easily obtain $M$ times as many functions with high orders of vanishing by considering the substitutions $z \xleftarrow[]{} \mu^\ell z$. Indeed, we have
\begin{equation}
\label{annulationdesfp}
    \forall p\in\llbracket 0, h \rrbracket, \;\forall \ell \in \llbracket 0, M-1 \rrbracket, \hspace{0.8cm} \mathrm{ord}_{\mu^{-\ell}z_0}\Big(f_{n, p}(\mu^\ell z)\Big) \geq \kappa n.
\end{equation}
The functions $g_{n, i}(\mu^\ell z)$ obey the same differentiation rules as $g_{n, i}(z)$, except for $i = 1$. Notice that by adding $M-1$ rows of zeroes at the top of the matrix in \eqref{premiersystemediff} and then replacing its first column by the $M$ columns
\begin{equation*}
    \begin{bmatrix}
        0 \\
        \vdots \\
        0\\
        \sum_{m=0}^{M-1} \frac{\xi_{n, 0}^{\{m\}}-\bar{\xi}_{n, 0}^{\{m\}}\mu^{m+\ell} z}{z(1-\mu^{m+\ell} z)} \\
        0 \\ 
        \vdots \\
        0
    \end{bmatrix}, \hspace{2cm} 0\leq \ell\leq M-1,
\end{equation*}
\noindent we obtain a differential system that admits $\prescript{t}{}{\big(}0, ..., 0, 1, 0, ..., 0, g_{n, 1}(\mu^\ell z), ..., g_{n, b_n+h}(\mu^\ell z)\big)$ as solutions, $0\leq \ell \leq M-1$, with the $1$ on the $(\ell+1)$-th slot.
The functions $f_{n, p}(\mu^\ell z)$ can be seen as the remainders associated with these solutions, weighted by the polynomials $T_{n, p}(z), T_{n, p}(\mu z), ...,$ $T_{n, p}(\mu^{M-1} z), Q_{n, 1, p}, ..., Q_{n, b_n+h, p}$.

\subsection{Construction of the remainders $\rho_{n, q}$}
\label{constructiondesrhoq}

As in \cite[Section 4.5, Step 2]{fischler2026}, we define the functions
\begin{equation}
\label{definitiondesrhoq}
    \rho_{n, q}(z): = \sum_{p=0}^q \binom{q}{p} \Big(-\log(z)\Big)^{q-p}f_{n, p}(z), \hspace{2cm} 0\leq q\leq h.
\end{equation}
This subsection introduces a new differential system, vectors of solutions and polynomials such that $\rho_{n, q}$ can be seen as remainders. 

Let $0_{u, v}$, resp. $0_u$, denote the zero matrix of $M_{u, v}(\C)$, resp. $M_u(\C)$. We denote by $J_u$ the matrix of $M_u(\C)$ with $1$ on the sub-diagonal and $0$ everywhere else. We consider the differential system $Y' = \mathscr{A}_n Y$ with
\begin{small}
\begin{equation}
    \label{grossystemediffcompliqué}
    \mathscr{A}_n: =
        \begin{bmatrix}
            0_M &0_M &0_M &\hdots &0_M &0_M &0_{M, b_n} \\
            \frac{-1}{z}I_M &0_M &0_M &\hdots &0_M &0_M &0_{M, b_n} \\
            0_M &\frac{-1}{z}I_M &0_M &\hdots &0_M &0_M &0_{M, b_n} \\
            0_M &0_M &\frac{-1}{z}I_M &\hdots &0_M &0_M &0_{M, b_n} \\
            \vdots &\vdots &\vdots &\ddots &\vdots &\vdots &\vdots \\
            0_M &0_M &0_M &\hdots &\frac{-1}{z}I_M &0_M &0_{M, b_n} \\
            0_{b_n, M} &0_{b_n, M} &0_{b_n, M} &\hdots &0_{b_n, M} &\mathscr{B}_n &\frac{-1}{z}J_{b_n} 
        \end{bmatrix} \in M_{M(h+1)+b_n}\big(\C(z)\big).
    \end{equation}
\end{small}
Notice that the matrix $\mathscr{A}_n$ only depends on $n$ through the integer $b_n$ and the block $\mathscr{B}_n\in M_{b_n, M}\big(\C(z)\big)$ whose columns are
\begin{small}
\begin{equation*}
    \begin{bmatrix}
        \sum_{m=0}^{M-1} \frac{\xi_{n, 0}^{\{m\}}-\bar{\xi}_{n, 0}^{\{m\}}\mu^{m+\ell} z}{z(1-\mu^{m+\ell} z)} \\
        0 \\
        \vdots \\
        0
    \end{bmatrix}, \hspace{2cm}{0\leq \ell\leq M-1}.
\end{equation*}
\end{small}

We also consider the polynomials
\begin{equation*}
    \begin{cases}
        S_{n, i}(z) = \frac{1}{(h-i)!} T_{n, h-i}(z), &0\leq i\leq h, \\
        S_{n, i}(z) = z^{rn}P_{n, i-h}(z), &h+1\leq i\leq b_n+h,
    \end{cases}
\end{equation*}
and for $0\leq q\leq h$ the functions
\begin{equation*}
    \begin{cases}
        y_{n, i, q}(z) = 0, &\hspace{1cm}0\leq i\leq h-q-1, \\
        y_{n, i, q}(z) = \frac{q!}{(i+q-h)!}\Big(-\log(z)\Big)^{i+q-h}, &\hspace{1cm}h-q\leq i\leq h, \\
        y_{n, i, q}(z) = \sum_{p=0}^q \binom{q}{p} (-1)^p(i-h)_p\Big(-\log(z)\Big)^{q-p}g_{n, i-h+p}(z), &\hspace{1cm}h+1\leq i\leq b_n+h.
    \end{cases}
\end{equation*}
We organize them inside column vectors by defining
\begin{align*}
    \mathcal{S}_n: = \prescript{t}{}{\Big(}S_{n,0}(z), S_{n,0}(\mu z), ..., S_{n,0}(\mu^{M-1}z), ...\;...,  S_{n,h}(z), &S_{n,h}(\mu z), ..., S_{n,h}(\mu^{M-1}z), \\
    &S_{n, h+1}(z), S_{n, h+2}(z), ..., S_{n, b_n+h}(z)\Big)
\end{align*}
and, for $0\leq q\leq h$ and $0\leq\ell\leq M-1$,
\begin{align*}
        Y_{n, q}^{\{\ell\}}(z): = \prescript{t}{}{\Big(}0, \;\;..., \;\;y_{n, 0, q}(\mu^\ell z), \;\;..., \;\;0, \;\;...\;\;..., \;\;0, \;\;..., \;\;&y_{n, h, q}(\mu^\ell z), \;\;..., \;\;0, \\
        &y_{n, h+1, q}(\mu^\ell z), ..., y_{n, b_n+h, q}(\mu^\ell z)\Big).
\end{align*}
In the definition of $Y_{n, q}^{\{\ell\}}(z)$, up to the index $i=h$, each block has length $M$ and contains $M-1$ zeroes and $y_{n, i, q}(\mu^\ell z)$ at the $(\ell+1)$-th slot. Starting from the index $i=h+1$, each block has length $1$ and only contains $y_{n, i, q}(\mu^\ell z)$.

Notice that because of the definition \eqref{definitiondesPki} of the polynomials $P_{n, i}$, the polynomials $S_{n, i}$ with $i \geq h+1$ satisfy $S_{n, i}(\mu^\ell z) = S_{n, i}(z)$ for every $\ell\in\llbracket 0, M-1\rrbracket$.
\begin{lemma}
    In this context, 
    \begin{enumerate}
        \item[$(i)$] for $0\leq q\leq h$ and $0\leq\ell\leq M-1$, $Y_{n, q}^{\{\ell\}}$ is a solution of the differential system $Y' = \mathscr{A}_n Y$;
        \item[$(ii)$] the functions $\rho_{n, q}(\mu^\ell z)$ are the remainders associated with these solutions and weighted by the polynomials $\mathcal{S}_n$: 
        \begin{equation*}
            \forall q\in\llbracket 0, h\rrbracket \;\;\forall\ell\in\llbracket 0, M-1\rrbracket \;\;\;\;\;\;\rho_{n, q}(\mu^\ell z) = \sum_{i=0}^{b_n+h} S_{n, i}(\mu^\ell z) y_{n, i, q}(\mu^\ell z) = \prescript{t}{}{\mathcal{S}_n}\cdot Y_{n, q}^{\{\ell\}};
        \end{equation*}
        \item[$(iii)$] for $\ell\in\llbracket 0, M-1 \rrbracket$, the family $\Big(\rho_{n, q}(\mu^\ell z)\Big)_{0\leq q\leq h}$ is linearly independent over $\C$.
    \end{enumerate}
\end{lemma}

\begin{proof}
    $(i)\;$We fix $q\in\llbracket 0, h\rrbracket$ and $\ell\in\llbracket 0, M-1\rrbracket.$
    \begin{enumerate}
        \item[$\bullet$] If $i=0$, $\;y_{n, 0, q}(\mu^\ell z)$ is constant, and its derivative is zero.
        \item[$\bullet$] If $1\leq i\leq h-q$, $\;y_{n, i, q}(\mu^\ell z)$ is constant and its derivative is zero, and $\frac{-1}{z}y_{n, i-1, q}(\mu^\ell z)$ is zero as well.
        \item[$\bullet$] If $h-q+1\leq i\leq h$, we have
            \begin{align*}
                \frac{d}{dz}\Big[y_{n, i, q}(\mu^\ell z)\Big] &= \frac{q!}{(i+q-h)!}(i+q-h)\frac{-\mu^\ell}{\mu^\ell z}\Big(-\log(\mu^\ell z)\Big)^{i+q-h-1} \\
                                                        &= \frac{-1}{z}\frac{q!}{(i+q-h-1)!}\Big(-\log(\mu^\ell z)\Big)^{i+q-h-1} 
                                                    = \frac{-1}{z}y_{n, i-1, q}(\mu^\ell z).
            \end{align*}
        \item[$\bullet$] If $i=h+1$, we have
        \begin{small}
            \begin{align*}
                \frac{d}{dz}\Big[y_{n, h+1, q}(\mu^\ell z)\Big] &= \Big(-\log(\mu^\ell z)\Big)^q \Bigg(\sum_{m=0}^{M-1} \frac{\xi_{n, 0}^{\{m\}}-\bar{\xi}_{n, 0}^{\{m\}}\mu^{m+\ell}z}{z(1-\mu^{m+\ell} z)} \Bigg) \\
                &\hspace{5mm}+ \sum_{p=1}^q \binom{q}{p} (-1)^p p!\Big(-\log(\mu^\ell z)\Big)^{q-p} \cdot\frac{-1}{z}g_{n, p}(\mu^\ell z) \\
                &\hspace{5mm}+ \sum_{p=0}^{q-1} \binom{q}{p}(-1)^p p! (q-p) \frac{-1}{z}\Big(-\log(\mu^\ell z)\Big)^{q-p-1}g_{n, p+1}(\mu^\ell z) \\
                &= \Bigg(\sum_{m=0}^{M-1} \frac{\xi_{n, 0}^{\{m\}}-\bar{\xi}_{n, 0}^{\{m\}}\mu^{m+\ell}z}{z(1-\mu^{m+\ell} z)} \Bigg) y_{n, h, q}(\mu^\ell z).
            \end{align*}
            \end{small}
            Indeed, the two sums over $p$ cancel each other out, as can be seen by applying the substitution $p\xleftarrow[]{} p+1$ in the second sum.
            
        \item[$\bullet$] Lastly, if $h+2\leq i\leq b_n+h$, we have
        \begin{small}
            \begin{align*}
                \frac{d}{dz}\Big[y_{n, i, q}(\mu^\ell z)\Big] &= \sum_{p=0}^q \binom{q}{p}(-1)^p(i-h)_p\Big(-\log(\mu^\ell z)\Big)^{q-p}\cdot\frac{-1}{z}g_{n, i-h+p-1}(\mu^\ell z) \\
                &\hspace{5mm}+\sum_{p=0}^{q-1} \binom{q}{p} (-1)^p (i-h)_p (q-p)\frac{-1}{z}\Big(-\log(\mu^\ell z)\Big)^{q-p-1}g_{n, i-h+p}(\mu^\ell z) \\
                &= \frac{-1}{z}\Big(-\log(\mu^\ell z)\Big)^qg_{n, i-h-1}(\mu^\ell z) \\
                &\hspace{5mm}+\frac{-1}{z}\sum_{p=1}^q \frac{q!}{p!(q-p)!}(-1)^p(i-h)_p\Big(-\log(\mu^\ell z)\Big)^{q-p}g_{n, i-h+p-1}(\mu^\ell z) \\
                &\hspace{5mm}+\frac{-1}{z}\sum_{p=1}^q \frac{q!}{(p-1)!(q-p)!}(-1)^{p-1}(i-h)_{p-1}\Big(-\log(\mu^\ell z)\Big)^{q-p} g_{n, i-h+p-1}(\mu^\ell z) \\
                &=\frac{-1}{z}\Big(-\log(\mu^\ell z)\Big)^qg_{n, i-h-1}(\mu^\ell z) \\
                &\hspace{5mm}+\frac{-1}{z}\sum_{p=1}^q \frac{q!}{(q-p)!}(-1)^p\bigg(\frac{(i-h)_p}{p!}-\frac{(i-h)_{p-1}}{(p-1)!}\bigg)\Big(-\log(\mu^\ell z)\Big)^{q-p}g_{n, i-h+p-1}(\mu^\ell z) \\
                &= \frac{-1}{z}\sum_{p=0}^q \binom{q}{p}(-1)^p(i-h-1)_p\Big(-\log(\mu^\ell z)\Big)^{q-p}g_{n, i-h+p-1}(\mu^\ell z) \\
                &= \frac{-1}{z}y_{n, i-1, q}(\mu^\ell z),
            \end{align*}
            \end{small}
            by using the formula $\frac{(i-h)_p}{p!}-\frac{(i-h)_{p-1}}{(p-1)!} = \frac{(i-h-1)_p}{p!}$.
    \end{enumerate}

    $(ii)\;$ We fix $\ell\in\llbracket 0, M-1\rrbracket$. Recall that for $i\geq 1$, the polynomials $Q_{n, i, p}$ defined by \eqref{définitiondesQip} vanish when $i\notin\llbracket p+1, b_n+p\rrbracket$ and are in $\Q[z^M]$ otherwise, so that $Q_{n, i, p}(\mu^\ell z) = Q_{n, i, p}(z)$. We use equations \eqref{définitiondesfp} and \eqref{definitiondesrhoq} to compute: 

    \begin{small}
    \begin{align*}
        \rho_{n, q}(\mu^\ell z) &= \sum_{p=0}^q \binom{q}{p} \big(-\log(\mu^\ell z)\big)^{q-p} \Bigg[T_{n, p}(\mu^\ell z) + \sum_{i=p+1}^{b_n+p} Q_{n, i, p}(z)g_{n, i}(\mu^\ell z) \Bigg] \\
                            &= \sum_{p=0}^q \frac{1}{p!}T_{n, p}(\mu^\ell z) \frac{q!}{(q-p)!}\big(-\log(\mu^\ell z)\big)^{q-p} \\
                            &\hspace{0.1cm}+ \sum_{p=0}^q \binom{q}{p} \big(-\log(\mu^\ell z)\big)^{q-p} \sum_{i=h+1}^{b_n+h} Q_{n, i-h+p, p}(z)g_{n, i-h+p}(\mu^\ell z) \hspace{21mm}(i \xleftarrow[]{} i+h-p) \\
                            &= \sum_{i=h-q}^{h} \frac{1}{(h-i)!}T_{n, h-i}(\mu^\ell z) \frac{q!}{(i+q-h)!} \big(-\log(\mu^\ell z)\big)^{i+q-h} \hspace{32mm}(i = h - p) \\
                            &\hspace{0.1cm}+ \sum_{i = h+1}^{b_n+h} z^{rn}P_{n, i-h}(z)\sum_{p=0}^q \binom{q}{p} (-1)^p(i-h)_p \big(-\log(\mu^\ell z)\big)^{q-p}g_{n, i-h+p}(\mu^\ell z) \hspace{3mm}\text{\big(from \eqref{définitiondesQip}\big)} \\
                            &= \sum_{i=0}^{b_n+h} S_{n, i}(\mu^\ell z) y_{n, i, q}(\mu^\ell z) = \prescript{t}{}{\mathcal{S}_n}\cdot Y_{n, q}^{\{\ell\}}.
    \end{align*}
    \end{small}

    $(iii)\;$ We fix $\ell\in\llbracket 0, M-1 \rrbracket$ and assume for a contradiction that there exist $\nu_0, ..., \nu_h \in\C$ not all zero such that $\sum_{q=0}^h \nu_q\rho_{n, q}(\mu^\ell z) = 0$. We use the result of $(ii)$ above and swap the sums to obtain
    \begin{equation*}
        \sum_{i=0}^{b_n+h} S_{n, i}(\mu^\ell z)\sum_{q=0}^h \nu_q y_{n, i, q}(\mu^\ell z) = 0.
    \end{equation*}
    Consider the functions $y_{n, i}(z): = \sum_{q=0}^h \nu_q y_{n, i, q}(z)$. On the one hand, they satisfy
    \begin{equation}
    \label{initialisationdelarécurrence}
        \sum_{i=0}^{b_n+h} S_{n, i}(\mu^\ell z)y_{n, i}(\mu^\ell z) = 0.
    \end{equation}
    On the other hand, they are the coordinates of the vector $Y_{n}^{\{\ell\}}(z)$ defined by
    \begin{small}
    \begin{align*}
        \sum_{q=0}^h \nu_q Y_{n, q}^{\{\ell\}}(z) = \prescript{t}{}{\Big(}0, ..., y_{n, 0}(\mu^\ell z), ..., 0, ...\;..., 0, ..., y_{n, h}(\mu^\ell z), ..., 0, y_{n, h+1}(\mu^\ell z), ..., y_{n, b_n+h}(\mu^\ell z)\Big),
    \end{align*}
    \end{small}
    which is still a solution of the linear differential system $Y' = \mathscr{A}_n Y$.

    \smallskip
   We denote by $q_{\max}$ the largest index $q$ with $\nu_q \neq 0$, so that
   \begin{align*}
       y_{n, h-q_{\max}}(z) = \sum_{q=0}^{q_{\max}} \nu_q y_{n, h-q_{\max}, q}(z) &= \nu_{q_{\max}} y_{n, h-q_{\max}, q_{\max}}(z) \hspace{1.5cm}\text{(since $y_{n, i, q} = 0$ if $i<h-q$)} \\
                            &= \nu_{q_{\max}} q_{\max} ! \neq 0.
   \end{align*}
   Let $i_{\min} = h-q_{\max}$, so that $y_{n, i_{\min}}(z)$ is a non-zero constant and $y_{n, i} = 0$ for $i<i_{\min}$.
    We prove by decreasing induction on $i_{\max}\in\llbracket i_{\min}, b_n+h\rrbracket$ the following fact: 
    \begin{equation*}
        H_{i_{\max}}: \hspace{5mm}\text{\guillemotleft \hspace{2mm}There exist polynomials $R_{i_{\max}, i} \in \C[z], \;i_{\min}\leq i\leq i_{\max},\;$ such that $R_{i_{\max}, i_{\max}} \neq 0$} 
    \end{equation*}
    \begin{equation*}
        \text{and $\sum_{i=i_{\min}}^{i_{\max}} R_{i_{\max}, i}(z) y_{n, i}(\mu^\ell z) = 0$}\hspace{8mm} (*)\hspace{2mm}\text{\guillemotright}.
    \end{equation*}
    According to \eqref{initialisationdelarécurrence}, this holds for $i_{\max} = b_n+h$ by letting $R_{b_n+h, i}(z) = S_{n, i}(\mu^\ell z)$. The polynomial $R_{b_n+h, b_n+h} = S_{n, b_n+h} = z^{rn}P_{n, b_n}$ is indeed non-zero, because of definition \eqref{définitiondebn} of $b_n$.
    Now, assume that $H_{i_{\max}}$ holds for some index $i_{\max}\in\llbracket i_{\min} + 1, b_n+h\rrbracket$. Let $v$ be the degree of $R_{i_{\max}, i_{\max}}$. Differentiating equation $(*)$ $v+1$ times, we get
    \begin{equation*}
        \sum_{i=i_{\min}}^{i_{\max}} \widetilde{R}_{i_{\max}-1, i}(z)y_{n, i}(\mu^\ell z) = 0,
    \end{equation*}
    where the expression of the rational functions $\widetilde{R}_{i_{\max}-1, i}$ can be computed by using the Leibniz rule and the fact that $Y_n^{\{\ell\}}$ is a solution of $Y'=\mathscr{A}_nY$, where $\mathscr{A}_n$ is given by \eqref{grossystemediffcompliqué}. In particular, we notice that
    
    \begin{enumerate}[left=0.5pt]
        \item[$\bullet$] for $i_{\min}\leq i\leq i_{\max}$, the only possible poles for $\widetilde{R}_{i_{\max}-1, i}$ are $0, 1, \mu, \mu^2, ..., \mu^{M-1}$, with multiplicity at most $v+1$ each.
        \item[$\bullet$] $\widetilde{R}_{i_{\max}-1, i_{\max}} = R_{i_{\max}, i_{\max}}^{(v+1)} = 0.$
        \item[$\bullet$] 
        $\widetilde{R}_{i_{\max}-1, i_{\max}-1} =$ 
        \begin{small}
        \begin{equation*}
            \begin{cases}
                R_{i_{\max}, i_{\max}-1}^{(v+1)} + \sum_{u=0}^v \left(\frac{d}{dz}\right)^{v-u} \left(\frac{-1}{z}R_{i_{\max}, i_{\max}}^{(u)}\right) &\text{if $i_{\max}\neq h+1$}, \\
                R_{i_{\max}, i_{\max}-1}^{(v+1)} + \sum_{u=0}^v \left(\frac{d}{dz}\right)^{v-u} \Bigg(\bigg(\sum_{m=0}^{M-1}\frac{\xi_{n, 0}^{\{m\}}-\bar{\xi}_{n, 0}^{\{m\}}\mu^{m+\ell}z}{z(1-\mu^{m+\ell}z)} \bigg)R_{i_{\max}, i_{\max}}^{(u)}\Bigg) &\text{if $i_{\max} = h+1$}.
            \end{cases}
        \end{equation*}
        \end{small}
    \end{enumerate}
    This last rational function is non zero. To prove it, it is enough to find a point in the complex plane where the term with $u = v$ in the sum has a non-zero residue. Indeed, no other term will be able to cancel this residue out, because polynomials and derivatives of rational functions have zero residue.
    \begin{enumerate}
        \item[\textrightarrow] In the case $i_{\max} \neq h+1$, the term in the sum with $u = v$ has residue $-R_{i_{\max}, i_{\max}}^{(v)}$ at $z=0$. This is a non-zero constant because of $H_{i_{\max}}$, since $v = \mathrm{deg}(R_{i_{\max}, i_{\max}})$.
        \item[\textrightarrow] In the case $i_{\max} = h + 1$, the term in the sum with $u = v$ has residue $-(\xi_{n, 0}^{ \{ \ell_0 \} }  - \bar{\xi}_{n, 0}^{ \{ \ell_0 \} })\break\times R_{i_{\max}, i_{\max}}^{(v)}$ at $z=\mu^{-(\ell+\ell_0)}$, which is non-zero by \eqref{hypothèsesurl0}.
    \end{enumerate}
     We thus obtain suitable polynomials for $H_{i_{\max}-1}$ by letting, for $i_{\min}\leq i\leq i_{\max}-1$: 
    \begin{equation*}
        R_{i_{\max}-1, i}: = z^{v+1}\prod_{m=0}^{M-1}(1-\mu^m z)^{v+1} \widetilde{R}_{i_{\max}-1, i}.
    \end{equation*}
    This concludes the inductive proof.

    \smallskip
    Now, $H_{i_{\min}}$ provides a non-zero polynomial $R_{i_{\min}, i_{\min}}$ such that $R_{i_{\min}, i_{\min}}(z)y_{n, i_{\min}} = 0$, which is a contradiction because $y_{n, i_{\min}}$ is a non-zero constant, by definition of $i_{\min}$.
    That means there cannot exist coefficients $\nu_0, ..., \nu_h\in\C$ not all zero and such that $\sum_{q=0}^h \nu_q\rho_{n, q}(\mu^\ell z) = 0$.
\end{proof}

\subsection{Construction of the remainders $\tau_{n, u}$}
\label{constructiondestaubeta}

We consider for $1\leq u\leq b_n$ the solutions $W_{n, u} = \prescript{t}{}{(}0, ..., 0, w_{1, u}, ..., w_{b_n, u})$ of the differential system $Y' = \mathscr{A}_n Y$, where only the last $b_n$ coordinates can be non-zero, and are given by
\begin{equation*}
    \begin{cases}
        w_{i, u}(z) = 0, &1\leq i \leq u -1, \\
        w_{i, u}(z) = \frac{\big(-\log(z)\big)^{i-u}}{(i-u)!}, &u \leq i \leq b_n.
    \end{cases}
\end{equation*}

The associated remainders are
\begin{equation}
\label{expressiondestaubeta}
    \tau_{n, u}(z): = \prescript{t}{}{\mathcal{S}_n}\cdot W_{n, u} = \sum_{i = u}^{b_n} z^{rn}P_{n, i}(z)\frac{\big(-\log(z)\big)^{i-u}}{(i-u)!}. 
\end{equation}

\begin{lemma}
    The family $\big(\tau_{n, u}(z)\big)_{1\leq u\leq b_n}$ is linearly independent over $\C$.
\end{lemma}

\begin{proof}
    Assume for a contradiction that there exist $\nu_1, ..., \nu_{b_n}\in\C$ not all zero such that \\ $\sum_{u=1}^{b_n} \nu_u \tau_{n, u} (z) = 0$. Let $u_{\min}$ be the smallest index in $\llbracket 1, b_n\rrbracket$ such that $\nu_{u_{\min}} \neq 0$.

    We then have the following $\C[z]$-linear relation among the powers of $\log(z)$: 
    \begin{equation*}
        \sum_{u=u_{\min}}^{b_n} \sum_{i=u}^{b_n} \nu_u z^{rn}P_{n, i}(z)\frac{\Big(-\log(z)\Big)^{i-u}}{(i-u)!} = 0.
    \end{equation*}
    Since $\log(z)$ is transcendental over $\C[z]$, the coefficient of $\Big(-\log(z)\Big)^{b_n-u_{\min}}$ in this relation has to be zero: 
    \begin{equation*}
        \frac{\nu_{u_{\min}}}{(b_n-u_{\min})!}z^{rn}P_{n, b_n}(z) = 0.
    \end{equation*}
    This is a contradiction, since $P_{n, b_n} \neq 0$ by the definition \eqref{définitiondebn} of $b_n$, and $\nu_{u_{\min}} \neq 0$ by definition of $u_{\min}$. 
\end{proof}

\subsection{Application of ``Shidlovskii's lemma" and contradiction}
\label{appplicationdeShidlovskii}

We will apply \sref{Corollary}{corollairedeShidlovskii} with $d = M(h+1) + b_n$ the size of the matrix $\mathscr{A}_n$ and $\Delta = 2rn$ bounding the degrees of the polynomials $S_{n, i}$. 

We consider the finite set of points $\Sigma = \{\mu^{-\ell}z_0 \;|\; 0\leq\ell\leq M-1\}\cup\{0, 1, \infty\}$. Since $z_0$ is an $M$-th root of $-1$, $\Sigma$ contains $M+3$ distinct points, and all the points $\mu^{-\ell}z_0$ are regular points of \eqref{grossystemediffcompliqué}, whose finite singular points are $0$ and the $M$-th roots of unity. Recall that the notion of order was defined in \sref{Subsection}{paragrapheénoncédeShidlovskii}, and that logarithmic factors have no influence on it. From equations \eqref{annulationdesfp} and \eqref{definitiondesrhoq}, we get $\;\;\forall q\in\llbracket 0, h \rrbracket, \;\;\;\forall \ell\in\llbracket 0, M-1\rrbracket, \hspace{0.3cm} \mathrm{ord}_{\mu^{-\ell}z_0}\Big(\rho_{n, q}(\mu^\ell z)\Big) \geq \kappa n$.

From equation \eqref{expressiondestaubeta}, we deduce $\forall u\in\llbracket 1, b_n\rrbracket, \hspace{0.5cm} \mathrm{ord}_0\Big(\tau_{n, u}(z)\Big) \geq rn$.

To get a lower bound on the order of $\tau_{n, u}$ at $\infty$, let us rewrite \eqref{expressiondestaubeta} as
\begin{small}
\begin{equation*}
    \tau_{n, u}(z) = \sum_{i=u}^{b_n} \left(\frac{1}{z}\right)^{-(r+1)n}\frac{1}{z^n}P_{n, i}(z)\frac{\log\left(\frac{1}{z}\right)^{i-u}}{(i-u)!}.
\end{equation*}
\end{small}
Since $\deg P_{n, i} \leq n$, the Laurent polynomials $\frac{1}{z^n}P_{n, i}(z)$ only have non-positive powers of $z$, and are thus holomorphic at infinity. We deduce $\forall u\in\llbracket 1, b_n\rrbracket, \hspace{0.5cm} \mathrm{ord}_\infty\Big(\tau_{n, u}(z)\Big) \geq -(r+1)n$.

Lastly, we deduce from \eqref{définitiondesRn} and \eqref{expressiondestaubeta} that $\tau_{n, 1}(z) = z^{rn}R_n(z)$. From the chain of equivalences at the end of \sref{Subsection}{souspartie3.1} and \sref{Proposition}{existencedescij} $(i)$, it follows that $\mathrm{ord}_1\Big(\tau_{n, 1}(z)\Big) \geq \omega n - 1$.

Thus, \sref{Corollary}{corollairedeShidlovskii} provides a constant $c_{1}$ depending only on $\Sigma$ and $\mathscr{A}_n$, such that
\begin{equation*}
    (h+1)M\kappa n + b_n rn - b_n (r+1)n + \omega n - 1 \leq (2rn + 1)\big((h+1)M+b_n-b_n\big) + c_{1},
\end{equation*}
that is
\begin{equation}
\label{nassezgrand2}
    \Big((h+1)(\kappa-2r)M -b_n+\omega\Big)n \leq (h+1)M+c_{1}+1.
\end{equation}

As explained in the remark below, $c_1$ is bounded independently of $n$. Since $\Sigma$ is fixed and $b_n \leq a$ takes finitely many values, we can thus assert that for $n$ large enough, the term multiplying $n$ on the left-hand side of this inequality cannot be positive. That means $(h+1)(\kappa -2r)M + \omega \leq b_n \leq a$.

This contradicts the assumption $(h+1)(\kappa-2r)M + \omega > a$, and the proof of \sref{Proposition}{propositiondéduitedeshidlovskii} is complete.

\begin{remark*}
    The proof of \sref{Theorem}{Shidlovskii} given by \cite[Theorem 3.1]{fischler2018} yields 
        \begin{equation*}
            c_{1} \leq \sum_{\sigma\in\Sigma}\left((\kappa_\sigma + 1)\frac{\nu(\nu-1)}{2}-(\nu-\#J_\sigma)r_\sigma\right) - \nu(\nu-1)-(\nu-\#J_\infty)
        \end{equation*}
        where $\kappa_\sigma$ and $r_\sigma$ are respectively an upper bound on the rank and a lower bound on the generalized orders of a complete system of solutions of $Y' = \mathscr{A}_n Y$. All of these notions are defined in \cite[p.~154]{fischler2018}. In the context of the present paper, we can set
        \begin{enumerate}
            \item[$\bullet$] $\kappa_\sigma = 0$ for all $\sigma\in\Sigma$. Indeed, all of the singular points of the differential system $Y' = \mathscr{A}_n Y$ are regular, which implies that any solution of this system is in the Nilsson class at every point.
            \item[$\bullet$] $r_\sigma = 0$ for all $\sigma\in\Sigma$. Indeed, the matrix $\mathscr{A}_n$ is strictly lower triangular, so that all exponents of the differential system $Y' = \mathscr{A}_n Y$ are zero (see \cite[§6.2.4]{coddingtoncarlson97}, where what we call exponents here are the eigenvalues of the matrix $\mathbf{S}$).
            \item[$\bullet$] $\nu \leq d = M(h+1)+b_n$ by considering the differential operator $L$ of order $d$ constructed in the proof of \sref{Corollary}{corollairedeShidlovskii}.
        \end{enumerate}
        We also have $\#\Sigma = M+3$ and $\#J_\infty = b_n$, so that
        \begin{small}
        \begin{align*}
            c_{1} \leq (M+1)\frac{\big(M(h+1)+b_n\big)\big(M(h+1)+b_n-1\big)}{2}-\big(M(h+1)\big) \leq (M+1)\frac{\big(M(h+1)+a\big)^2}{2}.
        \end{align*}
        \end{small}
\end{remark*}

\section{Conclusion}
\label{partie7}

In \sref{Subsection}{soussection7.1}, we state and prove a linear independence criterion that generalizes \cite[Proposition 1, p.~6]{fischler2026}, the latter corresponding to the case $\psi(\bm{x}) = x_0$.

In \sref{Subsection}{soussection7.1bis}, we apply this criterion to the linear combinations $\Lambda_{n, (p, k)}$ defined in \sref{Subsection}{paragraphenotationpourlasuite} to obtain \autoref{théorèmeplusgénéral}.

We deduce \autoref{théorèmeprincipal} from this theorem. The case of odd conductors is addressed in \sref{Subsection}{soussection7.2}, and the case of conductors divisible by $4$ is addressed in \sref{Subsection}{soussection7.3}.

\subsection{A linear independence criterion}
\label{soussection7.1}

\begin{proposition}
\label{critèred'indépendancelinéaire}
    (Linear independence criterion)

    Let $\mathbb{K}$ be a number field. We denote by $\mathbb{K}_\infty$ the field $\R$ if $\mathbb{K}\subset \R$ and the field $\C$ otherwise.
    
    Let $s\geq 0$ be an integer. Let $\bm{\zeta} = \prescript{t}{}{(}\zeta_0, ..., \zeta_s) \in \mathbb{K}_\infty^{s+1}$ and $\psi: \mathbb{K}_\infty^{s+1} \to \mathbb{K_\infty}$ be a linear form with coefficients in $\mathbb{K}$ such that $\psi(\bm{\zeta}) \neq 0$.

    Let $\tau > 0$ be a real number and $(Q_n)_n$ be a sequence of real numbers that goes to $+\infty$. Let $\mathcal{N}\subseteq \Z_{\geq 0}$ be an infinite set. Suppose we have, for every $n\in\mathcal{N}$, an integer $K_n \geq 1$ and coefficients $\lambda_{n, i, k} \in \mathcal{O}_{\mathbb{K}}$, $\;0\leq i \leq s, \;0\leq k\leq K_n$ such that: 
    \begin{enumerate}
        \item[$(i)$] the linear combinations are bounded above:
            \begin{equation*}
                \text{Uniformly in } k\in\llbracket 0, K_n\rrbracket, \;\;\;\left| \sum_{i=0}^s \lambda_{n, i, k}\zeta_i\right| \stackunder{$\leq$}{$\scriptscriptstyle{n\to+\infty}$} Q_n^{-\tau+o(1)},
            \end{equation*}
        \item[$(ii)$] the coefficients are bounded above:
            \begin{equation*}
                \text{Uniformly in } k\in\llbracket 0, K_n\rrbracket, \;\;\;\forall i\in\llbracket 0, s\rrbracket,\;\;\; \house{\lambda_{n, i, k}} \stackunder{$\leq$}{$\scriptscriptstyle{n\to+\infty}$} Q_n^{1+o(1)},
            \end{equation*}
        \item[$(iii)$] for every sufficiently large $n\in\mathcal{N}$, we have for all $\bm{x} = \prescript{t}{}{(}x_0, ..., x_s)\in\mathbb{K}^{s+1}$ : 
        \begin{equation*}
            \bigg( \forall k \in \llbracket 0, K_n \rrbracket \;\;\; \sum_{i=0}^s \lambda_{n, i, k}x_i = 0\bigg) \;\Rightarrow \;\;\psi(\bm{x}) = 0.
        \end{equation*}
    \end{enumerate}
    Then
    \begin{equation*}
        \mathrm{dim}_{\mathbb{K}}\mathrm{Span}_{\mathbb{K}}(\zeta_0, ..., \zeta_s) \geq \frac{[\mathbb{K}_\infty: \R]}{[\mathbb{K}:\mathbb{Q}]}(\tau+1).
    \end{equation*}
\end{proposition}

\begin{proof}
    Proposition 1 in \cite{fischler2026}, that we generalize here, corresponds to this exact statement when $\psi: \bm{x} \mapsto x_0$ is the projection on the first coordinate. To prove our criterion, we consider any linear form $\psi$ with coefficients in $\K$ and such that $\psi(\bm{\zeta}) \neq 0$, and we construct from the coefficients $\lambda_{n, i, k}$ some new coefficients $\widetilde{\lambda}_{n, i, k} \in \mathcal{O}_{\mathbb{K}}$ satisfying the assumptions of \cite[Proposition 1, p.~6]{fischler2026}.

    Let $u_0, ..., u_s\in\mathbb{K}$ be such that $\psi: (x_0, ..., x_s) \mapsto \sum_{i=0}^s u_ix_i$. The linear form $\psi$ is non-zero and its multiples satisfy the assumptions of the proposition too. Swapping the coordinates and multiplying by a common denominator of the $u_i$ if necessary, we can assume $u_0 \neq 0$ and $u_0, ..., u_s \in \mathcal{O}_\mathbb{K}$.

    We then consider $\tilde{\bm{\zeta}}: = \raisebox{0.3cm}{${}^{t}$}\left(\frac{-1}{u_0}\psi(\bm{\zeta}), \zeta_1, ..., \zeta_s\right)$, and for $n\in\mathcal{N}, k\in\llbracket 0, K_n\rrbracket$ the coefficients
    \begin{equation*}
        \tilde{\lambda}_{n, 0, k}: = -u_0\lambda_{n, 0, k},\;\;\;\;\;\;\;\;\tilde{\lambda}_{n, i, k}: = u_0\lambda_{n, i, k} - u_i \lambda_{n, 0, k}, \;\;\;1\leq i\leq s.
    \end{equation*}
    These coefficients are in $\mathcal{O}_\mathbb{K}$ and satisfy: 
    \begin{enumerate}
        \item[$(i)$] $\forall k\in\llbracket 0, K_n\rrbracket, \;\;\; \Big| \sum_{i=0}^s \tilde{\lambda}_{n, i, k}\tilde{\zeta}_i\Big| = \Big| u_0\sum_{i=0}^s \lambda_{n, i, k}\zeta_i\Big| \leq Q_n^{-\tau+o(1)}$,
        \item[$(ii)$] $\forall i\in\llbracket 0, s\rrbracket,\;\;\; \forall k\in\llbracket 0, K_n\rrbracket, \;\;\; \house{\tilde{\lambda}_{n, i, k}} \leq 2\Big(\displaystyle{\max_{0\leq i\leq s}\house{u_i}}\Big) \max\left(\house{\lambda_{n, 0, k}}, \house{\lambda_{n, i, k}}\right) \leq Q_n^{1+o(1)}$,
        \item[$(iii)$] For $\bm{x} = \prescript{t}{}{(}x_0, ..., x_s) \in \mathbb{K}^{s+1}$, we have 
        \begin{equation*}
            \frac{1}{u_0}\sum_{i=0}^s \tilde{\lambda}_{n, i, k}x_i = \frac{-1}{u_0}\lambda_{n, 0, k}\psi(\bm{x}) + \sum_{i=1}^s \lambda_{n, i, k}x_i.
        \end{equation*}
        Thus, if $\sum_{i=0}^s \tilde{\lambda}_{n, i, k}x_i = 0$ for all $k\in\llbracket 0, K_n\rrbracket$, the assumption on the coefficients $\lambda_{n, i, k}$ implies: 
        \begin{equation*}
            \psi\left(\frac{-1}{u_0}\psi(\bm{x}), x_1, ..., x_s\right) = 0,
        \end{equation*}
        that is $u_0x_0 = 0$, whence $x_0 = 0$.
    \end{enumerate}
    We can thus apply \cite[Proposition 1, p.~6]{fischler2026} to $\tilde{\bm{\zeta}}$ and to the coefficients $\tilde{\lambda}_{n, i, k}$ to obtain: 
    \begin{equation*}
        \mathrm{dim}_\mathbb{K} \mathrm{Span}_\mathbb{K} \left(\frac{-1}{u_0}\psi(\bm{\zeta}), \zeta_1, ..., \zeta_s\right) \geq \frac{[\mathbb{K}_\infty: \R]}{[\mathbb{K}:\mathbb{Q}]}(\tau+1).
    \end{equation*}
    This concludes the proof, since $\mathrm{Span}_\mathbb{K} (\zeta_0, \zeta_1, ..., \zeta_s) = \mathrm{Span}_\mathbb{K} \left(\frac{-1}{u_0}\psi(\bm{\zeta}), \zeta_1, ..., \zeta_s\right).$
\end{proof}


\subsection{A slightly more general theorem}
\label{soussection7.1bis}

In this subsection, we prove the following result. We will deduce \autoref{théorèmeprincipal} from it in the next subsection.

\begin{theorem}
\label{théorèmeplusgénéral}
    Let $\K$ be a number field. Let $M\geq 1$ be an integer and $z_0$ be an $M$-th root of $-1$. Let $f: \Z \to \C$ be an $M$-periodic function. If $\K\subset\R$, assume that the numbers $z_0^mf(m)$ are real, $0\leq m\leq M-1$. 
    
    Let $\varepsilon \in \{ 0, 1 \}$ be such that
    \begin{equation}
    \label{hypothèse}
        -z_0^mf(m) - (-1)^{\varepsilon}z_0^{M-m}f(M-m) \neq 0
    \end{equation}
    for some $m\in\llbracket 0, M\rrbracket$.
     
    For any sufficiently large integer $s$, we have

    \begin{equation*}
        \mathrm{dim}_\K \mathrm{Span}_\K \Big\{ \;L(f, i, z_0) \;|\; 2\leq i\leq s, \;i\equiv \varepsilon \mod 2 \; \Big\} \geq \frac{0.21[\K_\infty:\R]}{\sqrt{M}[\K:\Q]}\sqrt{\frac{s}{\log(s)}},
    \end{equation*}
    where $\K_\infty = \R$ if $\K\subset \R$ and $\K_\infty = \C$ otherwise.
\end{theorem}

Let $\K, M, z_0, f$ and $\varepsilon$ be as in the statement of this theorem. Let $m_0\in \llbracket 0, M\rrbracket$ be such that $-z_0^{m_0}f(m_0) - (-1)^{\varepsilon}z_0^{M-m_0}f(M-m_0) \neq 0$.

For now, we only consider integers $a\in\Z_{>0}$ that are multiples of $100$ and sufficiently large for $3.9\sqrt{a\log(a)/M} < \frac{a}{M}$ to hold. We set $r=3.9$, $\kappa = 10.58$, $\omega = 12$, $\Omega = \lfloor 3.9\sqrt{a\log(a)/M}\rfloor$ and $h=0.36a$.

We consider the infinite set $\mathcal{N} \subset \Z_{\geq 0}$ defined at the beginning of \sref{Section}{partie4}. For integers $n\in\mathcal{N}$ sufficiently large (in a sense given by equations \eqref{nassezgrand1} and \eqref{nassezgrand2}), we consider the integers $c_{n, i, j}$ obtained from \sref{Proposition}{existencedescij}, and the quantities $\Lambda_{n, (p, k)}$, $(p, k)\in\llbracket 0, h\rrbracket \times \llbracket 2rn+2, \kappa n\rrbracket$ constructed in \sref{Section}{partie4}.

Recall that according to \eqref{Lambdacommecldeszeta}, the quantities $\Lambda_{n, (p, k)}$ are linear combinations of the coordinates of $\bm{\zeta} = \prescript{t}{}{\big(}\zeta_0^{<0>}, ..., \zeta_0^{<M-1>}, \zeta_1, ..., \zeta_{a+h}\big)$. We will apply \sref{Proposition}{critèred'indépendancelinéaire} to them to conclude.

\smallskip

According to \sref{Proposition}{majorationdescoefficients}, these combinations have coefficients in $\Z$. In particular, their coefficients are in $\mathcal{O}_{\K}$ and their house is equal to their absolute value.

Moreover, according to Propositions\sref{\hspace{1mm}}{majorationdescoefficients} and\sref{\hspace{1mm}}{majorationdescombinaisons}, assumptions $(i)$ and $(ii)$ of the linear independence criterion are satisfied, with $Q_n =\beta^n$ and $\tau = -\frac{\log(\alpha)}{\log(\beta)}$, so that $Q_n^{-\tau} = \alpha^n$, with\footnote{The asymptotics of $-\log(\alpha)$ computed below show that $\alpha<1$ for sufficiently large $a$. In particular, $\alpha \neq 1$ as required in \sref{Proposition}{majorationdescombinaisons}, and $\tau > 0$.}
\begin{small}
\begin{equation*}
    \alpha = \frac{1}{r^\Omega} \Big(e^4 (2a+1)^{1.04}\Big)^\kappa \xi, \hspace{2cm}\beta = \Big(32e^3(2a+1)^{1.04}\Big)^\kappa \xi,
\end{equation*}
\end{small}
and
\begin{small}
\begin{equation*}
    \xi = \exp\Bigg(\frac{\omega\log(2)+2\omega^2 + 1.04\omega^2\log(a+1)+\frac{1}{2}\Omega^2\log(r)}{\frac{a}{M}-\omega}\Bigg)\end{equation*}
\end{small}

Lastly, since $(h+1)(\kappa-2r)M+\omega > h(\kappa-2r) = 1.0008a > a$, \sref{Proposition}{propositiondéduitedeshidlovskii} ensures that assumption $(iii)$ of the linear independence criterion is satisfied by the linear forms
\begin{equation*}
    \psi_{m} : \bm{x} = \prescript{t}{}{(}x_0^{<0>}, ..., x_0^{<M-1>}, x_1, ..., x_{a+h}) \mapsto x_0^{<m>} + (-1)^\varepsilon x_0^{<M-m>}, \hspace{0.5cm}0\leq m\leq M.
\end{equation*}

Recall that $\zeta_0^{<m>} = -z_0^mf(m)$, so that \eqref{hypothèse} yields $\psi_{m_0}(\bm{\zeta})\neq 0$. We can thus apply \sref{Proposition}{critèred'indépendancelinéaire} with the linear form $\psi_{m_0}$. We obtain
\begin{align*}
    \mathrm{dim}_{\K}\mathrm{Span}_{\K} \Big\{-z_0^mf(m) \;|\;0\leq m\leq M-1\Big\} \cup \Big\{L(f, i, z_0) \;|\;&1\leq i\leq a+h, \;i\equiv\varepsilon\mod 2\Big\} \\ 
    &\geq \hspace{5mm} \frac{[\K_\infty:\R]}{[\K:\Q]}\left( 1 - \frac{\log(\alpha)}{\log(\beta)}\right).
\end{align*}
Let us compute the asymptotics of this lower bound as $a \xrightarrow[]{} +\infty$: 

\begin{equation*}
    \log(\xi) \stackunder{$=$}{$\scriptscriptstyle{a\to+\infty}$} \frac{\frac{3.9^2}{2M}a\log(a)\log(3.9)}{\frac{a}{M}}\big(1+o(1)\big) \stackunder{$\leq$}{$\scriptscriptstyle{a\to+\infty}$} 10.36\log(a)\big(1+o(1)\big)
\end{equation*}
so that
\begin{equation*}
    \log(\beta) \stackunder{$=$}{$\scriptscriptstyle{a\to+\infty}$} \big(1.04\kappa\log(a) + \log(\xi)\big)\big(1+o(1)\big) \stackunder{$\leq$}{$\scriptscriptstyle{a\to+\infty}$} 21.37\log(a)\big(1+o(1)\big), 
\end{equation*}
and
\begin{align*}
    -\log(\alpha) &\stackunder{$=$}{$\scriptscriptstyle{a\to+\infty}$} \left(\frac{3.9}{\sqrt{M}}\sqrt{a\log(a)}+O(1)\right)\log(r) - 1.04\kappa \log(a)\big(1+o(1)\big) - \log(\xi)\\
    &\stackunder{$=$}{$\scriptscriptstyle{a\to+\infty}$} \frac{3.9}{\sqrt{M}}\sqrt{a\log(a)}\log(3.9)\big(1+o(1)\big) \\
    &\stackunder{$\geq$}{$\scriptscriptstyle{a\to+\infty}$} \frac{5.3}{\sqrt{M}}\sqrt{a\log(a)}\big(1+o(1)\big).
\end{align*}

Hence,
\begin{equation*}
    1 - \frac{\log(\alpha)}{\log(\beta)} \stackunder{$\geq$}{$\scriptscriptstyle{a\to+\infty}$} \frac{0.248\big(1+o(1)\big)}{\sqrt{M}}\sqrt{\frac{a}{\log(a)}}.
\end{equation*}

\smallskip
The numbers $z_0^mf(m), 0\leq m\leq M-1,\;$ and $L(f, 1, z_0)$ are in finite number: they span a $\K$-vector space of dimension $O(1) = o\left(\sqrt{\frac{a}{\log(a)}}\right)$ as $a \to +\infty$ and can thus be ignored.

Let $s := a+h$. By noticing that $$\sqrt{\frac{a}{\log(a)}} \stackunder{$\geq$}{$\scriptscriptstyle{a\to+\infty}$} \frac{1}{\sqrt{1.36}}\big(1+o(1)\big)\sqrt{\frac{s}{\log(s)}},$$ we can write
\begin{equation*}
    \mathrm{dim}_{\K}\mathrm{Span}_{\K} \Big\{L(f, i, z_0) \;|\;2\leq i\leq s, \;i\equiv\varepsilon\mod 2\Big\} \stackunder{$\geq$}{$\scriptscriptstyle{a\to+\infty}$} \frac{0.212[\K_\infty:\R]\big(1+o(1)\big)}{[\K:\Q]\sqrt{M}}\sqrt{\frac{s}{\log(s)}}.
\end{equation*}

This proves the inequality of \autoref{théorèmeplusgénéral} for sufficiently large integers $s$ that can be written as $1.36a$, with $100 | a$.

Finally, for arbitrary sufficiently large integers $s$, we pick the largest multiple $a_s$ of $100$ satisfying $1.36a_s \leq s$. We have 
\begin{small}
\begin{equation*}
    \mathrm{Span}_{\K} \Big\{L(f, i, z_0) \;|\;2\leq ia\leq s, \;i\equiv\varepsilon \hspace{-3mm}\mod 2 \Big\} \supseteq \mathrm{Span}_{\K} \Big\{L(f, i, z_0) \:|\;2\leq i\leq 1.36a_s, \:i\equiv\varepsilon \hspace{-3mm}\mod 2 \Big\},
\end{equation*}
\end{small}
and since $\frac{1.36a_s}{s} \xrightarrow[s\to+\infty]{} 1$ and $0.212 > 0.21$, \autoref{théorèmeplusgénéral} holds for all sufficiently large integers $s$.

\subsection{The case of odd conductors}
\label{soussection7.2}

Let us prove the case of odd conductors of \autoref{théorèmeprincipal}. As explained at the beginning of \sref{Section}{partie2}, it is enough to prove it for primitive characters. Let $N\geq 1$ be an odd integer and $\chi$ be a primitive Dirichlet character mod $N$.

Notice that $\K \subset \R$ if and only if $N = 1$ or $3$. In those cases, $\varphi(N) \leq 2$ and $\chi$ is real valued, so that the numbers $(-1)^m\chi(m)$ are real.

Moreover, since $N$ is odd, we can check that \eqref{hypothèse} is satisfied:
\begin{equation*}
    -(-1)\chi(1) - (-1)^\varepsilon (-1)^{N-1}\chi(N-1) = \chi(1) - (-1)^{\varepsilon}\chi(-1) = 2 \neq 0.
\end{equation*}

We can thus apply \autoref{théorèmeplusgénéral} with $\K = \Q(e^{2i\pi/\varphi(N)})$, where $\varphi$ denotes Euler's totient function, $M=N$, $z_0 = -1$, $f=\chi$ and $\varepsilon$ the parity opposite to the parity of $\chi$. We obtain for sufficiently large integers $s$
\begin{equation*}
        \mathrm{dim}_\K \mathrm{Span}_\K \Big\{ \;L(\chi, i, -1) \;|\; 2\leq i\leq s, \;i\equiv \varepsilon \mod 2 \; \Big\} \geq \frac{0.21[\K_\infty:\R]}{\sqrt{N}[\K:\Q]}\sqrt{\frac{s}{\log(s)}},
\end{equation*}

\smallskip
Now, notice that for any $i\geq 2$
\begin{align*}
\label{symétrieen1et-1}
    L(\chi, i, -1) &= \sum_{m=1}^{+\infty} \frac{(-1)^m\chi(m)}{m^i} = 2\sum_{m=1}^{+\infty} \frac{\chi(2m)}{(2m)^i} -\sum_{m=1}^{+\infty} \frac{\chi(m)}{m^i} = \Big(2^{1-i}\chi(2)-1\Big)L(\chi, i). 
\end{align*}
                    
Since $\chi(2)$ is a $\varphi(N)$-th root of unity ($\chi(2) \neq 0$ because $N$ is odd) and $|2^{1-i}\chi(2)|<1$ for $i\geq 2$, we have $\Big(2^{1-i}\chi(2)-1\Big) \in \K^*$, so that
\begin{align*}
    \mathrm{Span}_{\K} \Big\{L(\chi, i) \;|\;2\leq i\leq s, \;i\equiv\varepsilon \mod 2\Big\} = \mathrm{Span}_{\K} \Big\{L(\chi, i, -1) \;|\; 2\leq i\leq s, \;i\equiv\varepsilon \mod 2\Big\}.
\end{align*}

Notice that if $N \geq 5$, then $\varphi(N) \geq 4$ and $\K_\infty = \C$, while if $N=1$ or $3$, $\K_\infty = \R$. We thus have $0.21[\K_\infty:\R] = \mathcal{C}_N$ for any $N\geq 1$ odd. This proves \autoref{théorèmeprincipal} when $N$ is odd.

\subsection{The case of even conductors}
\label{soussection7.3}

Let us prove the case of conductors divisible by $4$ of \autoref{théorèmeprincipal}. As explained at the beginning of \sref{Section}{partie2}, it is enough to prove it for primitive characters. We recall the following property noted in \cite[(4.1)]{fischler2018}.
\begin{lemma}
\label{lemmechimod4}
    For any $N\geq 4$ multiple of $4$ and any primitive Dirichlet character $\chi$ modulo $N$, we have 
    \begin{equation*}
        \forall m\in\Z, \hspace{0.5cm} \chi(m + N/2) = -\chi(m).    
    \end{equation*}    
\end{lemma}

Let $N\geq 4$ be a multiple of $4$ and $\chi$ be a primitive Dirichlet character mod $N$. 

We will apply \autoref{théorèmeplusgénéral} with $\K = \Q(e^{2i\pi /\varphi(N)})$, where $\varphi$ denotes Euler's totient function, $M = N/2$, $z_0 = e^{i\pi/M}$, $f : m \mapsto \chi(m)z_0^{-m}$ and $\varepsilon$ the parity opposite to the parity of $\chi$. 

Notice that $\K \subset \R$ if and only if $N = 4$. In this case, $\varphi(N) = 2$ and $\chi$ is real valued, so that the numbers $z_0^m f(m) = \chi(m)$ are real.

Then, using \sref{Lemma}{lemmechimod4}, we check that for $m\in\Z$
\begin{equation*}
    f(m+M) = \chi(m+M)z_0^{-M-m} = -\chi(m)\cdot(-1)z_0^{-m} = f(m),
\end{equation*}
so that $f$ is $M$-periodic. Using this lemma again, we can also check that \eqref{hypothèse} is satisfied:
\begin{equation*}
    -z_0\chi(1)z_0^{-1} - (-1)^\varepsilon z_0^{M-1}\chi(M-1)z_0^{-(M-1)} = -\chi(1) - (-1)^{\varepsilon}\big(-\chi(-1)\big) = -2 \neq 0.
\end{equation*}

\autoref{théorèmeplusgénéral} thus applies and yields for sufficiently large integers $s$
\begin{equation*}
        \mathrm{dim}_\K \mathrm{Span}_\K \Big\{ \;L(f, i, z_0) \;|\; 2\leq i\leq s, \;i\equiv \varepsilon \mod 2 \; \Big\} \geq \frac{0.21[\K_\infty:\R]}{\sqrt{N/2}[\K:\Q]}\sqrt{\frac{s}{\log(s)}},
\end{equation*}

Now notice that for any $i\geq 2$
\begin{equation*}
    L(f, i, z_0) = \sum_{m=1}^{+\infty} \frac{\chi(m)z_0^{-m}z_0^m}{m^i} = \sum_{m=1}^{+\infty} \frac{\chi(m)}{m^i} = L(\chi, i).
\end{equation*}

We have thus obtained
\begin{equation*}
        \mathrm{dim}_\K \mathrm{Span}_\K \Big\{ \;L(\chi, i) \;|\; 2\leq i\leq s, \;i\equiv \varepsilon \mod 2 \; \Big\} \geq \frac{0.21\sqrt{2}[\K_\infty:\R]}{\sqrt{N}[\K:\Q]}\sqrt{\frac{s}{\log(s)}}.
\end{equation*}

Notice that if $N \geq 8$, then $\varphi(N) \geq 4$ and $\K_\infty = \C$, while if $N=4$, $\K_\infty = \R$. We thus have $0.21\sqrt{2}[\K_\infty:\R] = \mathcal{C}_N$ for any $N \equiv 0 \mod 4$. This proves \autoref{théorèmeprincipal} when $N$ is divisible by $4$.

\bibliographystyle{alpha}
\bibliography{bibliography}


\end{document}